\documentclass[a4paper,twoside,11pt,reqno]{amsart}
\usepackage{amssymb,amsfonts,amsmath,stmaryrd,bbm}
\usepackage{latexsym}
\usepackage{gensymb}
\usepackage{verbatim}
\usepackage{colortbl}
\usepackage{arydshln}  
\usepackage[utf8]{inputenc}
\usepackage[T1]{fontenc}

\usepackage{kpfonts}
\usepackage{afterpage}
\usepackage{dsfont}
\usepackage{color}
\usepackage{a4wide} 
\usepackage{url}
\usepackage{pifont, tikz, subfigure}
\usetikzlibrary{plotmarks,shapes,arrows,positioning}

\usepackage[plainpages=false,pdfpagelabels,colorlinks=true,citecolor=blue,hypertexnames=false]{hyperref}

\newcommand{\LandauO}{\mathcal{O}}

\newcommand{\tA}{\tilde A}
\newcommand{\tK}{\mathscr K}
\newcommand{\tS}{\mathscr S}
\newcommand{\tQ}{\mathscr Q}
\newcommand{\tV}{\mathscr V}
\newcommand{\tH}{\mathscr H}
\newcommand{\btH}{\overline{\mathscr H}}
\newcommand{\Hc}{\mathcal H}

\newcommand{\tX}{\mathscr X}
\newcommand{\tY}{\mathscr Y}
\newcommand{\Qc}{\mathcal{Q}}

\newcommand{\Cc}{\mathcal{C}}

\def\qee{$\hfill{\Box}$}

\catcode`\@=11
\def\section{\@startsection{section}{1}%
 \z@{.7\linespacing\@plus\linespacing}{.5\linespacing}%
 {\normalfont\Large\bfseries\scshape\centering}}

\def\subsection{\@startsection{subsection}{2}%
  \z@{.5\linespacing\@plus\linespacing}{.5\linespacing}%
  {\normalfont\large\bfseries\scshape}}

\def\subsubsection{\@startsection{subsubsection}{3}%
 \z@{.5\linespacing\@plus\linespacing}{-.5em}
 {\normalfont\large\bfseries}}
\catcode`\@=12

\newtheorem{theorem}{Theorem}[section]
\newtheorem{lemma}[theorem]{Lemma}

\newtheorem{cor}[theorem]{Corollary}
\newtheorem{prop}[theorem]{Proposition}

\newtheorem{definition}[theorem]{Definition}
\theoremstyle{definition}

\newtheorem{remark}[theorem]{Remark}

\newcommand{\beq}{\begin{equation}}
\newcommand{\eeq}{\end{equation}}

\renewcommand{\epsilon}{\varepsilon}

\newcommand{\ns}{\mathbb{N}}

\newcommand{\zs}{\mathbb{Z}}

\newcommand{\qs}{\mathbb{Q}}

\newcommand{\bx}{\bar x}
\newcommand{\by}{\bar y}

\newcommand{\bh}{\bar h}

\newcommand{\cC}{\mathcal C}
\newcommand{\cS}{\mathcal S}

\newcommand{\PIK}{V}
\newcommand{\PIIK}{W}
\newcommand{\PIIIK}{Z}

\newcommand{\PID}{M}
\newcommand{\PIID}{N}
\newcommand{\UI}{P_1}
\newcommand{\UII}{P_2}

\DeclareMathOperator{\Rat}{Rat}

\newcommand{\fps}{formal power series}
\newcommand{\gf}{generating function}
\newcommand{\gfs}{generating functions}

\def\emm#1,{{\em #1}}

\newcommand{\C}{C}

\newcommand{\Maple}{{\sc Maple}}
\begin{document}
%
\title
[Three-quadrant  walks  via invariants]
{Enumeration of three-quadrant  walks  via invariants:\\ [1.5ex]
some diagonally symmetric models}

\author[M. Bousquet-M\'elou]{Mireille Bousquet-M\'elou}

\thanks{MBM was partially supported by the ANR projects DeRerumNatura (ANR-19-CE40-0018) and Combiné (ANR-19-CE48-0011).}
 
\address{CNRS, LaBRI, Universit\'e de Bordeaux, 351 cours de la
  Lib\'eration,  F-33405 Talence Cedex, France} 
\email{bousquet@labri.fr}

\begin{abstract} In the past $20$ years, the enumeration of plane lattice walks confined to a convex cone --- normalized into the first quadrant --- has received a lot of attention, stimulated the development of several original approaches, and led to a rich collection of results. Most of them deal with the nature of the associated \gf: for which \emm models, is it algebraic, D-finite, D-algebraic? By \emm  model,, what  we mean is a finite collection of allowed steps.
 
  More recently, similar questions have been raised for non-convex cones, typically the three-quadrant cone $\Cc = \{ (i,j) : i \geq 0 \text{ or } j \geq 0 \}$.  They turn out to be more difficult than their quadrant counterparts. In this paper, we investigate a collection of eight models in $\Cc$, which
  can be seen as the first level of difficulty beyond quadrant problems. This collection consists of diagonally symmetric models in $\{-1, 0,1\}^2\setminus\{(-1,1), (1,-1)\}$.
  Three of them are known not to be D-algebraic.  We show that the remaining five can be solved in a uniform fashion using Tutte's notion of \emm invariants,,
which has   already proved useful for some quadrant models.  Three
models are found to be algebraic, one  is (only) D-finite, and the
last one is (only) D-algebraic. We also solve in the same fashion the
diagonal model $\{ \nearrow, \nwarrow, \swarrow, \searrow\}$, which is
D-finite. The three algebraic models are those of the Kreweras trilogy, $\cS=\{\nearrow, \leftarrow,
\downarrow\}$, $\cS'=\{\rightarrow, \uparrow, \swarrow\}$, and
$\cS\cup \cS'$. 

Our solutions take similar forms  for all six models. Roughly speaking, the square of the \gf\ of three-quadrant walks with steps in $\cS$ is an explicit rational function in the \emm quadrant, \gf\ with steps in $\tS:= \{(j-i,j): (i,j) \in \cS\}$. We derive various exact or asymptotic
corollaries, including an explicit algebraic description of the positive harmonic function
in $\cC$ for the five models that are at least D-finite.
\end{abstract}


\maketitle

\section{Introduction}
\label{sec:intro}

\subsection{Walks in a quadrant}
Over the last two decades, the enumeration of walks in the non-negative quadrant
\[
  \Qc:= \{ (i,j) : i \geq 0 \text{ and } j \geq 0 \}
\]
has attracted a lot of attention  and  established its own scientific community with {close to a hundred} research papers; see, e.g., \cite{BMM-10} and citing papers.
 Most of the attention has focused on  walks \emm with small steps,,
 that is, taking their steps in a fixed subset $\cS$ of $\{-1,0,
 1\}^2\setminus\{(0,0)\}$. We often call \emm $\cS$-walk, a walk
 taking its steps in $\cS$. For each  step set $\cS$  {(often called a \emm model, henceforth)}, one considers a  three-variate \gf\ $Q(x,y;t)$ defined by
\beq\label{Q-def}
  Q(x,y;t)= \sum_{n \ge 0}\sum_{i,j \in \Qc} q_{i,j}(n) x^i y^j t^n,
\eeq
where $q_{i,j}(n)$ is the number of quadrant $\cS$-walks starting from $(0,0)$, ending at $(i,j)$, and having in total
$n$ steps. For each small step model $\cS$, one now knows whether and where the series $Q(x,y;t)$ fits in the following classical hierarchy of series:
\beq\label{hierarchy}
  \hbox{ rational } \subset \hbox{ algebraic } \subset  \hbox{ D-finite }
        \subset \hbox{ D-algebraic}.
\eeq
Recall  that a series (say $Q(x,y;t)$ in our case) is \emm rational, if it is the ratio of two polynomials, \emm algebraic, if it satisfies a polynomial equation (with coefficients that are polynomials in the variables), \emm D-finite, if it satisfies three \emm linear, differential equations (one in each variable), again with polynomial coefficients, and finally \emm D-algebraic,
if it satisfies three \emm polynomial, differential equations. One central result in the classification of quadrant problems is that, for the $79$  (non-trivial, and essentially distinct) models with small steps, the series  $Q(x,y;t)$ is D-finite if and only if a certain group, which is easy to construct from the step set $\cS$, is finite~\cite{BMM-10,BoKa-10,BoRaSa-14,KR-12,MeMi13,MiRe09}.

\begin{figure}[t!]
	\centering
\includegraphics[height=40mm]{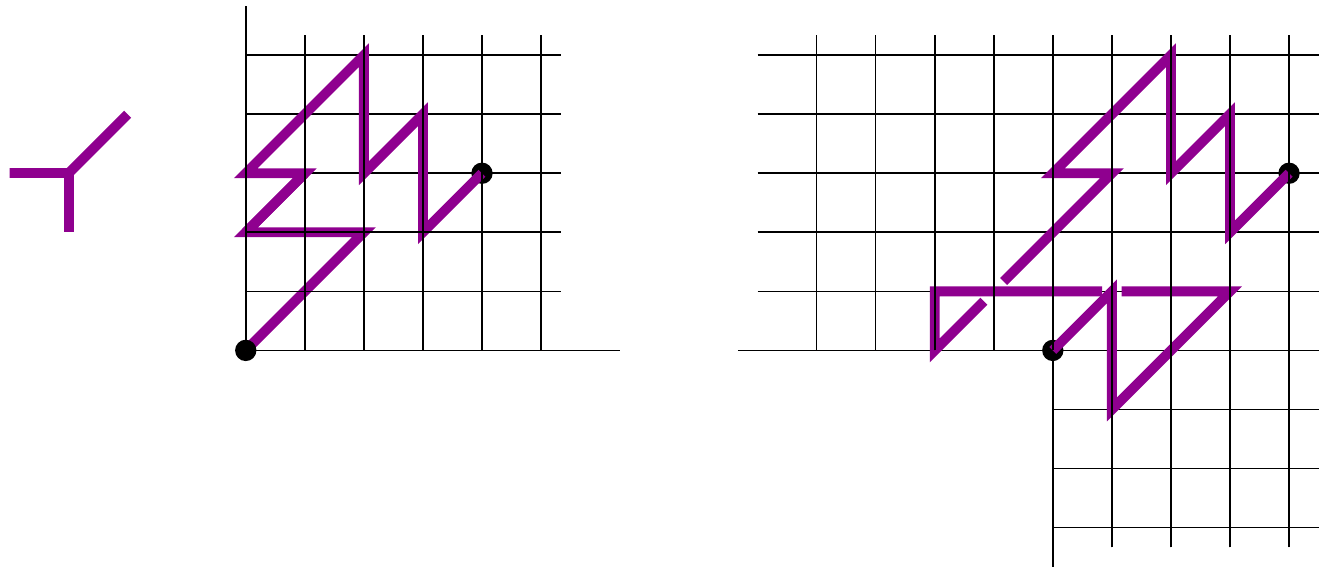}
	\caption{Two walks with Kreweras steps $\nearrow, \leftarrow, \downarrow$, one in the first quadrant~$\Qc$ (left), and one in three-quadrant cone $\Cc$ (right). The associated generating functions are algebraic.}
	\label{fig:kreweras}
      \end{figure}

\subsection{Walks in a three-quadrant cone} In 2016, the  author turned her attention to non-convex cones~\cite{BM-three-quadrants}, and initiated  the enumeration of  lattice paths confined to  the three-quadrant cone
\[ 
	\Cc := \{ (i,j) : i \geq 0 \text{ or } j \geq 0 \}
\] 
(see Figure~\ref{fig:kreweras}). We can say that such walks \emm
avoid the negative quadrant,. These problems turn out to be significantly harder than their quadrant counterparts. In~\cite{BM-three-quadrants}, the two most natural models were solved: \emm simple walks, with steps in $\{\rightarrow,  \uparrow,
\leftarrow, \downarrow \}$, and \emm diagonal walks, with steps in $\{\nearrow, \nwarrow, \swarrow,  \searrow \}$. The associated series $C(x,y;t)$, defined analogously to~\eqref{Q-def}:
\[ 
C(x,y;t)
= \sum_{n \geq 0} \sum_{(i,j) \in \Cc}  c_{i,j}(n) x^i y^j t^n,
\] 
 was proved to be D-finite for both models.
 It then became  natural to explore more three-quadrant problems, in particular to understand whether the D-finiteness of $C(x,y;t)$ was again related to the finiteness of the associated group --- or even  whether the series $C(x,y;t)$ and $Q(x,y;t)$ would always lie \emm in the same class, of the hierarchy~\eqref{hierarchy} (with the exception of five models that are non-trivial for the quadrant problem but become trivial, with a rational series, for the three-quadrant cone). This conjecture of Dreyfus and Trotignon~\cite{dreyfus-trotignon} holds so far for all solved cases, and will only be reinforced  by  this paper\footnote{In a paper in preparation~\cite{elvey-price-trois-quarts}, Elvey Price proves this conjecture for the dependence in $x$.}.
 
Using an asymptotic argument, Mustapha quickly proved that the $51$ three-quadrant problems associated with an infinite group have, as their quadrant counterparts, a non-D-finite solution~\cite{mustapha-3quadrant}. Regarding exact solutions,  Raschel and Trotignon obtained in~\cite{RaTr-19} sophisticated integral expressions for $C(x,y;t)$ for the step sets $\cS$ of Table~\ref{tab:sym} (apart from the diagonal model). The first four have a finite group, and the expressions of~\cite{RaTr-19} imply that $C(x,y;t)$ is indeed D-finite for these models (at least in $x$ and $y$).
The other four have an infinite group and $C(x,y;t)$ is non-D-finite
by~\cite{mustapha-3quadrant}. These four non-D-finite models, labelled from $\#6$ to $\#9$ in Table~\ref{tab:sym}, were further studied by Dreyfus and Trotignon~\cite{dreyfus-trotignon}: they proved that $C(x,y;t)$ is D-algebraic for the first one ($\#6$), but not for the other three. More recently,  the method used in the original paper~\cite{BM-three-quadrants} was  adapted to solve the so-called king model where all eight steps are allowed~\cite{mbm-wallner-conf,mbm-wallner}, again with a D-finite solution (and a finite group). Finally, remarkable results of Budd~\cite{Budd2020Winding} and Elvey Price~\cite{ElveyPrice2020Winding} on the winding number of various families of plane walks provide explicit D-finite expressions for several series counting walks in $\Cc$ with prescribed endpoints.

\definecolor{Gray}{gray}{0.9}
\newcolumntype{a}{ >{\columncolor{Gray}} c }

\renewcommand{\arraystretch}{1.5}

\begin{table}[tbh!]
  \makebox[\textwidth][c]{
  \begin{tabular}{|c||ccc:ca||c:ccc|}
         \hline
   \rule{0pt}{4ex}   &   Kreweras &
                  \begin{minipage}{15mm}
                    reverse\\ Kreweras 
                  \end{minipage}
&  \begin{minipage}{15mm}
                   double\\ Kreweras 
                 \end{minipage} & simple & diagonal & $\#6$ & $\#7$ &$\#8$ &$\#9$
      \\ \rule{0pt}{7ex}
  $\cS$&  \begin{tikzpicture}[scale=.45] 
    \draw[->] (0,0) -- (0,-1);
    \draw[->] (0,0) -- (1,1);
    \draw[->] (0,0) -- (-1,0);
  \end{tikzpicture}     &
 \begin{tikzpicture}[scale=.45] 
    \draw[->] (0,0) -- (0,1);
    \draw[->] (0,0) -- (-1,-1);
    \draw[->] (0,0) -- (1,0);
  \end{tikzpicture}&
    \begin{tikzpicture}[scale=.45] 
      \draw[->] (0,0) -- (0,-1);
    \draw[->] (0,0) -- (1,1);
    \draw[->] (0,0) -- (-1,0);
    \draw[->] (0,0) -- (0,1);
    \draw[->] (0,0) -- (-1,-1);
    \draw[->] (0,0) -- (1,0);
  \end{tikzpicture}
&
    \begin{tikzpicture}[scale=.45] 
      \draw[->] (0,0) -- (0,-1);
      \draw[->] (0,0) -- (-1,0);
    \draw[->] (0,0) -- (0,1);
      \draw[->] (0,0) -- (1,0);
  \end{tikzpicture}
&
    \begin{tikzpicture}[scale=.45] 
      \draw[->] (0,0) -- (1,-1);
      \draw[->] (0,0) -- (-1,1);
    \draw[->] (0,0) -- (1,1);
      \draw[->] (0,0) -- (-1,-1);
  \end{tikzpicture}
&
    \begin{tikzpicture}[scale=.45] 
      \draw[->] (0,0) -- (0,-1);
    \draw[->] (0,0) -- (1,1);
    \draw[->] (0,0) -- (-1,0);
    \draw[->] (0,0) -- (0,1);
       \draw[->] (0,0) -- (1,0);
     \end{tikzpicture}
&   \begin{tikzpicture}[scale=.45] %
      \draw[->] (0,0) -- (0,-1);
    \draw[->] (0,0) -- (1,1);
    \draw[->] (0,0) -- (-1,0);
      \draw[->] (0,0) -- (-1,-1);
  \end{tikzpicture}
&
    \begin{tikzpicture}[scale=.45] %
        \draw[->] (0,0) -- (1,1);
       \draw[->] (0,0) -- (0,1);
    \draw[->] (0,0) -- (-1,-1);
    \draw[->] (0,0) -- (1,0);
  \end{tikzpicture}
&
    \begin{tikzpicture}[scale=.45] %
      \draw[->] (0,0) -- (0,-1);
       \draw[->] (0,0) -- (-1,0);
    \draw[->] (0,0) -- (0,1);
    \draw[->] (0,0) -- (-1,-1);
    \draw[->] (0,0) -- (1,0);
  \end{tikzpicture}
\\
$\tS$ & \begin{tikzpicture}[scale=.45] 
    \draw[->] (0,0) -- (0,1);
    \draw[->] (0,0) -- (-1,-1);
    \draw[->] (0,0) -- (1,0);
  \end{tikzpicture}
 &  \begin{tikzpicture}[scale=.45] 
    \draw[->] (0,0) -- (0,-1);
    \draw[->] (0,0) -- (1,1);
    \draw[->] (0,0) -- (-1,0);
  \end{tikzpicture}
&\begin{tikzpicture}[scale=.45] 
      \draw[->] (0,0) -- (0,-1);
    \draw[->] (0,0) -- (1,1);
    \draw[->] (0,0) -- (-1,0);
    \draw[->] (0,0) -- (0,1);
    \draw[->] (0,0) -- (-1,-1);
    \draw[->] (0,0) -- (1,0);
  \end{tikzpicture}
&   
\begin{tikzpicture}[scale=.45] 
       \draw[->] (0,0) -- (1,1);
    \draw[->] (0,0) -- (-1,0);
       \draw[->] (0,0) -- (-1,-1);
    \draw[->] (0,0) -- (1,0);
  \end{tikzpicture}
& \begin{tikzpicture}[scale=.45] 
       \draw[->] (0,0) -- (1,1);
    \draw[->] (0,0) -- (0,-1);
       \draw[->] (0,0) -- (-1,-1);
    \draw[->] (0,0) -- (0,1);
  \end{tikzpicture}
                               &\begin{tikzpicture}[scale=.45] 
    \draw[->] (0,0) -- (1,1);
    \draw[->] (0,0) -- (-1,0);
    \draw[->] (0,0) -- (0,1);
    \draw[->] (0,0) -- (-1,-1);
    \draw[->] (0,0) -- (1,0);
  \end{tikzpicture}
  &   
\begin{tikzpicture}[scale=.45] 
      \draw[->] (0,0) -- (0,-1);
    \draw[->] (0,0) -- (0,1);
    \draw[->] (0,0) -- (-1,-1);
    \draw[->] (0,0) -- (1,0);
  \end{tikzpicture}
&   
\begin{tikzpicture}[scale=.45] 
      \draw[->] (0,0) -- (0,-1);
    \draw[->] (0,0) -- (1,1);
    \draw[->] (0,0) -- (-1,0);
    \draw[->] (0,0) -- (0,1);
  \end{tikzpicture}
&\begin{tikzpicture}[scale=.45] 
      \draw[->] (0,0) -- (0,-1);
    \draw[->] (0,0) -- (1,1);
    \draw[->] (0,0) -- (-1,0);
    \draw[->] (0,0) -- (-1,-1);
    \draw[->] (0,0) -- (1,0);
  \end{tikzpicture}
  \\
      &&&&Gessel &&&&&\\
      \hline
      $C_\cS(x,y;t)$ & alg. & alg. & alg. & DF & DF & DA&non-DA&non-DA& non-DA \\
      &Sec.~\ref{sec:K} & Sec.~\ref{sec:RK} &Sec.~\ref{sec:DK}&  Sec.~\ref{sec:simple}&  Sec.~\ref{sec:diag} &Sec.~\ref{sec:DA}&&& \\
            & &&& \cite{BM-three-quadrants,RaTr-19}  &\cite{BM-three-quadrants,RaTr-19}  &\cite{dreyfus-trotignon} &\cite{dreyfus-trotignon}& ~\cite{dreyfus-trotignon}& \cite{dreyfus-trotignon}
      \\
      \hline
    \end{tabular}}
    \vskip 4mm  \caption{The nine models  $\cS$  considered in this paper. One is the diagonal model (shaded column). The others are the eight models with $x/y$-symmetry and no step $\nwarrow$ nor $\searrow$.
    Each model is shown with its companion model $\tS$, with step
    polynomial $S(1/x, xy)$ (or $S(1/\sqrt x , \sqrt x y)$ for the
    diagonal model). The first five models have a finite group, the
    other four an infinite group.}
    \label{tab:sym}
  \end{table}

  \subsection{Invariants}  One of the many approaches that have been applied to quadrant walks relies on the notion of \emm invariants,, introduced by William Tutte in the seventies and eighties in his study of coloured planar triangulations~\cite{tutte-chromatic-revisited}. These invariants come by pairs and consist of two series $I(x;t)$ and $J(y;t)$ satisfying some properties (precise definitions will be given later). It was shown in~\cite{BeBMRa-17,bostan-non-DF,DHRS-17,DHRS-sing,hardouin-singer-20}
  that invariants play a crucial role in the classification of quadrant models. More precisely, it was first proved in~\cite{BeBMRa-17} that:
\begin{itemize}
  \item among the $79$ non-trivial quadrant models, exactly $13$ admit  invariants involving the series $Q(x,y;t)$: four if these models have a finite group, nine  an infinite group;
  \item one can use these invariants to prove, in a uniform manner,  that  $Q(x,y;t)$ is \emm  algebraic, for the four models with a finite group (and invariants),
     \item one can use these invariants to prove, in a uniform manner,  that  $Q(x,y;t)$ is \emm D-algebraic, for the nine models with an infinite group (and invariants).
    \end{itemize}
    Moreover, it was already known at that time that the other $19$
    models with a finite group are \emm transcendental, (i.e., not
    algebraic)~\cite{bostan-non-DF}. Hence for finite group models,
    algebraicity is \emm equivalent, to the existence of invariants
    involving $Q$. Similarly, the D-algebraicity results for the nine
    infinite group models were  complemented using differential Galois
    theory~\cite{DHRS-17,DHRS-sing}, proving that for infinite group
    models, D-algebraicity is equivalent to the existence of
    invariants involving $Q$. This equivalence has recently been generalized to quadrant walks with weighted steps (in the infinite group case)~\cite{hardouin-singer-20}.

 \subsection{Main results}    The aim of this paper is to explore the applicability  of invariants in the solution of three-quadrant models.  We focus on the nine models of Table~\ref{tab:sym},
 because the series $C(x,y;t)$ can then be described by an equation that is reminiscent of a quadrant equation, although more complex. These models (apart from the diagonal one) are already those considered  in~\cite{RaTr-19,dreyfus-trotignon}. Our first contribution is, for this (admittedly small) set of models, a collection of results that are perfect analogues of the above quadrant results:
\begin{itemize}
\item  exactly four of these nine   models admit  invariants involving the series $C(x,y;t)$: three of them (those of the \emm Kreweras trilogy, on the left of the table) have a finite group, and one ($\#6$) has  an infinite group. Moreover, these four models are also those that admit invariants involving the series $Q(x,y;t)$ (Proposition~\ref{prop:dec-three-quadrants});
\item we use these invariants to prove, in a uniform manner,  that the series $C(x,y;t)$ is \emm  algebraic, for the three models with a finite group (and invariants); this series was already known to be D-finite in $x$ and $y$~\cite{RaTr-19}, and conjectured to be algebraic~\cite{BM-three-quadrants};
     \item we use these invariants to (re)-prove  that the series $C(x,y;t)$ is \emm D-algebraic, for the (unique) model with an infinite group and invariants. This series was already known to be D-algebraic~\cite{dreyfus-trotignon}, but the expression that we obtain is new.
     \end{itemize}
     As in the quadrant case, these ``positive'' results are complemented by ``negative'' ones from the existing literature, which imply a tight connection between invariants and (D)-algebraicity:    the series $C(x,y;t)$ is transcendental for the fourth and fifth models with a finite group (simple walks and diagonal walks)~\cite{BM-three-quadrants}, and  non-D-algebraic for the other three models with an infinite group~\cite{dreyfus-trotignon}. These properties are summarized in Table~\ref{tab:sym}.

     Our second contribution is a new solution of two (transcendental)
     D-finite models via invariants:   simple walks with steps
     $\rightarrow,  \uparrow, \leftarrow, \downarrow$   and  diagonal
     walks  with steps $\nearrow, \nwarrow, \swarrow, \searrow$
     (fourth and fifth in the table). This may seem  surprising,
     since, for finite group models, invariants have so far been used
     to prove algebraicity results. But, as shown
     in~\cite{BM-three-quadrants} for both models, it is natural to
     introduce a series $A(x,y;t)$ that differs from $C(x,y;t)$ by a simple explicit D-finite series (expressed in terms of $Q(x,y;t)$), and then  the heart of the solution of~\cite{BM-three-quadrants} is to prove that $A(x,y;t)$ \emm is, algebraic. What we show here is how to re-prove this algebraicity via invariants.

     In the six cases that we solve via invariants, what we actually
     establish is an algebraic relation between the three-quadrant
     series $C(x,y;t)$ (or $A(x,y;t)$ in the simple and diagonal
     cases)  for the model $\cS$, and the quadrant series $\tQ(x,y;t)$
     for a companion model $\tS$ shown in Table~\ref{tab:sym}. The
     (D)-algebraicity of $C(x,y;t)$ (or $A(x,y;t)$) then follows from
     the (D)-algebraicity of $\tQ(x,y;t)$. (Note that $\tQ(x,y;t)$ is
     also D-algebraic for the three models that we do \emm not, solve,
     but the lack of invariants involving $C(x,y;t)$ prevents us from
     relating $C(x,y;t)$ and  $\tQ(x,y;t)$.) We thus obtain either
     explicit algebraic expressions of $C(x,y;t)$ or $A(x,y;t)$, or,
     for the model that is only D-algebraic, an  expression
     in terms of the \emm weak invariant,     of~\cite{BBMM-18}.

     Finally, for the five models that are (at least) D-finite, we
     derive from our results explicit algebraic expressions for the
     series $\sum_{i,j} H_{i,j} x^i y^j$, where $\left(H_{i,j}\right) _{(i,j)\in
       \Cc}$ is the discrete positive harmonic function associated with $\cS$-walks
     in $\Cc$. As can now be expected, this series
     is related to the  series describing the harmonic function of
      $\tS$-walks in the quadrant. This relation still holds, under natural assumptions, for the 6th model that we solve.

\subsection{Outline of the paper}
The paper is organized as follows. In Section~\ref{sec:tools}, we introduce  general tools, like basic functional equations for walks in $\Cc$, and also the key notion of \emm pairs of  invariants,. For each of the models $\tS$ of Table~\ref{tab:sym}, we give one or two pairs of $\tS$-invariants, borrowed from~\cite{BeBMRa-17}: one pair involves the \gf\ $\tQ(x,y;t)$ of quadrant walks with steps in $\tS$, and the other one, which only exists for finite group models, is rational. In Section~\ref{sec:dec} we construct,  for models of the Kreweras trilogy and for the $6$th model of the table, a new pair of $\tS$-invariants involving the \gf\ of walks in $\Cc$ with steps in $\cS$ (note the change in the step set). This construction is based on a certain \emm decoupling property, in these models, which we relate to the decoupling property used in the solution of quadrant walks with steps in $\cS$  by invariants (Proposition~\ref{prop:dec-three-quadrants}). A basic, but useful, \emm invariant lemma, (Lemma~\ref{lem:invariants}) allows us to relate these new invariants  to the known ones. This connection is made explicit in Sections~\ref{sec:K}, \ref{sec:RK},~\ref{sec:DK} and~\ref{sec:DA} for Kreweras steps, reverse Kreweras steps, double Kreweras steps, and for the $6$th model of Table~\ref{tab:sym},  respectively. We thus relate the \gf\ $C(x,y;t)$ for three-quadrant $\cS$-walks to the \gf\ $\tQ(x,y;t)$ for quadrant $\tS$-walks. In Section~\ref{sec:DF}, we apply the same approach to solve the simple and diagonal models, proving that the series $A(x,y;t)$ is algebraic in both cases. We finish in Section~\ref{sec:final} with comments and perspectives.

The paper is  accompanied by six \Maple\ sessions (one per step set)
available on the author's \href{http://www.labri.fr/perso/bousquet/publis.html}{webpage}.

\subsection{Kreweras' model}
\label{sec:K-statements}
To illustrate  our results, we now present those that deal with Kreweras' model, $\cS=\{\nearrow, \leftarrow, \downarrow\}$. Note that we often omit the dependency in $t$ of our series, denoting for instance $C(x,y)$ for $C(x,y;t)$. We also use the notation $\bx=1/x$ and $\by=1/y$.

\begin{theorem}\label{thm:K-U}
  The \gf\  $C(x,y)$  of walks with steps in $\{\nearrow, \leftarrow, \downarrow\}$ starting from $(0,0)$ and avoiding the negative quadrant is algebraic of degree $96$. It is given by the following identity:
\[
  (1-t(\bx+\by +xy)) C(x,y)=1 -t\by C_-(\bx) -t\bx C_-(\by),
\]
where $\C_-(x)$ is a series in $t$ with polynomial coefficients in $x$, algebraic of degree $24$. This series can be expressed explicitly in terms of the unique formal power series $\PIK\equiv \PIK(t)$ satisfying
\beq \label{PIK-k-def}
\PIK=t(2+\PIK^3).
\eeq
Indeed,
\[ 
  2\left(t\bx C_-(x)+ x - \frac 1 {2t}\right)^2= \frac{ (1-\PIK^3)^{3/2}}{\PIK^2}+
  (1-x\PIK)^2 \left(\frac 1 {\PIK^2}-\frac 1 x\right)
  + \left(\bx +\PIK-\frac{2x}\PIK\right) \sqrt{1-\PIK\frac{4+\PIK^3}4 x+\frac{ \PIK^2} 4 x^2}.
\] 
\end{theorem}

One recognizes, as announced, some ingredients of the quadrant \gf\ for the companion model $\tS=\{\rightarrow, \uparrow, \swarrow\}$, which satisfies~\cite{BMM-10,mishna-jcta}:
\beq\label{Qexpr-RK}
  \tQ(x,0)=\frac{\PIK(4-\PIK^3)}{16t} - \frac{t-x^2+tx^3}{2xt^2}+  \left(\bx +\PIK-\frac{2x}\PIK\right) \sqrt{1-\PIK\frac{4+\PIK^3}4 x+\frac{ \PIK^2} 4 x^2}.
\eeq
Note however that this series has degree $6$ only.

\smallskip
We can complement Theorem~\ref{thm:K-U} with the following one,  which counts walks ending on the diagonal: as before, the variable $t$ records the length, and the variable $x$ the abscissa of the endpoint. Note that the expressions in the two theorems only differ by a sign in front of the square root.

\begin{theorem}\label{thm:K-D}
  The \gf\ $D(x)$  of walks with steps in $\{\nearrow, \leftarrow, \downarrow\}$ starting from $(0,0)$, avoiding the negative quadrant and ending on the first diagonal is algebraic of degree $24$. It can be written explicitly in terms of the series $\PIK$ defined by~\eqref{PIK-k-def}:
\[ 
  \frac{\Delta(x)}2\left(xD(x)+  \frac 1 {t}\right)^2= \frac{ (1-\PIK^3)^{3/2}}{\PIK^2}+
  (1-x\PIK)^2 \left(\frac 1 {\PIK^2}-\frac 1 x\right)
  - \left(\bx +\PIK-\frac{2x}\PIK\right) \sqrt{1-\PIK\frac{4+\PIK^3}4 x+\frac{ \PIK^2} 4 x^2},
\] 
where
$  \Delta(x)= (1-tx)^2-4t^2\bx$.
\end{theorem}

\medskip

The \gf\ of \emm  excursions, (that is, walks ending at $(0,0)$) was conjectured in~\cite{BM-three-quadrants} to be algebraic of degree $6$.  More generally, it is natural to ask about the \gf\ $C_{i,j}$ of walks ending at a specific point $(i,j)$.  
In order to describe these series, we introduce  an extension of degree $4$ of $\qs(\PIK)$ (hence of degree $12$ over $\qs(t)$). First, due to the periodicity of the model (walks ending at position $(i,j)$ have length $-i-j$ modulo $3$), it makes sense to consider $t^3$ as the natural variable for this problem. Note that
\[
  \PIK^3=t^3(2+\PIK^3)^3,
\]
so that all series in $t^3$ can be written as series in $\PIK^3$. We then define the series $\PIIK$ and $\PIIIK$ as described in Table~\ref{tab:K}. Both are \fps\ in $t^3$, and we have the following  tower of extensions:
\[
  \qs(t^3) \stackrel{3}{\hookrightarrow} \qs(\PIK^3) \stackrel{2}{\hookrightarrow} \qs(\PIIK) \stackrel{2}{\hookrightarrow} \qs(\PIIIK).
  \]

\begin{table}[ht]
  \centering
  \begin{tabular}{|c|c|c|c|}
    \hline
    series & defining equation & expression & degree over $\qs(t)$\\
    \hline
    $\PIK$ & $\PIK=t(2+\PIK^3)$ &$-$ & 3\\
    $\PIIK$ & $4\PIIK (1-\PIIK)=\PIK^3, \quad \PIIK(0)=0$& $\frac{1-\sqrt{1-\PIK^3}}2$ & 6\\
    $\PIIIK$ & $2\PIIIK=\PIIK(1+\PIIIK^2)$&  $\frac{1-\sqrt{1-\PIIK^2}}{\PIIK}$ & 12\\ \hline
  \end{tabular}
 \vskip 2mm \caption{Relevant extensions of $\qs(t)$ for Kreweras (and reverse Kreweras) steps.}
  \label{tab:K}
\end{table}

\begin{cor}\label{cor:excursions-K}
Let us define $\PIK, \PIIK $ and $\PIIIK$ as above.   For any $(i,j) \in \Cc$, the \gf\ $C_{i,j}$ of walks avoiding the negative quadrant and ending at $(i,j)$ is algebraic of degree (at most) $12$, and belongs to $\qs(t, \PIIIK)$. More precisely, $t^{i+j}C_{i,j}$ belongs to $\qs(\PIIIK)$.  If $i=j$ or $i=j\pm1$ with $i,j\ge 0$, then $t^{i+j}C_{i,j}$ has degree $6$ at most and belongs to $\qs(\PIIK)$.  
\end{cor}
For instance, we have
    \allowdisplaybreaks
\begin{align}
\label{C00-K}
  C_{0,0}&={\frac {  \left(1+2\,\PIIK- 2\,{\PIIK}^{2} \right) \left( 3\,{\PIIK}^{3}-2\,{\PIIK}^{2}-4\,\PIIK+4 \right)
     }{4(1-\PIIK)}},
\\
  tC_{0,1} = \frac {C_{0,0}-1} {2} & =
  {\frac {\PIIK \left( -6 {\PIIK}^{4}+10 {\PIIK}^{3}+7 {\PIIK}^{2}-18 \PIIK+8
      \right) }{8(1-\PIIK)}}, \nonumber
  \\
t^2  C_{1,1}&= {\frac { \PIIK\left( 35 {\PIIK}^{6}-100 {\PIIK}^{5}+20 {\PIIK}^{4}+144 {\PIIK}^{3
}-96 {\PIIK}^{2}-32 \PIIK+32 \right)  
}{ 64\left(1- \PIIK \right)   \left(1+2 \PIIK- 2 {\PIIK}^{2} \right) }}, \nonumber
\\ \label{Cm10-K}
  t^2 C_{-1,0}&= {\frac {\PIIIK \left( {\PIIIK}^{3}+3 {\PIIIK}^{2}-3 \PIIIK+1 \right) }{ {\PIIIK}^{4}+4
       {\PIIIK}^{3}-6 {\PIIIK}^{2}+4 \PIIIK+1 }}.
  \end{align}
We can derive from such results asymptotic estimates for the number of walks ending at $(i,j)$ via \emm singularity analysis,~\cite[Ch.~VII.7]{flajolet-sedgewick}. For instance,
\[
    c_{0,0}(3m) \sim -\frac{27\sqrt 3}{4\Gamma(-3/4)} 3^{3m} (3m)^{-7/4}.
\]
Note that this asymptotic behaviour was already known, but only up to a constant~\cite{mustapha-3quadrant}. More generally, we obtain an explicit description of the \emm discrete positive harmonic function, associated with $\cS$-walks confined to the three-quadrant cone $\cC$.

\begin{cor}\label{cor:harmonic-K}
For $(i,j)\in \cC$, there exists a positive constant $H_{i,j}$ such that, as $n\rightarrow \infty$ with $n+i+j \equiv 0 \! \mod 3$,
\beq\label{cij-est}
  c_{i,j}(n) \sim -\frac{H_{i,j}}{\Gamma(-3/4)}\, 3^n n^{-7/4}.
\eeq
The function $H$ is \emm discrete harmonic, for the step set $\{\nearrow, \leftarrow, \downarrow\}$. That is,
\beq\label{harmonic}
  H_{i,j} = \frac 1 3 \left( H_{i-1,j-1}+ H_{i+1,j} + H_{i,j+1}\right),
\eeq
where  $H_{i,j}=0$ if $(i,j) \not \in \cC$. By symmetry, $H_{i,j}=H_{j,i}$. The \gf\
\beq\label{H-def}
  \Hc(x,y):=\sum_{j\ge 0, i\le j} H_{i,j} x^{j-i} y^j,
\eeq
which is a \fps\ in $x$ and $y$, is algebraic of degree $16$, given by
\beq\label{H-eq}
  \left(1+xy^2+x^2y-3xy\right)  \Hc(x,y)=  \Hc_-(x)+\frac 1 2 \left(2+xy^2-3xy\right) \Hc_d(y),
\eeq
where
\beq\label{Hminus-K}
  \Hc_-(x):=\sum_{i>0} H_{-i,0} x^i=\frac {9x}{2} \sqrt{ \frac{1+2x}{1-x} \sqrt{\frac {4-x}{1-x}}+2}
\eeq
and
\beq\label{Hd-K}
  \Hc_d(y):=\Hc(0,y)=\sum_{i\ge 0} H_{i,i}y^i = \frac 9  { (1-y) \sqrt{y(4-y)}}\sqrt{ \frac{1+2y}{1-y} \sqrt{\frac {4-y}{1-y}}-2}\,.
\eeq
\end{cor}
\hyphenation{Tro-ti-gnon}
A first expression of the harmonic function, implying its algebraicity, was  given by Trotignon~\cite{trotignon-harmonic}. It is, however,  less explicit.  The connection between walks in $\cC$ with steps in $\tS$ and walks in $\Qc$ with steps is the companion model $\tS$, already observed at the level of \gfs\ (see Theorem~\ref{thm:K-U} and~\eqref{Qexpr-RK}), is also visible at the level of harmonic functions. Indeed, the number of quadrant $\tS$-walks of length~$n$ ending at $(i,j)\in \Qc$ is
\beq\label{RK-quadrant-asympt}
\tilde q_{i,j}(n) \sim \frac{h_{i,j} }{\Gamma(-3/2)} 3^n n^{-5/2}
\eeq
(for $n\equiv i+j  $ mod $3$) with
\beq\label{harmonic-RK-quadrant}
  \sum_{i\ge0} h_{i,0} x^i= \frac 9 4 \left( \frac{1+2x}{1-x} \sqrt{\frac{4-x}{1-x}}+2\right).
\eeq
This can be derived from~\eqref{Qexpr-RK} by singularity analysis. See also~\cite{raschel-harmonic} for a more general result that applies to any quadrant model with small steps and zero drift. We explain in Remark~\ref{rem:invariants-harmonic}  how one can predict, under a natural assumption, the relation between the $\cS$-harmonic function $H_{i,j}$ in~$\cC$ and the $\tS$-harmonic function $h_{i,j}$ in~$\Qc$. Note that~\eqref{RK-quadrant-asympt} describes also the asymptotic number of walks with steps in the original model $\cS$ contained in the non-positive quadrant, joining $(0,0)$ to $(-i,-j)$.

\medskip
  
Let us finally consider the total number  of $n$-step walks avoiding the negative quadrant, which we denote by $c_n$.

\begin{cor}\label{cor:all-K}
  The \gf\ $C(1,1)$ counting all walks confined to the three-quadrant cone $\Cc$ is algebraic of degree $24$, and given by:
  \begin{align} \label{C11-K}
    \frac 1 2 (1-3t)^2\left( C(1,1)+\frac 1 t\right)^2
&=  2\left(t C_-(1)+ 1 - \frac 1 {2t}\right)^2
\\ &   = \frac{ (1-\PIK^3)^{3/2}}{\PIK^2}+
    (1-\PIK)^2 \left(\frac 1 {\PIK^2}-1\right)
  + \left(1+\PIK-\frac{2}\PIK\right) \sqrt{1-\PIK\frac{4+\PIK^3}4 +\frac{ \PIK^2} 4 }, \nonumber
\end{align}
where $\PIK$ is the series defined by~\eqref{PIK-k-def}. The series $C(1,1)$ has radius of convergence $1/3$, with dominant singularities at $\zeta/3$, where $\zeta$ is any cubic root of unity.  For the first order asymptotics of the coefficients $c_n$, only the singularity at $1/3$ contributes, and
  \beq\label{cn-est-K}
    c_n \sim \frac{3^{3/4}\sqrt{2-\sqrt 2}}{\Gamma(5/8)} 3^{n} n^{-3/8}.
\eeq
\end{cor}

\section{General tools}
\label{sec:tools}
We first introduce some basic notation.
For a ring $R$, we denote by $R[t]$
 (resp.~$R[[t]]$) the ring of polynomials (resp. formal power series) 
in $t$ with coefficients in $R$. If $R$ is a field, then $R(t)$ stands for the field of rational functions in~$t$, and  $R((t))$ for the
field of Laurent series in $x$, that is, series of the form
$ 
\sum_{n \ge n_0} a_n t^n,
$
with $n_0\in \zs$ and $a_n\in R$.
This notation is generalized to several variables.
For instance, the \gf\  $C(x,y;t)$ of walks confined to~$\cC$ belongs to $\qs[x,\bx,y,\by ][[t]]$, where we denote $\bx=1/x$ and $\by=1/y$. We often omit in our notation the dependency in $t$ of our series, writing for instance $C(x,y)$ instead of $C(x,y;t)$. For a series $F(x,y) \in \qs[x, \bx, y, \by][[t]]$ and two integers
$i$ and $j$, we denote by $F_{i,j}$ the coefficient of $x^i y^j$ in $F(x,y)$. 
This is a series in $\qs[[t]]$. Similarly, if  $F(x) \in \qs[x, \bx][[t]]$, then  $F_{i}$  stands for the coefficient of $x^i$ in $F(x)$.

\medskip
We consider  a set of ``small'' steps $\cS$ in $\zs^2$, meaning that $\cS\subset  \{-1, 0, 1\}^2\setminus\{(0,0)\}$.
We define the \emm step polynomial, $S(x,y)$ by
\[
  S(x,y)=\sum_{(i,j)\in \cS} x^i y^j,
\]
and six (Laurent) polynomials $H_0$, $H_-$, $H_+$, $V_0$, $V_-$ and $V_+$ by
\beq\label{S-split-H}
S(x,y)= \by H_-(x)+ H_0(x)+yH_+(x)
 = \bx V_-(y) + V_0(y) + xV_+(y).
\eeq
The notation $H$ (resp. $V$) recalls that these polynomials record horizontal (resp. vertical) moves: for instance, $H_-(x)$ describes the horizontal displacements of steps that move down. The group associated to $\cS$ can be easily defined at this stage, but we do not use it in this paper, and simply refer the interested reader to, e.g.,~\cite{BMM-10}.

\subsection{Functional equations for walks avoiding a quadrant}
\label{sec:eqfunc}

In this subsection, we first restrict our attention to  step sets~$\cS$ satisfying the following two assumptions:
\beq\label{conds}\begin{aligned}
   \bullet & \ \cS \text{ is  symmetric in the first diagonal,} \hskip 70mm \ \\
   \bullet &\  \cS\text{ does not contain the steps } (1,-1) \text{ nor } (-1,1).
 \end{aligned}
 \eeq
This corresponds to the  models shown in Table~\ref{tab:sym}, apart
from the diagonal model. As we shall explain, the latter model should still be considered in the same class, as the associated \gf\ satisfies the same kind of functional equation.

It is easy to write a functional equation defining the series $C(x,y)$, based in a step-by-step construction of the walks:
\[ 
 C(x,y)= 1 +tS(x,y)C(x,y) -t\by H_-(x) C_{-}(\bx) -t\bx H_-(y)C_{-}(\by) -t\bx\by
 C_{0,0} \mathbbm 1_ {(-1,-1) \in \cS},
\] 
where  the series $C_{-}(\bx)$ counts walks ending on  the negative $x$-axis:
\beq  \label{Chv}
 C_{-}(\bx) =  \sum_{i<0, n \geq 0} c_{i,0}(n)x^i t^n.
\eeq
In the above functional equation, each term with a minus sign  corresponds to a forbidden move: moving down from the negative $x$-axis, moving left from the negative $y$-axis, or performing a South-West step from the point $(0,0)$. Observe that we have exploited the $x/y$-symmetry.
The above equation rewrites as:
 \beq\label{eqfunc-gen}
 K(x,y)C(x,y)= 1 -t\by H_-(x) C_{-}(\bx) -t\bx H_-(y)C_{-}(\by) -t\bx\by
 C_{0,0} \mathbbm 1_{(-1,-1)\in \cS},
 \eeq
 where $K(x,y):=1-tS(x,y)$ is the \emm kernel, of the equation.

 \begin{figure}[t!]
   \centering
 \scalebox{0.95}{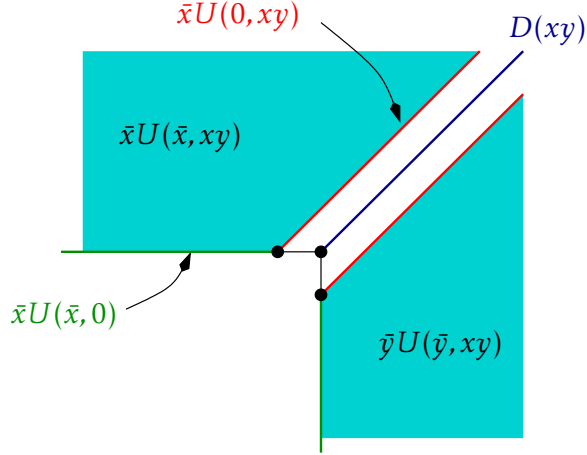}   
   \caption{The series $D(xy)$ counts walks ending on the diagonal, and $\bx U(\bx, xy)$ those ending above the diagonal.}
   \label{fig:DU-def}
 \end{figure}

 Following~\cite{RaTr-19}, we will now work in convex cones by splitting the three-quadrant cone in two symmetric halves. The rest of this subsection simply  rephrases Section~2.2 of~\cite{RaTr-19}, with a slightly different presentation. (Another strategy, namely splitting in three quadrants, is applied in~\cite{BM-three-quadrants,mbm-wallner-conf,mbm-wallner,elvey-price-trois-quarts}.)  We write
 \beq\label{C-split}
  C(x,y)=\bx U(\bx, xy)+ D(xy)+ \by U(\by,xy),
\eeq
with
$D(y) \in \qs[y][[t]]$ and $U(x,y) \in \qs [x,y][[t]]$. Observe that these properties, together with the above identity,  define the series $D$ and $U$ uniquely: the series $\bx U(\bx, xy)$ (resp. $D(xy)$, $\by U(\by,xy)$) counts walks ending strictly above (resp. on, strictly below) the diagonal; see Figure~\ref{fig:DU-def}. We will now write functional equations for $U$ and $D$, based again on a step-by-step construction (we could also derive them algebraically from~\eqref{eqfunc-gen}, but we prefer a combinatorial argument). This requires to  classify steps of $\cS$ according to whether they lie above, on or below the main diagonal. We write accordingly
\beq\label{V-def}
  S(x,y)=\bx\, \tV_+(xy) + \tV_0(xy) +x  \tV_-(xy) .
\eeq
Note that this is only possible because we have forbidden steps $(1,-1)$ and $(-1,1)$, which would contribute terms of the form $x^2 (\bx\by)$ and $\bx^2 (xy)$, respectively. Let us explain the notation $\tV_+, \tV_0, \tV_-$, which is reminiscent of~\eqref{S-split-H}. We define the \emm companion, set of steps
\beq\label{tS-def}
\tS:= \{(j-i,j): (i,j) \in \cS\},
\eeq
with step polynomial $\tS(x,y)= S(\bx,xy)$ (we hope that using the same notation $\tS$ for the  set of steps and its generating polynomial will not cause any inconvenience). Then
\[ 
\tS(x,y)= S(\bx, xy)=
x\tV_+(y) + \tV_0(y) + \bx\,\tV_-(y),
\] 
which is indeed  the decomposition of $\tS$ along horizontal moves. Note that the transformation $\cS\mapsto \tS$ maps Kreweras steps to reverse Kreweras steps, and vice-versa, and leaves double Kreweras steps globally unchanged. The full correspondence between $\cS$ and $\tS$ is shown in Table~\ref{tab:sym}.

\medskip
We claim that the \gf\ of walks ending on the diagonal satisfies:
\begin{multline*}
  D(xy)= 1+ t \tV_0(xy) D(xy) - t\bx\by D_0\mathbbm 1_{(-1,-1) \in \cS}
  +2t \left(x  \tV_-(xy)\right) \left(\bx U(0,xy)\right) -2t \bx\by U_{0,0}\mathbbm 1_{(0,-1) \in \cS}.
\end{multline*}
The two terms involving $D$ on the right-hand side count walks whose last step starts from the diagonal, and those involving $U$ count walks whose last step  starts from the lines $j=i\pm 1$. For instance, $\bx U(0,xy)$ counts walks ending on the line $j=i+1$, and the multiplication by $ t \left(x  \tV_-(xy)\right)$ corresponds to adding a sub-diagonal step. The factor $2$ accounts for the $x/y$ symmetry of the model. The final term prevents from moving from $(-1,0)$ to $(-1,-1)$ (or symmetrically).
 For walks ending (strictly) above the diagonal, we obtain:
\begin{multline*}
  \bx U(\bx,xy)=    t \bx \tV_+(xy)  D(xy)\\
 + t S(x,y) \bx U(\bx, xy) -t \by H_-(x) \bx U(\bx,0)
  -t \left(x  \tV_-(xy)\right) \bx U(0,xy)  + t \bx\by U_{0,0}\mathbbm 1_{(0,-1) \in \cS}.
\end{multline*}
Again, the terms involving $D$ (resp. $U$) count walks whose last step starts on (resp. above) the diagonal. The term involving
$H_-(x)$ corresponds to walks that would enter the negative quadrant through the negative $x$-axis, and the term involving $\tV_-(xy)$ counts walks that would in fact end on the diagonal. The final term corresponds to walks ending at $(-1,0)$, extended by a South step: they have been subtracted twice, once in each of the previous two terms.
  
We now perform the change of variables  $(x,y) \mapsto (\bx,xy)$. This gives:
\begin{align}\label{eqD2}
\left(1- t \tV_0(y)\right)   D(y)&= 1- t\by D_{0}\mathbbm 1_{(-1,-1) \in \cS}
  +2t  \tV_-(y) U(0,y) -2t\by U_{0,0}\mathbbm 1_{(0,-1) \in \cS},
\\
\left(1- t S(\bx,xy) \right)  x U(x,y)&=    t x \tV_+(y)  D(y)
  -t \by H_-(\bx)  U(x,0)
                                        -t   \tV_-(y)  U(0,y)  + t \by U_{0,0}\mathbbm 1_{(0,-1) \in \cS}.
    \label{eqU2}
\end{align}
Using the first equation, we can eliminate $U(0,y)$ (and $U_{0,0}$) from the second. Multiplying finally by $2y$, we obtain the functional equation that we are going to focus on:
\begin{multline*}
  2xy \tK(x,y) U(x,y)=   y +  y\, \big(t  \tV_0(y)+ 2tx \tV_+(y) -1\big) D(y)
  -2t H_-(\bx) U(x,0)  %
  -  t D_0\mathbbm 1_{(-1,-1) \in \cS}.
\end{multline*}
It will be  convenient to write it in terms of $\tS$ only: we thus observe that \beq\label{rewrite}
(-1,-1) \in \cS \Leftrightarrow (0,-1) \in \tS \qquad \text{and} \qquad
H_-(\bx)=x \tH_-(x),
\eeq
where we write
\beq \label{HV-def-alt}
  \tS(x,y)=\bx \, \tV_-(y) + \tV_0(y) +x  \tV_+(y)
= \by \tH_-(x) +\tH_0 +y \tH_+(x).
\eeq 

\begin{lemma}\label{lem:func_eq}
  For a step set $\cS$ satisfying the assumptions~\eqref{conds}, the series $U(x,y)$ and $D(y)$ defined by~\eqref{C-split} are related by
\begin{multline}\label{eq-U}
  2xy \tK(x,y) U(x,y)=   y +  y\, \big(t  \tV_0(y)+ 2tx \tV_+(y) -1\big) D(y)
  -2t  x \tH_-(x)  U(x,0)  
  -  t D_0\mathbbm 1_{(0,-1) \in \tS},  
\end{multline}
with $\tK(x,y):=K(\bx,xy)=1-tS(\bx,xy)$. Note that $\tK(x,y)$ is the kernel associated with the step set $\tS$ defined by~\eqref{tS-def}, which may not be $x/y$-symmetric.

This equation holds for the diagonal model $\cS=\{ \nearrow, \nwarrow,
\swarrow, \searrow\}$ as well, provided we define the series $D$ and~$U$ by
 \beq\label{C-split-diag}
  C(x,y)=\bx^2 U(\bx^2, xy)+ D(xy)+ \by^2 U(\by^2,xy),
  \eeq
  and take $\tS= \{\nearrow, \uparrow, \swarrow, \downarrow\}$, that
  is,
  \[
    \tK(x,y)= 1-t \tS(x,y)=
    1-tS\left(\sqrt {\bx},\sqrt xy\right)= 1-t(xy+y+\bx\by +\by).
    \]
\end{lemma}

\begin{proof}
  We have already justified the equation under the
  assumptions~\eqref{conds}. Now consider the diagonal model.  The
  only points $(i,j)$ that can be visited by a walk starting from
  $(0,0)$ are such that $i+j$ is even. This allows us to write
  $C(x,y)$ as~\eqref{C-split-diag}, where $U(x,y)$ and $D(y)$ are in
  $\qs[x,y][[t]]$ and $\qs[y][[t]]$, respectively. One then writes
  basic step-by-step equations for $D(xy)$ and $\bx^2 U(\bx^2, xy)$,
  and then performs the change of variables $(x,y) \mapsto (
  \sqrt{\bx}, \sqrt x y)$. This leads to the following counterparts
  of~\eqref{eqD2} and~\eqref{eqU2}:
  \begin{align*}
     \left(1-t(\by+y)\right) D(y)&= 1-t\by D_0+2t\by U(0,y)-2t\by U_{0,0},\\
    \left(1-t(y+\by +\bx\by+xy)\right)x U(x,y)&=txyD(y)
-t\by (1+x) U(x,0)  -t\by U(0,y)+ t\by U_{0,0}.
  \end{align*}
 The combination of these two equations gives
  \beq\label{eqT-C-diag}
  2xy \left(1-t(y+\by+xy+\bx\by)\right)
  U(x,y)=y+y\big(t(y+\by)+2txy-1\big) D(y) -2t(1+x)U(x,0)-tD_0,
  \eeq
  which coincides with~\eqref{eq-U}.
\end{proof}

\begin{remark}
  The (single) equation in Lemma~\ref{lem:func_eq} really defines the \emm two, series $U(x,y)$ and $D(y)$. Indeed, the equation~\eqref{eqD2} relating $D(y)$ and $U(0,y)$ can be recovered by setting $x=0$ in~\eqref{eq-U}. \qee
\end{remark}

 \medskip

Our solution of three-quadrant models will  involve the \gf\
of  walks with steps in $\tS$ confined to the first (non-negative) quadrant $\Qc$. This series $\tQ(x,y)\equiv \tQ(x,y;t)\in \qs[x,y][[t]]$ 
satisfies a similar looking equation~\cite{BMM-10}:
 \beq\label{eqfunc-gen-qu}
xy \tK(x,y)\tQ(x,y)= xy -tx \tH_-(x) Q(x,0) -ty \tV_-(y)\tQ(0,y) +t   \tQ_{0,0} \mathbbm 1_{(-1,-1)\in \tS},
 \eeq
 where we have used the notation~\eqref{HV-def-alt}.  This equation is simpler than~\eqref{eq-U}, because its right-hand side, apart from the simple term $xy$, is the sum of a function of $x$ and a function of $y$. This is not the case in~\eqref{eq-U}, because the factor $\left(t  \tV_0(y)+ 2tx \tV_+(y) -1\right)$ involves~$x$. However, the \emm square,  of this term is a function of $y$ only --- modulo $\tK(x,y)$. This property is already exploited in~\cite{RaTr-19}.

 \begin{lemma}\label{lem:square}
     Let $\tS$ be a collection of small steps,
     and define $\tV_+, \tV_0$ and $\tV_-$ by~\eqref{HV-def-alt}. Then
   \[
     \big(t  \tV_0(y)+ 2tx \tV_+(y) -1\big)^2=\Delta(y)-4tx  \tV_+(y) \tK(x,y),
   \]
   where $\tK(x,y)=1-t\tS(x,y)$ and 
\[ 
     \Delta(y)= \left(1-t\tV_0(y)\right)^2-4t^2\tV_-(y) \tV_+(y)
\] 
     is the discriminant (in $x$) of $x\tK(x,y)$.
 \end{lemma}
 The proof is a simple calculation, which we leave to the reader.

 \subsection{Invariants}
 We now define the notion of \emm invariant, that we use in this paper. We refer to~\cite{BeBMRa-FPSAC-16,BeBMRa-17} for a slightly different notion that applies to rational functions only. The definition that we adopt here allows us a uniform treatment of all the models that we solve, whereas in~\cite{BeBMRa-17}, quadrant walks with reverse Kreweras steps had to be handled by a different argument. We will only use our notion of invariants for small step models, but this subsection could actually be applied to any model in $ \{-1, 0, 1, 2, 3, \ldots\}^2$.

 For reasons that are related to the form of the functional equation~\eqref{eq-U}, we  consider in this subsection a set of steps that we denote $\tS$ rather than $\cS$. Before all, observe that the series
 \beq\label{K-inv}
   \frac 1{\tK(x,y)}=\frac 1{1-t \tS(x,y)} =\sum_{n\ge 0} t^n \tS(x,y)^n
 \eeq
  counts \emm all, walks with steps in $\tS$, starting from the
  origin, by the length (variable $t$) and the coordinates of the
  endpoint (variables $x$ and $y$).  The coefficient of $t^n$ in this series  is a Laurent polynomial in $x$ and $y$. As soon as $\tS$ is not contained in a half-plane, the collection of these polynomials has unbounded degree and valuation.

  Recall that $ \qs(x,y)((t))$ is the field of Laurent series in $t$ with rational coefficients in $x$ and~$y$. We denote by  $ \qs_{{mult}}(x,y)((t))$ the subring consisting of series in which for each $n$, the coefficient of $t^n$ is of the form $p(x,y)/(d(x)d'(y))$, where $p(x,y) \in \qs[x,y]$, $d(x) \in \qs[x]$ and $d'(y)\in \qs[y]$. The reason why we focus on this subring is that:
  \begin{itemize}
  \item all the series that we handle are of this type,
    \item later (in the proof of Lemma~\ref{lem:invariants}) we will expand their coefficients as Laurent series in~$x$ and $y$, and with this condition the two expansions can be performed independently, without having to prescribe an order on $x$ and $y$ (as would be the case if we had to expand $1/(x-y)$, for instance).
  \end{itemize}
 
 \begin{definition}\label{def:bounded}
   A Laurent series $H(x,y)$ in $ \qs_{{mult}}(x,y)((t))$ is said to have \emm poles of bounded order at $0$, if the collection of its coefficients (in the $t$-expansion) has poles of bounded order at $x=0$, and also at $y=0$. In other words, there exists a monomial $x^i y^j$ such that  $x^i y^j H(x,y)$ expands as a series in $t$ whose coefficients have no pole at $x=0$ nor at $y=0$.
 \end{definition}

 Clearly, the series $ 1/{\tK(x,y)}$ shown in~\eqref{K-inv} lies in $ \qs_{{mult}}(x,y)((t))$, but does not have poles of bounded order at $0$ (unless $\tS$ is contained in the first quadrant).
 
\begin{definition}[{\bf Divisibility}] \label{def:divisible}
   Fix a step set $\tS$ with kernel $\tK(x,y)=1-t\tS(x,y)$. A Laurent series $F(x,y)$ in $ \qs_{{mult}}(x,y)((t))$ is said to be \emm divisible, by $\tK(x,y)$ if the ratio $F(x,y)/\tK(x,y)$ has poles of bounded order at $0$.
    \end{definition}

    Note that if $F(x,y)$ in $ \qs_{{mult}}(x,y)((t))$, the ratio $F(x,y)/\tK(x,y)$ is always an element of $ \qs_{{mult}}(x,y)((t))$ (by~\eqref{K-inv}).  The following alternative characterization of divisibility will not be used, but it may clarify the notion.

    \begin{lemma}\label{lem:divisibility}
 Assume that $\tK(x,y)$ has valuation $-1$ and degree $1$ in $x$ and in $y$.      The Laurent series $F(x,y)$ in $ \qs_{{mult}}(x,y)((t))$ is divisible by $\tK(x,y)$ if and only if
      \begin{enumerate}
      \item $F(x,y)$ has poles of bounded order at $0$,
        \item if $\tX\equiv \tX(y)$ (resp. $\tY\equiv \tY(x)$) denotes the unique root of $\tK(\cdot, y)$ (resp. $\tK(x, \cdot)$) that is a formal power series in $t$ (with coefficients in some algebraic closure of $\qs(y)$ or $\qs(x))$, then $F(\tX,y)=F(x,\tY)=0$.
      \end{enumerate}
    \end{lemma}
\begin{proof}
      Assume that $F(x,y)$ is divisible by $\tK(x,y)$, and write
      \beq\label{FH-div}
        F(x,y)= \tK(x,y) H(x,y),
      \eeq
      where $H(x,y)$ has poles of bounded order at $0$. Since $\tK(x,y)$ is a Laurent polynomial in $x$,~$y$ and $t$, then $F(x,y)$ has poles of bounded order at $0$ as well. Regarding  Property~(2), it suffices to prove it for $\tX$, by symmetry. Observe that the equation $\tK(x, y)=0$ can be rewritten as
         \[
        x= tx \tS(x,y) = t \left( \tV_-(y)+ x \tV_0(y) + x^2 \tV_+(y)\right).
      \]
      This shows that exactly one of the roots  of $\tK(\cdot,y)$, denoted by $\tX$,
      is a formal power series in $t$; moreover, it has  constant term $0$ and coefficients in $\qs(y)$. Since $F(x,y)$ and $H(x,y)$ have poles of bounded order at $0$, we can specialize $x$ to $\tX$ in~\eqref{FH-div}, and this proves that $F(\tX,y)=0$. 

      Now assume that the two conditions of the lemma hold, and let us prove that the series $H(x,y)$ defined by~\eqref{FH-div} has poles of bounded order at $0$. By symmetry, it suffices to prove this for the variable $x$. The second root of $\tK(\cdot, y)$ is
      \[
        \tX'(y)= \frac{\tV_-(y)}{\tV_+(y)} \ \frac 1 {\tX(y)}.
      \]
      It is a Laurent series in $t$ of valuation $-1$, with coefficients in $\qs(y)$. A partial fraction expansion of $1/\tK(x,y)$ (in the variable $x$) gives
      \[
        \frac 1 {\tK(x,y)}=\frac 1{\sqrt{\Delta(y)}} \left( \frac 1{1-\bx \tX(y)} + \frac x {1-x/\tX'(y)}\right).
      \]
  By  assumption, $F(x,y)$ has poles of bounded order at $x=0$.     Since $\tX(y)=\LandauO(t)$ and $1/\tX'(y)=\LandauO(t)$, only the first part of the above series has poles at $x=0$, and it suffices to prove that 
      \beq\label{F-rat}
        \frac{F(x,y)}{1-\bx\tX(y)} = F(x,y) \sum_{n\ge 0} \bx^n \tX(y)^n,
      \eeq
      has poles of bounded order at $x=0$.  By assumption, there exists an integer $i$ such that
      \[
      x^i  F(x,y)= \sum_{n\ge 0}\frac{p_n(x,y)}{ d_n(x) d'_n(y)} t^n,
      \]
for polynomials $p_n$, $d_n$ and $d'_n$ such that $d_n(0)\not =0$.     
Then, by the second assumption,
  \[
   \frac{ x^iF(x,y)}{1-\bx\tX} =        \frac{x^iF(x,y)-\tX^iF(\tX,y)}{1-\bx\tX}
    =
    \sum_{n\ge 0}\left( \frac{p_n(x,y)}{ d_n(x) }-
      \frac{p_n(\tX,y)}{d_n(\tX) }\right) \frac x {x-\tX} \frac{t^n}{d'_n(y)}.
  \]
  The $n$th summand can be written
  \[
    \frac{ x   q_n(x,\tX,y)}{d_n(x) d_n(\tX)}\cdot \frac{t^n}{d'_n(y)},
  \]
for some trivariate polynomial $q_n$. As a series in $t$, it belongs to  $\qs_{mult}(x,y)((t))$, and its coefficients have no pole at $x=0$ since $d(0)\not = 0$ (they are  even  multiples of $x$).  We thus conclude that the series~\eqref{F-rat}, as claimed, has poles of bounded order at $x=0$. 
  \end{proof}
    
\begin{definition}[{\bf Invariants}] \label{def:invariants}
   Fix a step set $\tS$ with kernel $\tK(x,y)=1-t\tS(x,y)$.  Let $I(x)$ and $J(y)$ be Laurent series in~$t$ with coefficients in $\qs(x)$ and $\qs(y)$ respectively. 
   We say that $(I(x), J(y))$ is   a \emph{pair of $\tS$-invariants} if $I(x)-J(y)$ is divisible by $\tK(x,y)$.
\end{definition}

Note that $I(x)-J(y)$ is always an element of  $\qs_{{mult}}(x,y)((t))$.
Also, if $(I(x), J(y))$ is   a pair of invariants, then  $I(x)$ and $J(y)$ have poles of bounded order at $0$.

For any series $A\in \qs((t))$, the pair $(A,A)$ is a (trivial) invariant. More interesting examples will be given in the next subsection; see for instance~\eqref{I0-K}.

The term \emm invariant, comes from the fact that, if both roots $\tX$ and $\tX'$ of $\tK(\cdot, y)$ can be substituted for $x$ in $I(x)$, then $I(\tX)=I(\tX')$. Of course, a similar observation holds for $J(y)$.
\noindent

\begin{lemma}\label{lem:linear} The componentwise sum of two pairs of invariants $\left(I_1(x), J_1(y)\right)$ and $\left(I_2(x), J_2(y)\right)$
    is still a pair of invariants. The same holds for their componentwise product.
  \end{lemma}
  
 \begin{proof}
   The first result is obvious by linearity. For the second, we observe that
   \begin{align*}
      I_1(x) I_2(x) -J_1(y) J_2(y)&= \left(I_1(x)-J_1(y)\right) I_2(x)
                                    +J_1(y) \left(I_2(x)-J_2(y)\right) \\
     &= \tK(x,y) \left( H_1(x,y)I_2(x) + J_1(y) H_2(x,y)\right),
   \end{align*}
 where  $H_k(x,y)=(I_k(x)-J_k(y))/\tK(x,y)$, for $k=1,2$.  The result follows because series with poles of bounded order at $0$ form a ring.
  \end{proof}

 The following lemma is a key tool  when playing with  invariants.

 \begin{lemma}[{\bf Invariant lemma}] \label{lem:invariants}
 Let $I(x)$ and $J(y)$ be Laurent series in~$t$ with coefficients in $\qs(x)$ and $\qs(y)$ respectively. Assume that for all $n$, the coefficient of $t^n$ in the ratio series
   \[
     \frac{I(x)-J(y)}{\tK(x,y)}
   \]
   vanishes at $x=0$ and at $y=0$.   Then  $(I(x), J(y))$ is a pair of invariants by Definition~\ref{def:invariants}, but in fact 
 $I(x)$ and $J(y)$ depend only on $t$, and are equal. That is, the above ratio is zero.   
 \end{lemma}
 \begin{proof}
   The argument is adapted from Lemma~4.13 in~\cite{BeBMRa-17}. Let us assume, without loss of generality, that $I(x)$ and
   $J(y)$ contain no negative power of $t$. By assumption, we have
   \[
     I(x)-J(y) = xy \tK(x,y) G(x,y),
   \]
   where for each $n$, the coefficient of $t^n$ in $G(x,y) $ is of the form $p(x,y)/(d(x) d'(y))$ and has no pole at $x=0$ nor $y=0$.
    Hence this coefficient expands  as a formal power series in $x$ and~$y$, and we consider $G(x,y)$ itself as an element of $\qs[[x,y,t]]$.
    We define a total order on monomials $t^n x^i y^j$, with $(n,i,j) \in \ns^3$, by the lexicographic order on the exponents $(n,i,j)$. Observe that $xy\tK(x,y) \in \qs[t,x,y]$, and the smallest monomial that occurs in it is $xy$. Assume that $G(x,y)\not=0$, and  let $M$ be the smallest monomial in $G(x,y)$. Then $xyM$ is the smallest monomial in $xy\tK(x,y)G(x,y)$, hence in $I(x)-J(y)$ (once expanded as a series in $t$, whose coefficients are Laurent series in $x$ or in $y$). But the latter series cannot contain monomials involving both~$x$ and $y$. Hence $G(x,y)=0$ and the lemma is proved.
 \end{proof}

 \subsection{Some known invariants}
 \label{sec:known-inv}

 For each model $\tS$ of Table~\ref{tab:sym}, at least one non-trivial pair of invariants is already known from~\cite[Sec.~4]{BeBMRa-17}, as we now review.

 Consider for instance the set $\tS=\{\nearrow, \leftarrow, \downarrow\}$ of Kreweras' steps, with $\tK(x,y)=1-t(xy+\bx+\by)$. The first pair of invariants exhibited in~\cite{BeBMRa-17} is rational,  and symmetric in $x$ and $y$: 
 \[
   I_0(x)=\bx^2-\bx/t -x, \qquad J_0(y)=I_0(y).
 \]
 Indeed, it is elementary to check that 
\beq\label{I0-K}
   I_0(x)-J_0(y)= \tK(x,y) \frac{x-y}{txy},
 \eeq
and the ``series'' $\frac{x-y}{txy}$  has poles of bounded order at $0$.  To describe the second pair, let $\tQ(x,y)$ be the \gf\ of quadrant walks with steps in $\tS$. Then $\tQ(x,y)$ satisfies the functional equation
 \beq\label{eqQK}
   xy \tK(x,y) \tQ(x,y) =xy- tx \tQ(x,0) -ty \tQ(0,y),
 \eeq
 and the $x/y$-symmetry entails that $\tQ(x,0)=\tQ(0,x)$. Observe that
  \beq\label{dec:xy-k}
 xy =   \frac 1 t  - \bx -\by- \frac 1 t \tK(x,y).
 \eeq
   Then the second pair of invariants $(I_1(x), J_1(y))$ is derived from~\eqref{eqQK} and~\eqref{dec:xy-k}:
 \[
   I_1(x)=
   tx \tQ(x,0)+ \bx -\frac 1{2t}, \qquad
   J_1(y)=  -I_1(y)= -ty\tQ(0,y) -\by+\frac 1{2t} .
 \]
 Indeed,
\[ 
  I_1(x)-J_1(y)
  = -\tK(x,y) \left(xy\tQ(x,y)+\frac 1t\right). 
\]

 More generally, it is proved in~\cite{BeBMRa-17} that among the nine models $\tS$ of Table~\ref{tab:sym}, exactly the first five (those associated with a finite group) admit a pair of rational invariants $(I_0(x), J_0(y))$. We make them explicit in Table~\ref{tab:inv-qu-1}. On the other hand, the above construction of the $\tQ$-related pair $(I_1(x), J_1(y))$ generalizes to all nine models, upon noticing that for each of them, the monomial $xy$ \emm decouples, explicitly as the sum of a function of $x$ and a function of $y$ --- modulo~$\tK$:
 \beq\label{dec:xy}
  xy =   f(x)+g(y)+ h(x,y)\tK(x,y),
 \eeq
 for rational series $f(x), g(y)$ and $h(x,y)$ having poles of bounded order at $0$. A possible choice of  $f(x)$ and $g(y)$, borrowed from~\cite{BeBMRa-17}, is given in Tables~\ref{tab:inv-qu-1} (for finite groups) and~\ref{tab:inv-qu-2} (for infinite groups). The decoupling identity~\eqref{dec:xy} allows us to rewrite the quadrant equation~\eqref{eqfunc-gen-qu} as
 \[
   \tK(x,y)  \left(h(x,y) -xy \tQ(x,y)\right)  = I_1(x)-J_1(y),
 \]
 where
 \beq\label{I1J1-def}
   I_1(x)=tx\tH_-(x)\tQ(x,0)-f(x), \qquad J_1(y)=-ty\tV_-(y) \tQ(0,y)+t \tQ_{0,0}\mathbbm 1_{(-1,-1)\in \tS} +g(y).
 \eeq
 Since $h(x,y) -xy \tQ(x,y)$ has poles of bounded order at $0$, the pair $(I_1(x), J_1(y))$ is indeed a pair of invariants. This pair is used in~\cite{BeBMRa-17} to prove that $\tQ(x,y;t)$ is algebraic (for finite group models) and D-algebraic (for infinite group models), for all models $\tS$ in Table~\ref{tab:sym}.
 The same approach applies to  five other models $\tS$  that are not relevant for this paper~\cite{BeBMRa-17}.
 
 \newcommand\Tstrut{\rule{0pt}{2.6ex}}       
\newcommand\Bstrut{\rule[-2ex]{0pt}{2pt}} 

  \begin{table}[ht] 
    \centering
    \begin{tabular}{|c||ccc:ca|}
      \hline
\rule{0pt}{7ex} $\cS$&  \begin{tikzpicture}[scale=.45] 
    \draw[->] (0,0) -- (0,-1);
    \draw[->] (0,0) -- (1,1);
    \draw[->] (0,0) -- (-1,0);
  \end{tikzpicture}     &
 \begin{tikzpicture}[scale=.45] 
    \draw[->] (0,0) -- (0,1);
    \draw[->] (0,0) -- (-1,-1);
    \draw[->] (0,0) -- (1,0);
  \end{tikzpicture}&
    \begin{tikzpicture}[scale=.45] 
      \draw[->] (0,0) -- (0,-1);
    \draw[->] (0,0) -- (1,1);
    \draw[->] (0,0) -- (-1,0);
    \draw[->] (0,0) -- (0,1);
    \draw[->] (0,0) -- (-1,-1);
    \draw[->] (0,0) -- (1,0);
  \end{tikzpicture}
&
    \begin{tikzpicture}[scale=.45] 
      \draw[->] (0,0) -- (0,-1);
      \draw[->] (0,0) -- (-1,0);
    \draw[->] (0,0) -- (0,1);
      \draw[->] (0,0) -- (1,0);
  \end{tikzpicture} 
&  \begin{tikzpicture}[scale=.45] 
      \draw[->] (0,0) -- (1,-1);
      \draw[->] (0,0) -- (-1,1);
    \draw[->] (0,0) -- (1,1);
      \draw[->] (0,0) -- (-1,-1);
  \end{tikzpicture}\\
$\tS$ & \begin{tikzpicture}[scale=.45] 
    \draw[->] (0,0) -- (0,1);
    \draw[->] (0,0) -- (-1,-1);
    \draw[->] (0,0) -- (1,0);
  \end{tikzpicture}
 &  \begin{tikzpicture}[scale=.45] 
    \draw[->] (0,0) -- (0,-1);
    \draw[->] (0,0) -- (1,1);
    \draw[->] (0,0) -- (-1,0);
  \end{tikzpicture}
&\begin{tikzpicture}[scale=.45] 
      \draw[->] (0,0) -- (0,-1);
    \draw[->] (0,0) -- (1,1);
    \draw[->] (0,0) -- (-1,0);
    \draw[->] (0,0) -- (0,1);
    \draw[->] (0,0) -- (-1,-1);
    \draw[->] (0,0) -- (1,0);
  \end{tikzpicture}
&   
\begin{tikzpicture}[scale=.45] 
       \draw[->] (0,0) -- (1,1);
    \draw[->] (0,0) -- (-1,0);
       \draw[->] (0,0) -- (-1,-1);
    \draw[->] (0,0) -- (1,0);
  \end{tikzpicture}
& \begin{tikzpicture}[scale=.45] 
       \draw[->] (0,0) -- (1,1);
    \draw[->] (0,0) -- (0,-1);
       \draw[->] (0,0) -- (-1,-1);
    \draw[->] (0,0) -- (0,1);
  \end{tikzpicture}
  \\
      \hline

      $I_0(x)$ &$\bx+x/t-x^2$ &$\bx^2-\bx/t -x$&$\bx -x -\frac{1+2t}{t(1+x)}$& $x+\bx-t(\bx-x)^2$&$\frac x{t(1+x)^2}+\frac{t(1+x)^2}x$
                                   \\
      $J_0(y)$ &   $\by+y/t-y^2$  &$\by^2-\by/t -y$&$\by -y -\frac{1+2t}{t(1+y)}$&$\frac y{t(1+y)^2}+\frac{t(1+y)^2}y$ \Bstrut
 & $y+\by-t(\by-y)^2$    \\
      \hline
      $f(x) $ &  $\frac x t -x^2$&$\frac 1 {2t}  - \bx$& $ \frac{x-t-tx^2}{t(1+x)}$& $-\bx$ 
 &   $ \frac x {t(1+x)}$  \\
      $g(y)$ & $-\by $&$\frac 1 {2t}  - \by$ & $-\by $ & $ \frac y {t(1+y)}$
                                                         \Bstrut
& $-\by$ \\ 
       \hline
    \end{tabular}
\vskip 3mm    \caption{Rational  $\tS$-invariants $(I_0,J_0)$, and $\tS$-decoupling of $xy$ by $f$ and $g$ for finite group models. The associated invariants $I_1(x), J_1(y)$ defined by~\eqref{I1J1-def} are algebraic.}
    \label{tab:inv-qu-1}
  \end{table}

   \begin{table}[ht]
    \centering
    \begin{tabular}{|c||c:ccc|}
      \hline
 \rule{0pt}{7ex}$\cS$
&
    \begin{tikzpicture}[scale=.45] 
      \draw[->] (0,0) -- (0,-1);
    \draw[->] (0,0) -- (1,1);
    \draw[->] (0,0) -- (-1,0);
    \draw[->] (0,0) -- (0,1);
       \draw[->] (0,0) -- (1,0);
     \end{tikzpicture}
&   \begin{tikzpicture}[scale=.45] %
      \draw[->] (0,0) -- (0,-1);
    \draw[->] (0,0) -- (1,1);
    \draw[->] (0,0) -- (-1,0);
      \draw[->] (0,0) -- (-1,-1);
  \end{tikzpicture}
&
    \begin{tikzpicture}[scale=.45] %
        \draw[->] (0,0) -- (1,1);
       \draw[->] (0,0) -- (0,1);
    \draw[->] (0,0) -- (-1,-1);
    \draw[->] (0,0) -- (1,0);
  \end{tikzpicture}
&
    \begin{tikzpicture}[scale=.45] %
      \draw[->] (0,0) -- (0,-1);
       \draw[->] (0,0) -- (-1,0);
    \draw[->] (0,0) -- (0,1);
    \draw[->] (0,0) -- (-1,-1);
    \draw[->] (0,0) -- (1,0);
  \end{tikzpicture}
\\
      $\tS$
      &\begin{tikzpicture}[scale=.45] 
    \draw[->] (0,0) -- (1,1);
    \draw[->] (0,0) -- (-1,0);
    \draw[->] (0,0) -- (0,1);
    \draw[->] (0,0) -- (-1,-1);
    \draw[->] (0,0) -- (1,0);
  \end{tikzpicture}
&  
\begin{tikzpicture}[scale=.45] 
      \draw[->] (0,0) -- (0,-1);
    \draw[->] (0,0) -- (0,1);
    \draw[->] (0,0) -- (-1,-1);
    \draw[->] (0,0) -- (1,0);
  \end{tikzpicture}
&   
\begin{tikzpicture}[scale=.45] 
      \draw[->] (0,0) -- (0,-1);
    \draw[->] (0,0) -- (1,1);
    \draw[->] (0,0) -- (-1,0);
    \draw[->] (0,0) -- (0,1);
  \end{tikzpicture}
&\begin{tikzpicture}[scale=.45] 
      \draw[->] (0,0) -- (0,-1);
    \draw[->] (0,0) -- (1,1);
    \draw[->] (0,0) -- (-1,0);
     \draw[->] (0,0) -- (-1,-1);
    \draw[->] (0,0) -- (1,0);
  \end{tikzpicture}
      \\
      \hline
      $f(x) $ &$-\bx$ & $-\bx$ & $-\bx + \frac 1 t $ & $-\bx$ 
      \\
      $g(y)$ & $ \frac{y(1-ty)}{t(1+y)}$ & $-1 + \frac y t -y^2 $ & $-\by -y$ & $\frac{y-t}{t(1+y)}$ \Bstrut
\\ 
       \hline
    \end{tabular}
\vskip 3mm    \caption{$\tS$-decoupling of $xy$ by $f$ and $g$ for infinite group models. The associated invariants $I_1(x), J_1(y)$ defined by~\eqref{I1J1-def} are D-algebraic.}
    \label{tab:inv-qu-2}
  \end{table}

 \section{Decoupling and invariants for three-quadrant walks}\label{sec:dec}

 \subsection{A new type of decoupling}
 For the nine models $\tS$ of Table~\ref{tab:sym},  we have exhibited in the previous section invariants $(I_1(x),J_1(y))$ involving the \gf\ $\tQ(x,y;t)$ for quadrant walks, by combining  the functional equation~\eqref{eqfunc-gen-qu} with the  decoupling of $xy$ modulo $\tK$, given by~\eqref{dec:xy}.

 Consider now the three-quadrant equation~\eqref{eq-U}. One main difference with~\eqref{eqfunc-gen-qu} is that the coefficient of $yD(y)$, being $\left(t  \tV_0(y)+ 2tx \tV_+(y) -1\right)$, involves both $x$ and~$y$.
  Our aim will be to write the right-hand side of~\eqref{eq-U} as the sum of a function of $x$ and a function of $y$ \emm multiplied by  $\left(t  \tV_0(y)+ 2tx \tV_+(y) -1\right)$,, modulo the kernel $\tK(x,y)$. Clearly, the only difficulty is to write the constant term $y$ in this form. This turns out to be possible in exactly four of the nine cases, and this property is in fact related to ``classical'' decoupling of $xy$ modulo  the original model~$\cS$ (not $\tS$!).

 \begin{prop}\label{prop:dec-three-quadrants}
   Let  $\cS$ be a one of the nine models of Table~\ref{tab:sym}, and  $\tS$ its companion model, with associated kernels $K(x,y)$ and $\tK(x,y)$, respectively. Define $\tV_-$, $\tV_0$ and $\tV_+$ by~\eqref{HV-def-alt}. The following conditions are equivalent:
   \begin{enumerate}
   \item there exist rational functions $F(x)\in \qs(x,t)$ and $G(y)\in \qs(y,t)$ such that the numerator of
     \beq\label{expr-rat}
       y- \left(t  \tV_0(y)+ 2tx \tV_+(y) -1\right) G(y) - F(x)
     \eeq
     contains a factor   $xy\tK(x,y)$,
    \item there exists rational functions ${\sf f(x)}\in \qs(x,t)$ and ${\sf g}(y)\in \qs(y,t)$ such that the numerator of 
  \beq\label{expr-rat-usual}  
       xy- {\sf f}(x) -{\sf g}(y) 
     \eeq
  contains a factor $xyK(x,y)$.   
\end{enumerate}

The only models of Table~\ref{tab:sym} for which these properties hold  are  the three models of the Kreweras trilogy, and  the $6$th model.
A solution is then given by
\beq\label{F-expr}
  F(x)= \bx^2+ \frac{K(\bx,\bx)}{t H_-(\bx)}, \qquad tG(y)=
  \frac{1+H_-(\by)}{y H_+(\by)} -1,
  \eeq
  and
  \beq\label{fg-dec-classic-expr}
  {\sf f}(x)= {\sf g}(x)=\frac 1 2 \left( x^2+ \frac{K(x,x)}{t H_-(x)}\right).
\eeq
  For this choice, stronger properties hold, as the rational functions~\eqref{expr-rat} and~\eqref{expr-rat-usual} are
  divisible by $\tK(x,y)$ and $K(x,y)$, respectively, in the sense of Definition~\ref{def:divisible}.
\end{prop}
   The explicit values of the rational functions $F(x)$ and $ G(y)$ are given in Table~\ref{tab:decoupling} for further reference.

  \begin{remark}
 It follows from Lemma~\ref{lem:divisibility} that if a rational function is divisible by  $K(x,y)$,  its numerator contains a  factor $xyK(x,y)$. But the converse is wrong, as shown by $xyK(x,y)/(x-t)$. Hence the term  ``stronger property'' used in the proposition. \qee
  \end{remark}

 \begin{table}[htb]
   \makebox[\textwidth][c]{
     \begin{tabular}{|c||c|c|c||c|}
    \hline 
\rule{0pt}{6ex}   $\cS$ &
 \begin{tikzpicture}[scale=.4] 
    \draw[->] (0,0) -- (1,1);
    \draw[->] (0,0) -- (-1,0);
    \draw[->] (0,0) -- (0,-1);
  \end{tikzpicture}
&\begin{tikzpicture}[scale=.4] 
    \draw[->] (0,0) -- (-1,-1);
    \draw[->] (0,0) -- (1,0);
    \draw[->] (0,0) -- (0,1);
  \end{tikzpicture}
 &\begin{tikzpicture}[scale=.4] %
    \draw[->] (0,0) -- (-1,-1);
    \draw[->] (0,0) -- (1,1);
    \draw[->] (0,0) -- (-1,0);
    \draw[->] (0,0) -- (1,0);
    \draw[->] (0,0) -- (0,-1);
    \draw[->] (0,0) -- (0,1);
  \end{tikzpicture}
 &   \begin{tikzpicture}[scale=.45] 
      \draw[->] (0,0) -- (0,-1);
    \draw[->] (0,0) -- (1,1);
    \draw[->] (0,0) -- (-1,0);
    \draw[->] (0,0) -- (0,1);
       \draw[->] (0,0) -- (1,0);
     \end{tikzpicture}
  \\
   $\tS$ &
   \begin{tikzpicture}[scale=.4] 
    \draw[->] (0,0) -- (-1,-1);
    \draw[->] (0,0) -- (1,0);
    \draw[->] (0,0) -- (0,1);
  \end{tikzpicture}
   &
      \begin{tikzpicture}[scale=.4] 
    \draw[->] (0,0) -- (1,1);
    \draw[->] (0,0) -- (-1,0);
    \draw[->] (0,0) -- (0,-1);
  \end{tikzpicture}
   &\begin{tikzpicture}[scale=.4] %
    \draw[->] (0,0) -- (-1,-1);
    \draw[->] (0,0) -- (1,1);
    \draw[->] (0,0) -- (-1,0);
    \draw[->] (0,0) -- (1,0);
    \draw[->] (0,0) -- (0,-1);
    \draw[->] (0,0) -- (0,1);
  \end{tikzpicture}
   & \begin{tikzpicture}[scale=.45] 
    \draw[->] (0,0) -- (1,1);
    \draw[->] (0,0) -- (-1,0);
    \draw[->] (0,0) -- (0,1);
    \draw[->] (0,0) -- (-1,-1);
    \draw[->] (0,0) -- (1,0);
  \end{tikzpicture}
  \\
    \hline    
    $t  \tV_0(y)+ 2tx \tV_+(y) -1$ &$ty+2tx-1$& 
                                                $ t\by+2txy-1$&$ t(\by+y) +2tx(1+y)-1$ &$ty+2xt(y+1)-1$ 
    \\ \hline
    $F(x)$ & $1/t-2x$ & $-\bx^2+\bx/t-x$ &$ \frac{1+2t}{t(1+x)}-1-x-\bx $ & $\frac 1t - 2x-2\bx $   \Bstrut\\
    $G(y)$ &$1/t$ &$\by/t$& $ \frac 1 {t(1+y)}$ &
                                                  $ \frac {1-y}{ t(1+y)}$
    \\
    $H(x,y)$ &$0$ &$\bx\by(x-y)/t$ & $\frac{x-y}{t(1+x)(1+y)}$
                                     & $-\frac{2y}{t(1+y)}$ \Bstrut
    \\ \hline
  \end{tabular}}
  \vskip 2mm
  \caption{Decoupling of $y$ in the form $y=(t  \tV_0(y)+ 2tx \tV_+(y) -1)G(y)+F(x)+\tK(x,y) H(x,y)$ for four symmetric models in the three-quadrant cone.}
   \label{tab:decoupling}
 \end{table}

\begin{proof}[Proof of Proposition~\ref{prop:dec-three-quadrants}]
  We first consider the eight models
  with no step $\nwarrow$ nor $\searrow$.  Assume that
   the numerator of~\eqref{expr-rat} contains a factor $xy\tK(x,y)$.  Since $K(x,y)= \tK(\bx, xy)$, it is equivalent to say that the numerator of 
  \[
    xy-\left(t \tV_0(xy)+2t\bx \tV_+(xy)-1\right) G(xy)-F(\bx)
  \]
  contains a factor $xyK(x,y)$. Since  $K(x,y)=K(y,x)$, the same holds for the numerator of
  \[
    xy-\left(t \tV_0(xy)+2t\by \tV_+(xy)-1\right) G(xy)-F(\by),
  \]
  and, by summing the two previous rational functions, for the numerator of
 \[
    2xy-2 \left(t \tV_0(xy)+t(\bx+\by) \tV_+(xy)-1\right) G(xy)-F(\bx)-F(\by).
  \]
  Since $S(x,y)=S(y,x)$, it follows from the expression~\eqref{V-def} of $S(x,y)$ that $\by \tV_+(xy)=x \tV_-(xy)$. Hence the coefficient of $G(xy)$ in the above function is in fact $2 K(x,y)$. Since the denominator of $G(xy)$ cannot contain a factor $xyK(x,y)$ (all steps of $\cS$ would be on the diagonal), we conclude that the numerator of $2xy-F(\bx)-F(\by)$ contains a factor $xyK(x,y)$, which means that Condition (b) holds with ${\sf f}(x)={\sf g}(x)= F( \bx)/2$.


    It is straightforward to adapt the above argument to the diagonal case, where now $S(x,y)=\tS(\bx^2,xy)= x^2\tV_-(xy)+\tV_0(xy)+\bx^2 \tV_+(xy)$, with $\by^2\tV_+(xy)=x^2\tV_-(xy)$. One then finds a solution to Problem (b) with ${\sf f}(x)={\sf g}(x)=F(\bx^2)/2$.

  \medskip
  Conversely, assume that Condition (b) holds. As proved in~\cite[Sec.~4.2]{BeBMRa-17}, then $\cS$ is either one of three models of the Kreweras trilogy, or the $6$th model of Table~\ref{tab:sym}. Moreover, Condition~(b) then holds for a pair $({\sf f}, {\sf g})$ such that ${\sf f}={\sf g}$: it suffices to take for $
  {\sf f}$ the half-sum of the two functions $F$ and $G$ of
  \cite[Tables~4-5]{BeBMRa-17}, and this yields~\eqref{fg-dec-classic-expr}. We leave it to the reader to check
  that the functions $F(x)$ and $G(y)$ defined by~\eqref{F-expr}, and
  listed in Table~\ref{tab:decoupling}, then satisfy Condition~(a) of
  the lemma. Finally, since  ${\sf f}$, ${\sf g}$, $F$ and $G$ are Laurent polynomials in $t$, the same holds when dividing~\eqref{expr-rat}   by $\tK(x,y)$ or~\eqref{expr-rat-usual}   by $K(x,y)$, and the resulting ratios thus have poles of bounded order at zero. This establishes the claimed divisibility properties.
\end{proof}

 \subsection{A new pair of invariants}
We now restrict our attention to the four models of Table~\ref{tab:decoupling}, for which the conditions of Proposition~\ref{prop:dec-three-quadrants} hold. 
We return to the functional equation~\eqref{eq-U} that relates $U(x,y)$ and $D(y)$. We rewrite the first term $y$ in the right-hand side using Proposition~\ref{prop:dec-three-quadrants}, and obtain:
 \[
   \tK(x,y) \left(2xy U(x,y)-H(x,y)\right) = \big(t  \tV_0(y)+ 2tx \tV_+(y) -1\big) S(y)-R(x),
 \]
 with
 \[
   H(x,y)= \frac{y-(t  \tV_0(y)+ 2tx \tV_+(y) -1)G(y)-F(x)}{\tK(x,y)}
 \]
 and
 \beq\label{RS-def}
 S(y)=yD(y)+G(y), \qquad
 R(x)= 2tx\tH_-(x) U(x,0) -F(x) +tD_0 \mathbbm 1_{(0,-1)\in \tS}.
 \eeq
 The functions $F(x)$, $G(y)$, and $H(x,y)$ are those of Table~\ref{tab:decoupling}. Upon multiplying the above identity by
 \[
   \left(t  \tV_0(y)+ 2tx \tV_+(y) -1\right) S(y)+R(x),
 \]
 and using Lemma~\ref{lem:square}, we exhibit a new pair of invariants.

 \begin{prop} \label{prop:inv-tq}
   With the above notation, the pair $(I(x), J(y)):=(R(x)^2, \Delta(y) S(y)^2)$ is a pair of invariants for the step set $\tS$, in the sense of Definition~\ref {def:invariants}. More precisely,
   \beq
   \label{IJ-id}
     \frac{     I(x)-J(y)}{ \tK(x,y) }
     = -4tx \tV_+(y)S(y)^2
 +    \left( H(x,y) -2xy U(x,y) \right)
 \big(  \left(t  \tV_0(y)+ 2tx \tV_+(y) -1\right) S(y)+R(x)\big),
\eeq 
where we recall that $\Delta(y)=\left(1-t\tV_0(y)\right)^2-4t^2\tV_-(y) \tV_+(y)$.
\end{prop}
 \begin{proof}
The identity is straightforward, and  we only need to check that all series that occur on the right-hand side have poles of bounded order at $0$. First, this holds for the series    $F$, $G$ and~$H$; see Proposition~\ref{prop:dec-three-quadrants} or Table~\ref{tab:decoupling}. Then $D(y)$ belongs to $\qs[y][[t]]$, $U(x,y)$ to $\qs[x,y][[t]]$, and $\tH_-(x)$ is a Laurent polynomial in $x$. Thus the series $R(x)$ and $S(y)$ defined by~\eqref{RS-def} have poles of bounded order at $0$. Since $\tV_0(y)$ and $\tV_+(y)$ are Laurent polynomials in $y$, we conclude that the right-hand side of~\eqref{IJ-id} has poles of bounded order at $0$.
 \end{proof}

 \bigskip
 \begin{remark}
 There is a useful alternative expression of $R(x)$, and thus of $I(x)=R(x)^2$. Let us return to the basic functional equation~\eqref{eqfunc-gen} for $C(x,y)$, written at $(\bx, \bx)$, and observe that $C_-(x)=xU(x,0)$ and $C_{0,0}=D_0$ (see~\eqref{C-split}). This gives:
 \begin{align*}
   K(\bx,\bx)C(\bx,\bx)&=1-2tx^2 H_-(\bx)U(x,0)-tx^2 D_0  \mathbbm 1_{(-1,-1)\in \cS}\\
                       &= 1- x^2\left(R(x) + F(x)\right)  \hskip 40mm \hbox{by~\eqref{RS-def} and~\eqref{rewrite}}\\
                       &=  -x^2\left(R(x) + \frac{K(\bx,\bx)}{t H_-(\bx)}\right) \hskip 38mm \hbox{by~\eqref{F-expr}}.
 \end{align*}
We conclude that
 \[
   R(x)=- K(\bx, \bx) \left(\bx^2 C(\bx, \bx)+ \frac 1 {tH_-(\bx)}\right).
 \]
 In particular, for $x=1$ we obtain:
 \beq\label{I1-C11}
 R(1)=
 -(1-|\cS|t) \left( C(1,1) + \frac 1  {tH_-(1)}\right).
 \eeq
 Since  we are going to provide the value of $I(x)=R(x)^2$  explicitly, we will obtain  an expression for the square of the above series. This is reflected in the characterization of $C(1,1)$ already given for  Kreweras' model; see~\eqref{C11-K}. Further illustrations are~\eqref{C11-RK}, \eqref{C11-DK}, \eqref{C11-fork}. \qee
  \end{remark}

 \subsection{General strategy}
  Let us now describe the strategy that we are going to apply to each of the four models of Table~\ref{tab:decoupling}, in Sections~\ref{sec:K} to~\ref{sec:DA}. For each of these models, we have at hand two, or sometimes three, pairs of invariants for the step set $\tS$:
 \begin{itemize}
      \item  the pair $(I(x),J(y))$ from Proposition~\ref{prop:inv-tq}; it involves the series $U(x,0)=\bx C_-(x)$ and $D(y)$, which are keys in the determination of the \gf\ $C(x,y)$ of walks avoiding the negative quadrant;
 \item  a rational one, $(I_0(x), J_0(y))$, for the three models of the Kreweras trilogy  (Table~\ref{tab:inv-qu-1});
 \item for all four models, the pair   $(I_1(x), J_1(y))$ defined by~\eqref{I1J1-def}, where the functions $f$ and~$g$ are given in Tables~\ref{tab:inv-qu-1} and~\ref{tab:inv-qu-2}. This pair of invariants involves  series counting quadrant walks  with steps in $\tS$, and these series are known. They are algebraic for the Kreweras trilogy~\cite{BMM-10}, and D-algebraic for the last model~\cite{BeBMRa-17}.
\end{itemize}
Using Lemma~\ref{lem:linear}, we are going to form polynomial combinations of these pairs to construct a new pair of invariants involving $U(x,0)$ and $D(y)$ and satisfying the condition of Lemma~\ref{lem:invariants} --- and thus trivial. This will give us expressions of $I(x)$ and $J(y)$, hence $U(x,0)$ and
$D(y)$, in terms of some known quadrant series. For the Kreweras trilogy,  an alternative would be to construct trivial invariants based on $(I(x), J(y))$ and $(I_0(x),J_0(y))$ only, as is done in~\cite{BeBMRa-FPSAC-16,BeBMRa-17} to (re)derive quadrant \gfs. But exploiting the three pairs that we have at hand yields more direct derivations.
%

 \section{Kreweras steps}
 \label{sec:K}

 In this section  we  take $\cS= \{\nearrow, \leftarrow, \downarrow\}$, so that $\tS=\{\rightarrow, \uparrow, \swarrow\}$. We prove the results that were stated at the end of  the introduction (Section~\ref{sec:K-statements}): first the exact results in Section~\ref{sec:GF-K-proofs}, then the asymptotic ones in Section~\ref{sec:K-asympt-proofs}. We refer to the \Maple\ session available on the author's~\href{http://www.labri.fr/perso/bousquet/publis.html}{webpage} for details of the calculations.

 \subsection{Generating functions}
 \label{sec:GF-K-proofs}
 To begin with, we establish the expressions of $U(x,0)$ and $D(y)$.

\begin{proof}[Proof of Theorems~\ref{thm:K-U} and~\ref{thm:K-D}.]
 
 The first equation in Theorem~\ref{thm:K-U} is of course the basic functional equation~\eqref{eqfunc-gen}. The invariants of Proposition~\ref{prop:inv-tq} are:
 \beq\label{IJ-K}
   I(x)=\left( 2tU(x,0)+2x-\frac 1 t \right)^2, \qquad J(y)=\Delta(y)\left(yD(y)+ \frac 1 t \right)^2
 \eeq
with
 $
   \Delta(y)= (1-ty)^2-4t^2\by.
   $
The rational invariants $(I_0,J_0)$ are given in Table~\ref{tab:inv-qu-1} (but we will not use them, in fact).  The invariants related to quadrant walks with steps in $\tS$, defined by~\eqref{I1J1-def}, are
   \beq\label{I1-K}
   I_1(x)=t\tQ(x,0)-x/t+x^2, \qquad J_1(y) = -t\tQ(0,y)-\by +t\tQ_{0,0}.
 \eeq
 They satisfy
 \beq\label{ratio:K1}
     I_1(x)-J_1(y)=-\frac x t \tK(x,y) \left( 1+ty \tQ(x,y)\right).
\eeq
 We want to construct a pair $\left(\tilde I(x), \tilde J(y)\right)$ of invariants satisfying the condition of Lemma~\ref{lem:invariants}: the ratio $(\tilde I(x)-\tilde J(y))/\tK(x,y)$ should be a multiple of $xy$ (here we are abusing terminology, since the property that we actually require involves the coefficient of $t^n$ for each $n$). The ratio  $(I(x)-J(y))/\tK(x,y)$ is the right-hand side of~\eqref{IJ-id}. It is easily seen to be a Laurent series in $t$ with coefficients in $\qs[x, y]$, and in fact a multiple of $x$. It  equals $-4x/t$ at $y=0$.
 Similarly, the ratio $(I_1(x)-J_1(y))/\tK(x,y)$, derived from~\eqref{ratio:K1}, is a multiple of $x$
 and equals $-x/t$ at $y=0$. This leads us to introduce the following pair of invariants:
 \[
\left(\tilde I(x), \tilde J(y)\right):=   \big(I(x)-4I_1(x), J(y)-4J_1(y)\big).
 \]
 By Lemma~\ref{lem:invariants}, this pair of invariants is trivial, of the form $(A,A)$ where $A$ is a series in $t$. Returning to the explicit values~\eqref{IJ-K} and~\eqref{I1-K} of our invariants, this gives
 \beq\label{eqinv-K}
 \begin{aligned}
 I(x)-4I_1(x)&=  \left( 2tU(x,0)+2x-\frac 1 t \right)^2-4t\tQ(x,0)+4x/t-4x^2 =A,
   \\
   J(y)-4J_1(y)&= \Delta(y)\left(yD(y)+ \frac 1 t \right)^2 +4 t\tQ(0,y)+4\by -4t\tQ_{0,0} =A.
 \end{aligned}
 \eeq
 To complete the solution, it suffices to:
 \begin{itemize}
 \item inject the known algebraic expressions of $\tQ(x,0)$ and $\tQ_{0,0}$, taken for instance from Proposition~14 in~\cite{BMM-10} (recall that $\tQ(0,y)=\tQ(y,0)$). These expressions involve a series in $t$ denoted $W$ in~\cite{BMM-10}, which is the series $\PIK$ of Theorem~\ref{thm:K-U}. We refer to the \Maple\ session  for details;
 \item specialize the second identity obtained in this way, of the form
   \[
     \Delta(y)\left(yD(y)+ \frac 1 t \right)^2 -4J_1(y) =A,
   \]
(where $J_1(y)$ is now explicit),   at $y=\PIK^2$: this is the (only) root of $\Delta(y)$ lying in $\qs[[t]]$.  This  gives:
   \[
     A= -4 J_1(\PIK^2)=2\,{\frac { (1-\PIK^3)^{3/2}}{{\PIK}^{2}}}
     +{\frac {{\PIK}^{6}+12\,{\PIK}^{3}+8}{4{\PIK}^{2}}}
.
\]
   \end{itemize}
   These two ingredients yield, after elementary manipulations involving the equation $\PIK=t(2+\PIK^3)$, the expressions announced in Theorems~\ref{thm:K-U} and~\ref{thm:K-D}, which are those of $I(x)/2$ and $J(x)/2$  (recall that $U(x,0)=\bx C_-(x)$).

   Since $\PIK$ has degree $3$ over $\qs(t)$,  it follows from these expressions that $C_-(x)$ and $D(x)$ have degree at most $24$. This bound is proved to be tight  by computing explicit minimal polynomials, at $x=2$ for instance (this is lighter than keeping the indeterminate $x$). Now the identities of Theorem~\ref{thm:K-U} shows that $C(x,y)$ will have degree $24\times 4=96$ at most, and at any rate the same degree as $xC_-(\bx)+yC_-(\by)=U(\bx,0)+U(\by,0)$. So we only need to determine the degree of $U(x,0)+U(y,0)$, which we do at $x=2$ and $y=3$ (by successive eliminations as before). We find it to be $96$ indeed, which concludes the proof of Theorems~\ref{thm:K-U} and~\ref{thm:K-D}.
    \end{proof}

    Now we prove the results that deal with walks ending at a prescribed position $(i,j)$.
    
\begin{proof}[Proof of Corollary~\ref{cor:excursions-K}]
We begin with walks ending on the diagonal, the \gf\ of which is given in Theorem~\ref{thm:K-D}. We consider the series 
\beq\label{diag-K-shift}
x D(x/\PIK) + \frac \PIK t 
\eeq
which, due to the periodicity of the model and the fact that $\PIK/t$ is a series in $t^3$, is also a series in~$t^3$ (or equivalently, in $\PIK^3$) with coefficients in $\qs[x]$. Using  Theorem~\ref{thm:K-D} and the identity $t=\PIK/(2+\PIK^3)$, we can write:
\begin{align*}
  \left( x D(x/\PIK) + \frac \PIK t \right) ^2 &= \PIK^2 \frac{J( x /\PIK)}{\Delta( x /\PIK)}
  \\
 & =\frac {2(2+\PIK^3)^2}{(x-\PIK^3)(4-4x-x\PIK^3+x^2)} 
   \\&\quad \times  \left( x(1-\PIK^3)^{3/2} +(1-x)^2 (x- \PIK^3) - ( \PIK^3+x\PIK^3-2x^2) \sqrt{1- \frac{4+\PIK^3} 4 x +\frac{x^2}4}\right).
 \end{align*}
 The right-hand side appears as a formal power series in $x$, with constant term $(2+\PIK^3)^2$ and coefficients in $\qs[  \sqrt{1-\PIK^3},1/\PIK^3]$. Upon taking square roots,
 we conclude that the series~\eqref{diag-K-shift} is also a series in $x$ with coefficients in $\qs[\sqrt{1-\PIK^3},1/\PIK^3]$. For $i\ge 0$, the coefficient of $x^{i+1}$ in this series is  $C_{i,i}/\PIK^{i}$.  We have thus proved that
 \beq\label{Cdiag-i}
 D_i=C_{i,i}\in \PIK^{i} \qs\left[  \sqrt{1-\PIK^3},1/\PIK^3\right], \quad \text{so that}\quad 
 C_{i,i}\in t^i \qs\left[  \sqrt{1-\PIK^3},1/\PIK^3\right].
 \eeq
 Since  $\PIK=t(2+\PIK^3)$, and  $\sqrt{1-\PIK^3}=1-2\PIIK$ by definition of $\PIIK$,   the part of the corollary dealing with walks that end on the diagonal follows. We still retain the above more precise statement for further use. We obtain in particular the values of $C_{0,0}$ and $C_{1,1}$ given below Corollary~\ref{cor:excursions-K}.  To obtain the claimed result for walks ending just above of below the diagonal ($j-i=\pm1$, with $i,j\ge 0$), it suffices to note that $C_{i,i}-tC_{i-1,i-1}=2t C_{i,i+1} + \mathbbm1_{i=0}$ due to the choice of steps. In particular, we thus obtain the value of $C_{0,1}$ given below Corollary~\ref{cor:excursions-K}.

We go on with walks ending on the negative $x$-axis, the \gf\ of which is given in Theorem~\ref{thm:K-U}. By periodicity, the series
\beq\label{x-axis-K-shift}
\PIK \left(2t U(x/\PIK)+2x/\PIK-\frac 1 t\right) =
 \PIK \left(2t \bx \PIK C_-(x/\PIK)+2x/\PIK-\frac 1 t\right) 
\eeq
is a series in $t^3$ with coefficients in $\qs[x]$. Using  Theorem~\ref{thm:K-U}, and the definition of $\PIIK$, we first rewrite its square in terms of $x$ and  $\PIIK$:
\begin{multline*}
 \PIK^2 \left(2t \bx \PIK C_-(x/\PIK)+2x/\PIK-\frac 1 t\right) ^2=  \PIK^2 I\left(\frac x \PIK\right)
   \\=2\bx   \left( x(1-\PIK^3)^{3/2} +(1-x)^2 (x- \PIK^3) + ( \PIK^3+x\PIK^3-2x^2) \sqrt{1- \frac{4+\PIK^3} 4 x +\frac{x^2}4}\right).
 \end{multline*}
 The right-hand side is a \fps\ in $x$ with constant term $2(1-\PIK^3)^{3/2}+2+5\PIK^3-\PIK^6/4=4(1-\PIIK^2)(1+\PIIK)^2$, while the coefficients of higher powers of $x$ are polynomials in $\PIK^3$. Given that $\PIK^3=4\PIIK(1-\PIIK)$, we conclude that~\eqref{x-axis-K-shift} has its coefficients in $\sqrt{1-\PIIK^2}\, \qs[\PIIK, 1/(1-\PIIK^2)]$.
 For $i\ge 1$, the coefficient of  $x^{i-1}$ in~\eqref{x-axis-K-shift} is
\[
  2t\PIK^2 C_{-i,0}/\PIK^{i} -\frac{\PIK} t \mathbbm 1_{i=1} + 2 \mathbbm1_{i=2}.
\]
Hence, using once again $\PIK=t (2+\PIK^3)$, we find that
\beq\label{Cmi0}
 C_{-i,0} \in t^i \qs\left[\sqrt{1-\PIIK^2}, \PIIK, 1/\PIK^3, 1/(1-\PIIK^2)\right].
  \eeq
  We now introduce the series $\PIIIK$, which satisfies  $\sqrt{1- \PIIK^2}= (1-\PIIIK^2)/(1+\PIIIK^2)$. The corollary now follows for walks that end on negative $x$-axis.
  We obtain in particular the expression~\eqref{Cm10-K} of $C_{-1,0}$.
  
  Let us finally prove the corollary for any point $(i,j)$ in the three quadrant cone. Without loss of generality, we assume $j\ge i$
  and argue by induction on $j\ge -1$. 
   If $j=-1$, then $i<0$ and the result is obvious because $C_{i,j}=0$.  The corollary has also been proved for $j=0$, $j=i$ and $j=i+1$.  Now assume that $j\ge 1$ and $j\ge i+2$. We clearly have
\beq\label{id-rec-K}
  C_{i,j-1} = t \left( C_{i-1, j-2} + C_{i+1,j-1} + C_{i,j} \right),
\eeq
and the result holds for the three series $C_{i,j-1}, C_{i-1, j-2} $ and $ C_{i+1,j-1} $ by the induction hypothesis. Thus it holds for $C_{i,j}$ as well.
\end{proof}

\subsection{Asymptotic results}
\label{sec:K-asympt-proofs}
We go on  with the determination of the harmonic function.

\begin{proof}[Proof of Corollary~\ref{cor:harmonic-K}]
  Let us first assume that we have established an estimate of the form~\eqref{cij-est} for $i=j\ge 0$ and for $i<0, j=0$.
  For $(i,j)\in \Cc$, fix a walk $w$ from $(0,0)$ to $(i,j)$ in $\cC$,
  and let $k$ be its length, with $k+i+j\equiv 0$ mod $3$. Then
  $c_{i,j}(3m+k)\ge c_{0,0}(3m)$ (because one can concatenate~$w$ with a walk in $\cC$ starting and ending at $(0,0)$, translated so that its starting point is $(i,j)$).  Given the estimation of $c_{0,0}(3m)$, the numbers $c_{i,j}(n)$ cannot be $o(3^n n^{-7/4})$.

Now for $i\ge 0$, 
  \[
    c_{i,i}(n+1)=c_{i-1,i-1}(n)+ 2c_{i,i+1} (n).
  \]
This identity implies that an estimate of the form~\eqref{cij-est} also holds  for $j=i+1 \ge 1$, with
  \beq\label{H-surdiag}
    3H_{i,i}=H_{i-1,i-1}+2H_{i,i+1}.
  \eeq
 Here  we use the fact  that  $ 3H_{i,i}-H_{i-1,i-1}\not =0$;  otherwise $c_{i,i+1} (n)$ would be $o(3^n n^{-7/4})$, which we have excluded.   Then the induction on $j$, with $j\ge i$, that we have just used in the proof of the previous corollary establishes~\eqref{cij-est} for all $i$ and $j$, based on the following version of~\eqref{id-rec-K}:
  \[
 c_{i,j-1}(n+1) = c_{i-1, j-2}(n) + c_{i+1,j-1} (n)+ c_{i,j} (n).
\]

The identity~\eqref{harmonic} saying that $H$ is harmonic similarly follows from
  \[
 c_{i,j}(n+1) = c_{i-1, j-1}(n) + c_{i+1,j} (n)+ c_{i,j+1} (n).
\]

Let us now define $\Hc(x,y)$, $\Hc_-(x)$ and $\Hc_d(y)$ in terms of the numbers $H_{i,j}$ as in Corollary~\ref{cor:harmonic-K}. Then, using the harmonicity of $H$, we have:
  \allowdisplaybreaks
  \begin{align*}
    \Hc(x,y)& =\frac 1 3 \sum_{j\ge 0, i\le j } \left( H_{i-1,j-1}+ H_{i+1,j} + H_{i,j+1}\right)x^{j-i} y^{j}
    \\             &= \frac 1 3 \sum_{j\ge -1, i\le j }  H_{i,j}x^{j-i} y^{j+1}+ \frac 1 3 \sum_{j\ge 0, i\le j+1 } H_{i,j}x^{j-i+1} y^{j} +  \frac 1 3 \sum_{j\ge 1, i\le j-1 }H_{i,j}x^{j-i-1} y^{j-1}
    \\                     &=\frac y 3 \Hc(x,y) +\frac x 3\left( \Hc(x,y) + \bx\sum_{j \ge 0} H_{j+1,j} y^j\right) +\frac {\bx \by}3\left( \Hc(x,y)- \sum_{i
                             >0} H_{-i,0} x^{i}- \sum_{i\ge 0} H_{i,i} y^i\right)
    \\ &= \frac 1 3 \left( (y+x+\bx\by)\Hc(x,y)+ \sum_{j \ge 0} H_{j+1,j} y^j- \bx\by \Hc_-(x)-\bx\by \Hc_d(y)\right)
         \\ &=\frac 1 3 \left( (y+x+\bx\by)\Hc(x,y)+ \frac {3-y} 2 \Hc_d(y)- \bx\by \Hc_-(x)-\bx\by \Hc_d(y)\right),
  \end{align*}
  where we have used~\eqref{H-surdiag} to express the final sum. This gives~\eqref{H-eq} upon regrouping terms.

  \smallskip

  At this stage, it remains to prove~\eqref{cij-est} for $i=j\ge 0$ and for $i<0, j=0$, and to establish the announced values of the corresponding \gfs\ $\Hc_d(y)$ and $\Hc_-(x)$. Let us begin with the numbers $c_{i,i}(n)$, for $n$ of the form $i+3m$. Recall the form of $C_{i,i}/t^i$ given by~\eqref{Cdiag-i}.
    An elementary singular analysis of the series $\PIK$
     (seen as a function of $t^3$) shows that it has a unique dominant singularity, located at $t^3=1/27$. It increases on the interval $[0,1/27]$, and we have  there the following singular expansions:
  \begin{align}
    \PIK ^3&= 1-\sqrt 3 \sqrt{1-27t^3} + \LandauO(1-27t^3),  \label{PIK3-sing} \\
    \sqrt{1-\PIK^3} &= 3^{1/4} (1-27t^3)^{1/4}+ \LandauO\left((1-27t^3)^{3/4}\right). \nonumber
  \end{align}
Moreover, the equation $\PIK=t(2+\PIK^3)$ shows that $\PIK$ cannot vanish on its disk of convergence. In sight of~\eqref{Cdiag-i}, $C_{i,i}/t^i$ is thus an algebraic series in $t^3$, with radius of convergence at least $1/27$, taking a finite value at this point, and we expect a singular behaviour in $(1-27t^3)^{\alpha}$ for some $\alpha\ge 1/4$. Note that there cannot be any singularity other than $1/27$ on the circle of radius $1/27$, because this is the only value of $t^3$ for which $\PIK$ equals $1$.

  We now  use Theorem~\ref{thm:K-D} to express the series $(yD(y/t)+1)^2$ (which is a series in $t^3$) in terms of $\PIK^3$. Plugging in this expression the above expansion of $\PIK^3$ then gives:
  \begin{align*}
    (yD(y/t)+1)^2  &= t^2 \frac{J(y/t)}{\Delta(y/t)}\\
    & =    2\frac{1-3y}{4-3y} +\frac{1+6y}{\sqrt{(1-3y)(4-3y)}}
 -   6\cdot{3}^{3/4} y\,{\frac {{(1-27t^3)}^{3/4}}{ \left(1- 3\,y\right) ^{2}
 \left(4- 3\,y \right) }} + \LandauO(1-27t^3).
  \end{align*}
We refer once again to the \Maple\ session for details. It follows that as $t^3$ approaches $1/27$,
\beq\label{Dyt-sing}
  D(y/t) = c_0(y) + c_1(y) (1-27t^3)^{3/4}+ \LandauO(1-27t^3),
\eeq
where $c_0(y)$ is an algebraic function of $y$ that we do not need to make explicit, and
\[
  c_1(y)= -\frac 1{3^{1/4}(1-3y)\sqrt{y(4-3y)}}\sqrt{\frac{1+6y}{1-3y} \sqrt{\frac{4-3y}{1-3y}}-2}.
\]
By extracting from~\eqref{Dyt-sing} the coefficient of $y^i$, we obtain
\[
  C_{i,i} = t^i [y^i] c_0(y) + t^i [y^i] c_1(y) (1-27t^3)^{3/4}+ \LandauO(1-27t^3),
\]
so that, for $n=i+3m$,
\begin{align*}
  c_{i,i}(n) &\sim \frac{ [y^i] c_1(y) } {\Gamma(-3/4)}\, 27^m {m^{-7/4}}\\
             &\sim \frac{ 3^{-i+7/4}[y^i] c_1(y) } {\Gamma(-3/4)}\, 3^n{n^{-7/4}}\\
  & \sim -\frac{H_{i,i}}{\Gamma(-3/4)}\, 3^n{n^{-7/4}}
\end{align*}
where
$
  H_{i,i}=- 3^{7/4} [y^i] c_1(y/3).
$
The announced expression of $\Hc_d(y)=- 3^{7/4} c_1(y/3)$ follows.

Note that the above argument would fail  if the coefficient of $y^i$ in $c_1(y)$ were zero. But in this case $c_{i,i}(n)$ would be $ o(3^n n^{-7/4})$, which we have excluded at the beginning of the proof.

\smallskip
We proceed similarly for walks ending on the negative $x$-axis. Equation~\eqref{Cmi0} gives the form of $C_{-i,0}/t^i$ in terms of the series $\PIK$ and $\PIIK$. As $\PIK^3$, the series $\PIIK$ has a unique dominant singularity at $t^3=1/27$. It increases on the interval  $[0, 1/27]$ and as $t^3$ approaches $1/27$,
\beq\label{PIIK-sing}
\begin{aligned}
    \PIIK&= \frac 1 2 - \frac{3^{1/4}}2 (1-27t^3)^{1/4} + \LandauO\left((1-27t^3)^{3/4}\right),\\
   1-\PIIK^2&=\frac 3 4 + \frac{3^{1/4}}2 (1-27t^3)^{1/4}+ \LandauO\left(\sqrt{1-27t^3}\right).
 \end{aligned}
 \eeq
Since $\PIIK$ has non-negative coefficients, and equals $1/2$ when $t^3=1/27$, it remains away from the values $\pm1$ on the disk $|t^3|<1/27$. Hence by~\eqref{Cmi0}, the series $C_{-i,0}/t^i$ has radius of convergence at least $1/27$, and no other singularity of modulus $1/27$.

Then we consider the series  $2C_-(x/t)+2x^2/t^3-x/t^3= x^2 I(x/t)/t^4$, which is a series in $t^3$. We express its square in terms of $\PIK^3$ using Theorem~\ref{thm:K-U}, and inject the singular expansion of $\PIK^3$. This gives
\begin{multline}\label{Imodsing}
    \left(2C_-(x/t)+2x^2/t^3-x/t^3\right)^2=
27x(1-3x)(1+6x)\sqrt{(1-3x)(4-3x)}\\  -54 x(1-3x)^3
+ 162\cdot {3}^{3/4}\,{x}^{2}(1-27t^3)^{3/4} + \LandauO(1-27t^3).
\end{multline}
We must be a bit careful with the sign when taking square roots. Indeed, both sides are $\LandauO(x^2)$, and the coefficient of $x^2$ in the right-hand side is
\[
  {\frac{2187}{4}}+ \LandauO((1-27t^3)^{3/4}),
  \]
  and this is the singular expansion of $(2C_{-1,0}/t-1/t^3)^2$ at $t^3=1/27$. But it follows from the expression~\eqref{Cm10-K} of $C_{-1,0}$, evaluated at $t=1/3$ (where $\PIIIK=2-\sqrt 3$), that at this point, $2C_{-1,0}/t-1/t^3= -27 \sqrt 3/2$, with a minus sign. Hence the square root of the right-hand side of~\eqref{Imodsing} is $-(2C_-(x/t)+2x^2/t^3-x/t^3)$, and  as $t^3$ approaches $1/27$, 
\[
  C_-(x/t) = \tilde c_0(x)+\tilde c_1(x) (1-27t^3)^{3/4} + \LandauO(1-27t^3),
\]
where 
\[
  \tilde c_1(x)=- \frac{3^{5/4} x}2 \sqrt{\frac{1+6x}{1-3x}\sqrt{\frac{4-3x}{1-3x}}+2}
  .\]
As we have just done for walks ending on the diagonal, we  conclude that for $n$ of the form $i+3m$ and $i>0$,
\[
  c_{-i,0}(n) \sim -\frac{H_{-i,0}}{\Gamma(-3/4)}\, 3^n{n^{-7/4}},
\]
where
$
  H_{-i,0}=- 3^{7/4} [x^i] \tilde c_1(x/3).
$
The announced expression of $\Hc_-(x)=- 3^{7/4} \tilde c_1(x/3)$ follows.
\end{proof}

\begin{remark}
  \label{rem:invariants-harmonic}
  \noindent{\bf An invariant approach for harmonic functions.} Let us now explain how we can predict the relation that we observed below Corollary~\ref{cor:harmonic-K} between the $\cS$-harmonic function $H_{i,j}$ in $\Cc$ and the $\tS$-harmonic function $h_{i,j}$ in $\Qc$.
  
Assume that an asymptotic estimate of the form~\eqref{cij-est} holds for the numbers $c_{i,j}(n)$.  Then the function $H_{i,j}$ is $\cS$-harmonic in $\cC$,  symmetric in the first diagonal, and the series $\Hc(x,y)$ defined by~\eqref{H-def} satisfies~\eqref{H-eq}. Denoting by $c(x,y)=\frac 1 2 \left(2+xy^2-3xy\right) $ the coefficient of $\Hc_d(y)$ in this equation, we observe that
\[
  c(x,y)^2= \frac {x^2\delta(y)} 4-3xy \tK(x,y;1/3) = \frac {x^2\delta(y)} 4 + (1+xy^2+x^2y-3xy),
\]
where $\delta(y)=y(y-4)(y-1)^2$ is the discriminant of $3xy\tK(x,y;1/3)$ with respect to $x$. This identity is of course closely related to Lemma~\ref{lem:square}.  After multiplying~\eqref{H-eq} by
\[
 \bx^2\left( \Hc_-(x) - c(x,y)\Hc_d(y)\right),
\]
we obtain
\beq\label{eq-H-barre}
  (1+xy^2+x^2y-3xy)\btH(x,y)=
  \bx^2\,\Hc_-(x)^2-  \frac {\delta(y)} 4\Hc_d(y)^2,
  \eeq
  where
   \beq\label{btH}
    \btH(x,y):=  \bx^2 \,\Hc ( x,y ) \left(\Hc_-(x) - c(x,y)\Hc_d(y) \right)  +\bx^2\, {\Hc_d}( y)   ^{2}.
  \eeq
Using the functional equation~\eqref{H-eq} satisfied by $\Hc(x,y)$,   we can  check that $\btH(x,y)$ is a \fps \ in $x$ and $y$, and that moreover,
\beq\label{H-Hbar}
  \btH(x,0)= \bx^2\,\Hc_-(x)^2 \qquad \text{and} \qquad 
  \btH(0,y)= -\frac {\delta(y)} 4\, \Hc_d(y)^2+H_{-1,0}^2.
  \eeq
  In particular, if we denote  $\btH(x,y)=\sum_{i,j\ge 0} \bh_{i,j} x^i y^j$, we have $\btH(0,0)=\bh_{0,0}=H_{-1,0}^2$. Equation~\eqref{eq-H-barre} thus reads
  \[
    (1+xy^2+x^2y-3xy)\btH(x,y)=\btH(x,0)+\btH(0,y)-\btH(0,0),
  \]
  which precisely means that the function $(\bh_{i,j})_{(i,j)\in \Qc}$
  is $\tS$-harmonic in the quadrant $\Qc$. 

Now  recall that  there exists a unique \emm positive, $\tS$-harmonic function $h_{i,j}$ in $\Qc$ (up to a multiplicative constant)~\cite{raschel-harmonic}. The associated \gf\ is $\tH(x,y)=\sum_{i,j\ge 0} h_{i,j} x^i y^j$, where $\tH(x,0)=\tH(0,x)$ is given by~\eqref{harmonic-RK-quadrant}. Hence, if we assume that the above series $\btH(x,y)$  has positive coefficients, then there exists a positive constant $\kappa$ such that $\btH(x,y)=\kappa^2 \tH(x,y)$. 
Returning to~\eqref{H-Hbar}, we thus predict in particular that
  \[
  \Hc_-(x) =\kappa x \sqrt{\tH(x,0)} \qquad \text{and} \qquad 
   \Hc_d(y)=2\kappa\,   \sqrt{-\frac { \tH(0,y)-\tH(0,0)}{ \delta(y)}},
 \]
 where the constant $\kappa$ must be
 \[
   \kappa = \frac{H_{-1,0}}{\sqrt{h_{0,0}}}=  \frac{H_{0,0}}{\sqrt{h_{0,1}}}.
 \]
Using the expression~\eqref{harmonic-RK-quadrant} of $\tH(x,0)=\tH(0,x)$, we can now check from the  expressions~\eqref{Hminus-K} and~\eqref{Hd-K} of $\Hc_-(x)$ and $\Hc_d(y)$ that this prediction indeed holds true,  with $\kappa=3$.
 \qee
\end{remark} 
    
    We finally complete this section with the enumeration of all walks avoiding the negative quadrant, regardless of their final position.

\begin{proof}[Proof of Corollary~\ref{cor:all-K}]
  We specialize the equations of Theorem~\ref{thm:K-U} at $x=y=1$. The first one gives the link between $C(1,1)$ and $C_-(1)$ shown on the first line of~\eqref{C11-K} (this is in fact $I(1)/2$), 
  and the second one yields the expression on the second line.
  The degree of $C(1,1)$ is found to be $24$ by elimination, first of the two square roots, and finally of $\PIK$.

    Now for the asymptotics, we need more details about $\PIK$ than what we have used so far, which only involved $\PIK^3$ (seen as a series in $t^3$). The series $\PIK$ has three dominant singularities, located at $\zeta/3$ where $\zeta$ is any cubic root of unity. The singular expansions of $\PIK$ at these points can be computed by combining $\PIK=t(2+\PIK^3)$ with the known expansion of $\PIK^3$ around $1/27$ (see~\eqref{PIK3-sing}). Denoting $\zeta_0=1$ and $\zeta_\pm= (-1\pm i\sqrt 3)/2$, we find that around any dominant singularity $\zeta/3$,
\beq\label{PIK-sing-refined}
      \frac \PIK\zeta =          1-     {(1-3t/\zeta)}^{1/2}+
                        (1-3 t/\zeta)/3-4/9
 {(1-3t/\zeta)}^{3/2}+\LandauO \left( {(1-3 t/\zeta)}^{2} \right).
\eeq
    Observe that in the expression of $(1-3t)^2 \left (C(1,1)+1/t\right)^2$, the square root term involving 
    \[
      1-\PIK\frac{4+\PIK^3}4 +\frac{ \PIK^2} 4 =\frac{(1-\PIK)(2+\PIK)(2-\PIK+\PIK^2)}4
    \]
    does not vanish for $|\PIK|<1$ (but does vanish  at $\PIK=1$ of course). 
We now plug the above expansion of $\PIK$ in the expression of $(1-3t)^2 \left (C(1,1)+1/t\right)^2$, extract square roots,   and find the following behaviours, respectively at  $\zeta_0/3$ and $\zeta_\pm/3$:
      \begin{align*}
        (1-3t) \left (C(1,1)+1/t\right) &=
                                        {3}^{3/4}\sqrt {2-\sqrt 2}\,{(1-3t)}^{3/8} +\LandauO((1-3t)^{7/8}),\\
   (1-3t) \left (C(1,1)+1/t\right)       &= a_\pm + b_\pm  (1-3\zeta_\mp t)^{3/4}+\LandauO \left( 1-3\zeta_\mp t \right),
      \end{align*}
      for some non-zero constants $a_\pm$ and $b_\pm$. In sight of the factor $(1-3t)$ on the left-hand side, we conclude that, as far as the first order of $c_n$ is concerned, the only singularity that contributes is $1/3$, and that $c_n$ satisfies~\eqref{cn-est-K}.      
\end{proof}

\section{Reverse Kreweras steps}
\label{sec:RK}

In this section we take $\cS= \{\rightarrow, \uparrow, \swarrow\}$, so that $\tS=\{\nearrow, \leftarrow, \downarrow\}$.

\subsection{Statements of the results}
We will establish the following counterpart of Theorems~\ref{thm:K-U} and~\ref{thm:K-D}.

\begin{theorem}\label{thm:RK}
  The \gf\  $C(x,y)$  of walks with steps in $\{\rightarrow, \uparrow, \swarrow\}$ starting from $(0,0)$ and avoiding the negative quadrant is algebraic of degree $96$. It is given by the following equation:
\[
  (1-t(x+y +\bx\by)) C(x,y)=1 -t\bx\by C_-(\bx) -t\bx\by C_-(\by)- t \bx\by C_{0,0},
\]
where $C_{0,0}$ is algebraic of degree $6$ and $\C_-(x)$ is algebraic of degree $24$. These series can be expressed explicitly in terms of the series $\PIK$,
$\PIIK$ and $\PIIIK$ defined in Section~\ref{sec:intro} (see Table~\ref{tab:K}).
Indeed,
\[
  C_{0,0}= \frac\PIK t \cdot \frac{4-4\PIIK -2\PIIK^2+3\PIIK^3}{8 (1-\PIIK)},
  \]
  and
\[ 
    \left(2t C_-(x)+\bx^2-\frac \bx t +x +t C_{0,0}\right)^2=
    \left(\bx^2-\frac\bx t-x\right)^2
  +A_2\left(\bx^2-\frac\bx t-x\right)+ A_1\left( \frac 1 x - \frac 1 \PIK\right) \sqrt{1-x \PIK^2} + A_0,
\] 
where the series $A_0$, $A_1$ and $A_2$ belong to $\qs(t,\PIIIK)$, with respective degrees $6,12$, and $6$:
\[ 
      \frac{ A_0}{\PIK^2}= -  \frac{8+18\PIIK-20\PIIK^2+5\PIIK^3-6\PIIK^4+4\PIIK^5}{8\PIIK (1-\PIIK)},
\]
\[
      A_1=  \frac{(2-\PIIK)^3(1+\PIIK)}2 \sqrt{\frac{1+\PIIK}{1-\PIIK}}
     = \frac{(2-\PIIK)^3(1+\PIIK)}2       \cdot {\frac{1+\PIIIK}{1-\PIIIK}},
\]
and
\[
      \frac{A_2}\PIK=
    \frac{4-4\PIIK -2\PIIK^2+3\PIIK^3}{4 (1-\PIIK)}.
\]
The \gf\ $D(x)$ of walks ending on the diagonal is algebraic of degree $24$, given by
\[ 
 \Delta(x) \left(xD(x)+ \frac{1}{tx}\right)^2=\left(\bx^2-\frac\bx t-x\right)^2
  +A_2\left(\bx^2-\frac\bx t-x\right)- A_1\left( \frac 1 x - \frac 1 \PIK\right) \sqrt{1-x \PIK^2} + A_0,
\] 
with the above values of $A_0, A_1, A_2$ and
\beq\label{Delta-def-RK}
   \Delta(x)= (1-t\bx)^2-4t^2x.
 \eeq
\end{theorem}

For walks ending at a prescribed position, we obtain the following result.

\begin{cor}\label{cor:excursions-RK}
  Let us define $\PIK, \PIIK $ and $\PIIIK$ as above.   For any $(i,j)\in \Cc$, the \gf\ of walks avoiding the negative quadrant and ending at $(i,j)$ is algebraic of degree (at most) $12$ and belongs to $\qs(t, \PIIIK)$.
  More precisely, $C_{i,j}/t^{i+j}$ belongs to $\qs(\PIIIK)$.
\end{cor}
This holds obviously for the series $C_{0,0}$ given in Theorem~\ref{thm:RK}. Note that this series coincides with the series $C_{0,0}$ obtained for Kreweras walks (it suffices to rewrite $\PIK/t$ in terms of $\PIIK$ to obtain the form~\eqref{C00-K}): this is clear, upon reversing the direction of walks.  Other examples are
\beq\label{Cat11-RK}
  tC_{1,1} = {\frac {2\PIIIK \left( 2 {\PIIIK}^{9}-{\PIIIK}^{8}-4 {\PIIIK}^{7}+10 {\PIIIK}^{6}-10 {\PIIIK}^{4}+6 {\PIIIK}^{3}+4 {\PIIIK}^{2}-4 \PIIIK+1 \right) }
    { \left(1- \PIIIK \right) ^{2}
 \left(1+ {\PIIIK}^{2} \right) ^{4}}}
,
\eeq
\[
  t  C_{-1,0} = \frac {A_1}4-1=\frac{(2-\PIIK)^3(1+\PIIK)}8 \cdot{\frac{1+\PIIIK}{1-\PIIIK}}-1.
\]

\medskip
The length \gf\  $C(1,1)$ of walks avoiding the negative quadrant can be characterized using the first two results of Theorem~\ref{thm:RK}.

\begin{cor}\label{cor:C11-RK}
  The  \gf\  $C(1,1)$ is algebraic of degree $24$ over $\qs(t)$. It is given by
\beq\label{C11-RK}
  (1-3t)^2(1+tC(1,1))^2=1- tA_2 + t^2A_1\left(1- \frac 1 \PIK\right)\sqrt{1-\PIK^2}+t^2A_0,
  \eeq
  where the series $\PIK$, $A_0$, $A_1$ and $A_2$ are those of Theorem~\ref{thm:RK}.
  The asymptotic behaviour of its $n$-th coefficient  $c_n$ is  obtained via singularity analysis:
\[
  c_n \sim \frac 9 {4\Gamma(5/8)} \left( \frac 9 2 - 3\sqrt 2\right) ^{1/4} 3^n n^{-3/8}.
\]
\end{cor}

\noindent{\bf Harmonic function.} We have also derived the counterpart of Corollary~\ref{cor:harmonic-K}, that is, the
harmonic function for walks with steps in $\cS=\{\rightarrow,
\uparrow, \swarrow\}$ confined to $\Cc$. We still have an asymptotic behaviour of the form~\eqref{cij-est}, this time with $n\equiv i+j$ mod $3$. We do not give here all details of the  numbers $H_{i,j}$,
but simply the values of
\[
  \Hc_-(x):=\sum_{i>0} H_{-i,0} x^i=\frac {27\sqrt 3}{8} \left(\sqrt{1+ \frac{3\sqrt 3 x}{ 2(1-x)^{3/2}}}-1\right)
\]
and
\[
  \Hc_d(y):=\sum_{i\ge 0} H_{i,i}y^i =\frac{27\sqrt 3}{4(1-y)
    \sqrt{1-4y}}
  \sqrt{1- \frac{3\sqrt 3 y}{ 2(1-y)^{3/2}}},
\]
from which all numbers $H_{i,j}$ can be reconstructed via an equation
that is the counterpart of~\eqref{H-eq}. The proof is analogous to
Kreweras' case. Once again, details of the calculations can be found
in our \Maple\ session. For comparison, the number of quadrant $\tS$-walks of length~$n$ going from $(0,0)$ to $(i,j)$ is  asymptotic to $h_{i,j} 3^n n^{-5/2}/\Gamma (-3/2)$, with
\[
  \sum_{i\ge 0 } h_{i,0} x^i= \frac 9{(1-x)^{3/2}}.
\]

\subsection{Proofs for reverse Kreweras steps}
\begin{proof}[Proof of Theorem~\ref{thm:RK}]
  The first equation of the  theorem is of course the basic functional equation~\eqref{eqfunc-gen}.  The functional equations~\eqref{eqD2} and~\eqref{eqU2} defining $D(y)$ and $U(x,y)$ read
  \begin{align}
 \label{eqD-RK}
(1-t\by) D(y) &=1 -t\by D_0+2tU(0,y) ,
\\
 \label{eqU-RK}
   x(1-t(\bx+\by+xy)) U(x,y) &=     tx y  D(y)-tx\by U(x,0) -tU(0,y).
  \end{align}
 The invariants of Proposition~\ref{prop:inv-tq} are
 \beq\label{IJ-RK}
   I(x)=\left( 2txU(x,0)+\frac 1 {x^2} -\frac 1 {tx} + x+ tD_0\right)^2, \qquad J(y)=\Delta(y)\left(yD(y)+ \frac 1 {ty} \right)^2,
 \eeq
 with $\Delta$ defined by~\eqref{Delta-def-RK}.
 The rational invariants $(I_0,J_0)$ are given in Table~\ref{tab:inv-qu-1}:
 \beq\label{I0-RK}
   I_0(x)=\bx^2-\bx/t -x, \qquad J_0(y)=I_0(y).
 \eeq
 They satisfy
\beq\label{ratio:I0-RK}
\frac{I_0(x)-J_0(y)}{ \tK(x,y)}= \frac{x-y}{txy}.
\eeq
 The invariants related to quadrant walks with steps in $\tS$, defined by~\eqref{I1J1-def}, are
  \beq\label{I1-RK}
    I_1(x)=   tx \tQ(x,0)+ \bx -\frac 1{2t}, \qquad
   J_1(y)= -ty\tQ(0,y) -\by+\frac 1{2t}  = -I_1(y).
 \eeq
 They satisfy
 \beq\label{ratio:I1-RK}
  \frac {  I_1(x)-J_1(y)}{\tK(x,y)}=- xy\tQ(x,y)-\frac 1t. 
  \eeq

\medskip
\paragraph{{\bf Applying the invariant lemma.}} 
 As in the previous section, we want to combine the pair of invariants $(I(x),J(y))$ defined by~\eqref{IJ-RK} with those of~\eqref{I0-RK} and~\eqref{I1-RK} to form a pair $(\tilde I(x), \tilde J(y))$ satisfying the condition of Lemma~\ref{lem:invariants}.
 Since the conclusion of this lemma is that $\tilde I(x)$ and $\tilde J(y)$ belong to $\qs((t))$,
 we must look for a polynomial combination of $I(x)$, $I_0(x)$ and $I_1(x)$ that, at least, has no pole at $x=0$. In sight of the expansions of these three series at $x=0$, we are led to consider
 \beq\label{tildeI-RK}
 \tilde I(x):=  I(x) -I_0(x)^2 -A_2I_0(x)-A_1 I_1(x),
 \eeq
 with
 \beq\label{A12-KR}
   A_2=2tD_0, \qquad A_1= 4(1+tU_{0,0}),
 \eeq
 which indeed has no pole at $x=0$.  Let us define accordingly
\beq\label{tildeJ-RK}
 \tilde J(y):=     J(y) -J_0(y)^2 -A_2J_0(y)-A_1 J_1(y),
 \eeq
 and examine whether $(\tilde I(x)-\tilde J(y))/\tK(x,y)$ is a multiple of $xy$. Denoting by $\Rat$, $\Rat_0$ and $\Rat_1$ the right-hand sides of~\eqref{IJ-id}, \eqref{ratio:I0-RK} and~\eqref{ratio:I1-RK} respectively, we have
\beq\label{ratio-RK-tilde} 
   \frac{\tilde I(x)-\tilde J(y)}{\tK(x,y)}=  \Rat -(I_0(x)+J_0(y))\Rat_0 -A_2 \Rat_0-A_1 \Rat_1.
\eeq
This ratio has poles of bounded order at $0$. We first expand it around $x=0$, and observe that it is $\LandauO(1/x)$. In order to prove that the coefficient of $1/x$ is in fact $0$,  we use~\eqref{eqD-RK} (see our \Maple\ session for details). In order to prove that the coefficient of $x^0$ is also $0$, we use moreover the first term in the expansion of~\eqref{eqU-RK} at $x=0$, which gives an expression of  $U'_x(0,y)$ in terms of $U(0,y)$, $D(y)$ and $U_{0,0}$. Similarly, in order to prove that the  ratio~\eqref{ratio-RK-tilde}  is a multiple of $y$, we  inject in its $y$-expansion (which is $\LandauO(y^0)$) the expansion at $y=0$ of~\eqref{eqD-RK}, which relates $D_0$, $D_1$ and $U_{0,0}$,  and the first term in the expansion of~\eqref{eqU-RK} at $y=0$, which gives an expression of  $U'_y(x,0)$ in terms of $U(x,0)$ and $U_{0,0}$.  Then we happily conclude that the pair $(\tilde I(x), \tilde J(y))$ is independent of $x$ and $y$, equal to $(A_0, A_0)$ for some series $A_0$ that depends on $t$ only.

 Recall the expression~\eqref{I1-RK} of $I_1(x)$ in terms of the \gf\ $\tQ(x,0)$ of quadrant walks with Kreweras steps. This series is known, and we obtain from Proposition~13 in~\cite{BMM-10}:
 \beq\label{I1-expl-K}
   I_1(x)= \left(\frac 1 x-\frac 1 \PIK\right) \sqrt{1-x \PIK^2}=-J_1(x).
 \eeq
 This gives the expressions of $I(x)$ and $J(x)$ announced in Theorem~\ref{thm:RK}, but we still need to determine the three series $A_0, A_1$ and $A_2$.

\medskip
\paragraph{{\bf The series $A_0$, $A_1$, and $A_2$.}} 
 In the Kreweras case (Section~\ref{sec:K}) we only had one unknown series to determine, denoted~$A$ (see~\eqref{eqinv-K}). We derived it using the (only) root of $\Delta(y)$ that was finite at $t=0$ for this model.  Now with the new value of $\Delta(y)$, we have two such roots -- but three series to determine. A third identity between the $A_i$'s will follow from the fact that $I(x)$ has a double root; see~\eqref{IJ-RK}.

 We begin with a useful property, which tells us that $I_0(x)$ is (almost) the square of $I_1(x)$. Indeed, the expressions of $I_0(x)$ and $I_1(x)$, together with the definition~\eqref{PIK-k-def} of $\PIK$, imply that
 \beq\label{I1-I0-RK}
   I_1(x)^2= I_0(x)+ \frac 1 {\PIK^2} + 2\PIK.
 \eeq
The series $J_1(y)=-I_1(y)$ and $J_0(y)=I_0(x)$ are related by the
same identity. (If we did not know the explicit
expression~\eqref{I1-expl-K} of $I_1(x)$, we could still derive this identity from the invariant lemma, which gives a polynomial relation between the pairs $(I_1(x), J_1(y))$ and $(I_0(x), J_0(y))$; this is how quadrant walks with steps in $\tS$ are solved in~\cite{BeBMRa-17}.)

 Using this, and the definitions~\eqref{tildeI-RK} and~\eqref{tildeJ-RK} of $\tilde I$ and $\tilde J$, the fact that $\tilde I(x)=\tilde J(y)=A_0$ translates into
 \beq\label{IJ-P-RK}
   I(x)= P_4(I_1(x)), \qquad J(y)=P_4(J_1(y)),
 \eeq
 where
\[ 
   P_4(u)= \left(u^2- \frac 1 {\PIK^2} - 2\PIK\right)^2+A_2\left(u^2- \frac 1 {\PIK^2} - 2\PIK\right) + A_1 u+A_0.
\] 
 Note that the identities of~\eqref{IJ-P-RK} are those that we would obtain from the invariant lemma by playing with the pairs $(I(x), J(y))$ and $(I_1(x),J_1(y))$ only, with no reference to $(I_0(x),J_0(y))$.

 As already mentioned, two roots of $\Delta(y)$ are finite at $t=0$. We denote them by $Y_+$ and $Y_-$:
 \[
   Y_{\pm}= t \pm 2t^{5/2} +6t^4\pm 21 t^{11/2} +\mathcal O(t^7).
\]
By replacing $t$ by its expression in terms of $\PIK$ in $\Delta(y)$, we see that $Y_\pm$  are the roots of the following quadratic polynomial in $Y$:
\[
  4Y^2-\PIK(\PIK^3+4) Y+\PIK^2.
\]
The third root of $\Delta$ is $1/\PIK ^2$. From this, and the  expression~\eqref{I1-expl-K} of $I_1(x)=-J_1(x)$,
we conclude that the values $J_1(Y_\pm)$ are  roots of
\[
  \tilde P_4(u)=4u^4\PIK^4+\PIK^2(\PIK^6-20\PIK^3-8)u^2-4\PIK^9+12\PIK^6-12\PIK^3+4.
\]
Since $J(y)$ contains a factor $\Delta(y)$, the values $J_1(Y_\pm)$ are also roots of $P_4(u)$, and hence of the following quadratic polynomial:
\begin{align*}
  P_2(u)&= 4\PIK^2P_4(u) -\tilde P_4(u)/\PIK^2\\
  &=
  \PIK^2(4A_2-\PIK^4+4\PIK)u^2+4\PIK^2A_1u+4(\PIK^7-2A_2\PIK^3+\PIK^4+A_0\PIK^2-A_2+7\PIK).
\end{align*}
Hence $P_4(u)$ contains a factor $P_2(u)$.

Now consider the equation $I(X)=0$, that is,
\[
  X=t \left(1+X^3+tX^2D_0+2tX^3U(X,0)\right).
\]
From the form of this equation, we see that it admits a unique solution $X$ in the ring $\qs[[t]]$. The first terms in its expansion are
\[
  X= t+ 2t^4+16 t^7 + \mathcal O(t^{10}).
\]
Since $I(X)=0$ and $I'(X)=0$, it follows from~\eqref{IJ-P-RK} that $I_1(X)$ is a root of $P_4(u)$, and even a double root, unless $I'_1(X)=0$. But $I'_1(X)= -1/t^2 + \LandauO(t)$, hence $P_4(u)$ admits indeed $I_1(X)$ as a double root. Moreover, we easily check that  $I_1(X)\not = J_1(Y_\pm)$.  Thus we have the following factorization:
\[
  P_4(u)= \frac 1 { \PIK^2(4A_2-\PIK^4+4\PIK)} (u-I_1(X))^2 P_2(u).
\]
(The multiplicative constant is adjusted using the leading coefficients of $P_2$ and $P_4(u)$.) By extracting  the coefficients of $u^3, \ldots , u^0$ in this identity, we obtain four polynomial equations that relate $I_1(X)$, $A_0$, $A_1$, and $A_2$. By eliminating the first three series, we obtain a polynomial equation for $A_2$. We  determine which of its factors  vanishes thanks to the first coefficients of $A_2=2tD_0$ (see~\eqref{A12-KR}), and, after introducing the series $\PIIK$,  we obtain the value of $A_2$ given in the theorem. The values of $A_1$ and $A_0$ follow similarly. We refer to our \Maple\ session for details.
Finally, the expression of $C_{0,0}$ follows from $C_{0,0}=D_0=A_2/(2t)$.

\medskip
\paragraph{\bf Degrees.} Let us now discuss the algebraic degrees of our \gfs. The expression of $C_{0,0}$ shows that it has degree at most $6$, and this is easily checked to be tight by elimination of $\PIIK$ and $\PIK$. The expression of $I(x)$ shows that it belongs to an extension of degree $12$ of $\qs(t,x)$: this can be seen by extending this field first by $\PIK$ (degree $3$), then $\PIIK$ (degree $2$), and finally by $A_1\sqrt{1-x\PIK^2}$; the square of the latter series being in $\qs(x,t, \PIK, \PIIK)$ (because $\frac{1+\PIIIK}{1-\PIIIK}= \sqrt{\frac{1+\PIIK}{1-\PIIK}} $), this yields another extension of degree $2$, resulting in total into a degree $12$ at most for $I(x)$. Since $C_{0,0} \in \qs(\PIIK)$, this implies that $U(x,0)$ has degree at most $24$. An effective elimination procedure  shows that these bounds are tight. The same argument shows that $J(x)$ has degree $12$, and $D(x)$ degree~$24$. In fact, $I(x)$ and $J(x)$ satisfy the same equation over $\qs(t,x)$, as shown by their respective expressions. Finally, the basic functional equation implies that the degree of $C(x,y)$ is the degree of $C_-(x)+C_-(y)+C_{0,0}$. We have seen that $C_{0,0}$ belongs to $\qs(t,\PIK, \PIIK)$, and that $C_-(x)$ belongs to an extension of degree $4$ of $\qs(t,x, \PIK, \PIIK)$. Hence the above sum belongs to an extension of $\qs(t,x, \PIK, \PIIK)$ of degree at most $4\times 4=16$, and thus has degree at most $16 \times 6=96$. We  compute its minimal polynomial over $\qs(t)$ for $x=2$ and $y=3$, and find this bound to be tight.
\end{proof}

Next we establish the results that deal with walks ending at a prescribed point $(i,j)$.
\begin{proof}[Proof of Corollary~\ref{cor:excursions-RK}]
We begin with walks ending on the diagonal, the \gf\ of which is given in Theorem~\ref{thm:RK}. We consider the series 
\beq\label{diag-RK-shift}
x D(x/\PIK^2) + \frac {\PIK^4}{ tx},
\eeq
which, due to the periodicity of the model and the properties of $\PIK$, is a series in $t^3$ (or equivalently, in $\PIK^3$) with coefficients in $\qs[\bx,x]$. Using the second part of Theorem~\ref{thm:RK}, we first rewrite its square in terms of $x$ and the series $\PIIIK$ (this series is only needed because of the term $A_1$, otherwise we could do with $\PIIK$ as in Kreweras' case).  We thus obtain a (Laurent) series in $x$ with rational coefficients in $\PIIIK$. The coefficient of $x^{-2}$ in this series is ${\PIK^8}/{t^2}$.
Hence the square root of this series, which is~\eqref{diag-RK-shift}, is also a series in $x$ with coefficients in $\qs(\PIIIK)$. The coefficient of $x^i$ in~\eqref{diag-RK-shift} is
$C_{i-1,i-1}/\PIK^{2i-2}$ as soon as $i>0$. The  corollary thus follows for  walks  ending on the diagonal.   In particular, for $i=2$ we obtain the expression~\eqref{Cat11-RK} of $C_{1,1}$.

The same argument, starting now from
\[ 
  2
  C_-\left( \frac x {\PIK^2}\right) + \frac{\PIK^4}{tx^2} - \frac{\PIK^2}{t^2x} + \frac x{t\PIK ^2} + C_{0,0},
\] 
establishes the result for walks ending on the negative $x$-axis, thanks to the expression of $I(x)$. Indeed, one can show that this is a (Laurent) series in $x$ with rational coefficients in $\PIIIK$, and the coefficient of $x^i$ in this series, for $i>0$, is $2 C_{-i,0}/\PIK^{2i}$. We obtain in particular the expression of $C_{-1,0}=U_{0,0}$ given below Corollary~\ref{cor:excursions-RK} (but it also follows from~\eqref{A12-KR} and the expression of $A_1$).

We can now prove the corollary for any point $(i,j)$ in the three quadrant cone. Without loss of generality, we assume $j\ge i$, and  argue by induction on $j\ge -1$. The result is obvious when $j=-1$ (because then $i<0$ and $C_{i,j}=0$), and has just been established for $j=0$. Since it has also been proved  for $j=i$,  we assume that $j\ge i+1$ and $j\ge 1$. We have
\[
  C_{i-1,j-1} = t \left( C_{i-1, j-2} + C_{i-2,j-1} + C_{i,j} \right),
\]
and the result holds for the three series $  C_{i-1,j-1}, C_{i-1, j-2} $ and $ C_{i-2,j-1} $ by the induction hypothesis. Thus it holds for $C_{i,j}$ as well.
\end{proof}

We conclude this section with the series $C(1,1)$.
\begin{proof}[Proof of Corollary~\ref{cor:C11-RK}] 
  The value of $C(1,1)$ is obtained by specializing at $x=y=1$ the first and third identities of Theorem~\ref{thm:RK}. The fact that the degree of $C(1,1)$ is $24$ follows either by direct elimination, or by using the connection~\eqref{I1-C11} between $I(1)=R(1)^2$ and $C(1,1)$ and the  polynomial equation (of degree $12$) over $\qs(t,x)$ that we have established for $I(x)$ in the proof of Theorem~\ref{thm:RK}.

  The asymptotic result is obtained through singularity analysis, using the expressions of $A_0$, $A_1$ and $A_2$, and the singular properties of $\PIK$ and $\PIIK$ (see~\eqref{PIK3-sing},~\eqref{PIIK-sing},~\eqref{PIK-sing-refined}, and the proof of Corollary~\ref{cor:all-K} at the end of Section~\ref{sec:K}).
\end{proof}

\section{Double Kreweras steps}
\label{sec:DK}

In this section we take $\cS= \{\rightarrow, \uparrow, \swarrow,\nearrow, \leftarrow, \downarrow\}$, so that $\tS=\cS$.

\subsection{Statements of the results}
In order to state  the counterparts of Theorems~\ref{thm:K-U} and~\ref{thm:K-D}, we first need to introduce an extension of $\qs(t)$ of degree~$16$.

First, we define $\PID\equiv \PID(t)$ as the unique formal power series in $t$
satisfying
\[
  \PID=t(1+2\PID+4\PID^2).
\]
The  coefficient of $t^{n+1}$ in $\PID$ is $2^n$ times the $n$th Motzkin number, hence the notation $M$. The series $\PID/t$ is known to count walks with steps in $\cS$ confined to the North-West quadrant $\{(i,j), i\le 0, j \ge 0\}$ (see~\cite[Prop.~10]{BMM-10}). Then we consider an extension of degree $4$ of $\qs(\PID)$, 
generated by the series
$\UI=t+ \LandauO(t^2)$ satisfying
 \[
   \UI ={\frac {\PID \left(1+ \PID \right) ^
      {3} }{ \left(1+ 4\,\PID \right) ^{3}}}\cdot\frac{\left(1+ \UI \right) ^{4}}{\left( 1-\UI \right) ^{2}}
\]
Finally, we define
  \[
    A_1= 4 (1+\PID) \sqrt{\frac{(1+\PID)\UI}\PID} =
    \frac{4\,(1+\PID)^3}{(1+4\PID)^{3/2}} \cdot \frac
    {(1+\UI)^2}{1-\UI                    },
  \]
  which has degree $8$ above $\qs(\PID)$ and degree $16$ above $\qs(t)$. Between $\qs(\PID)$ and $\qs(\PID,A_1)$, there are three extensions of $\qs(\PID)$ of degree $2$, respectively generated by $\sqrt{1+4\PID}$, $\sqrt{1-4\PID^2}$ and $\sqrt{(1+4\PID)(1-4\PID^2)}$.  We shall consider in particular a generator of $\qs(\sqrt{1+4\PID})$, denoted~$\PIID$, defined by $\PIID=\mathcal O(t)$ and
\beq\label{PIID-def}
    \PIID = t\, \frac{\PIID^4-2\PIID^3+6\PIID^2-2\PIID+1}{(1-\PIID)^2}, \quad \hbox{or equivalently} \quad \PIID=\PID(1-\PIID)^2.
\eeq
Figure~\ref{fig:ext-DK} gives a complete description of $\qs(t,A_1)$
and its subfields. The series denoted by $\UII$ is the only solution of
\[
 \UII^4 - (1+4\PID)^3\UII^2 + 4\PID(1+ \PID)^3(1+4\PID)^3 =0
  \]
  satisfying $\UII=1+\LandauO(t)$. It is related to $\PID$, $\UI$ and
  $A_1$ by
  \beq\label{P2-char}
    \UII= (1+4\PID)^{3/2} \, \frac{1-\UI}{1+\UI} = \frac{\PID  A_1}{4} + 4 \frac{(1+\PID)^3}{A_1}.
    \eeq

\begin{figure}[t!]
\hskip -15mm  \scalebox{0.8}{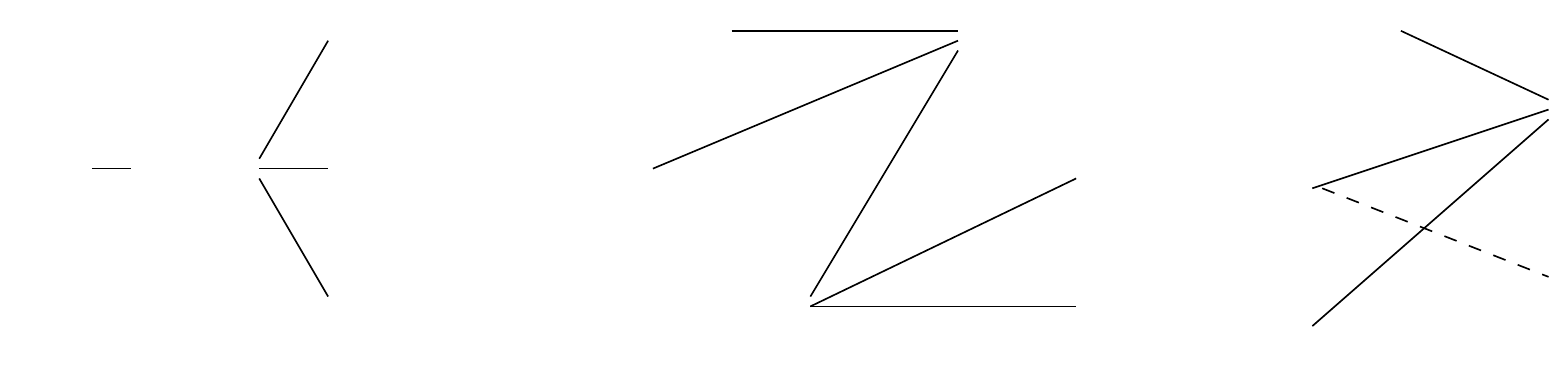}
  \caption{The extension of $\qs(t)$ of degree $16$ generated by $A_1$. All elementary extensions have degree $2$. The series $C(1,1)$ has degree $16$ over $\qs(t)$, as $A_1$, but lies in a different extension of $\qs(t)$.}
  \label{fig:ext-DK}
\end{figure}

\begin{theorem}\label{thm:DK}
  The \gf\  $C(x,y)$  of walks with steps in $\{\rightarrow, \uparrow, \swarrow,\nearrow, \leftarrow, \downarrow \}$ starting from $(0,0)$ and avoiding the negative quadrant is algebraic of degree $256$. It is given by the following equation:
\[
  (1-t(x+y +xy+\bx\by+\bx+\by)) C(x,y)=1 -t\by(1+\bx) C_-(\bx) -t\bx(1+\by) C_-(\by)- t \bx\by C_{0,0},
\]
where $C_{0,0}$ is algebraic of degree $16$ and $\C_-(x)$ is algebraic of degree $64$. These series can be expressed in terms of the series $\PID$, $\PIID$ and $A_1$  defined  above. Indeed,
  \beq\label{C00-DK}
    t C_{0,0}=
       1+ \frac{(1+2\PID)^2}{2\PID} -\frac{3A_1}8 - \frac{2(1+\PID)^3}{\PID A_1},
  \eeq
 and
\begin{multline}\label{C-DK}
    \left(2t (1+\bx) C_-(x)+\bx +1+x-\frac{1+2t}{t(1+x)} +t C_{0,0}\right)^2=
    \left( \frac 1 x-x-{\frac {1+2\,t}{t \left( x+1 \right) }}\right)^2\\
    +\tA_2 \left( \frac 1 x-x-{\frac {1+2\,t}{t \left( x+1 \right) }}\right)+A_1 \frac{\PIID+2x\PIID/(1-\PIID)-x^2}{2x(1+x)\PIID}\sqrt{\Delta_+(x)}+\tA_0,   
\end{multline}
  where $\Delta_+(x)$ is the following polynomial in $x$:
  \beq\label{Delta-plus-DK}
    \Delta_+(x)=1- \frac{2\PIID(1+\PIID^2)}{(1-\PIID)^2} x +\PIID^2 x^2,
    \eeq
    and  $\tA_2$ and $\tA_0$ are algebraic series of respective degree $8$ and $16$, both in $\qs(t, A_1)$:
\[ 
      \tA_2= \frac{(1+2\PID)^2}\PID -\frac{A_1}4 -
      \frac{4(1+\PID)^3}{\PID A_1}
      = \frac{(1+2\PID)^2}\PID - \frac{\UII}{\PID},
\]
\[ 
      \tA_0=
      \frac{A_1^{2}}8
   +   \left({\frac {A_1\,  }{8\PID}}+2\,{
\frac { \left( \PID+1 \right) ^{3}  }{A_1\,{\PID}
  ^{2}}} \right) \left((1+4\PID)^{3/2}-(1+2\PID)^2\right)
  +{\frac {2\,{\PID}^{3}-14\,{\PID}^  {2}-12\,\PID-3}{\PID}}
.
\] 

The \gf\ $D(x)$ of walks ending on the diagonal is algebraic of degree $64$, given~by
\begin{multline}\label{D-DK}
  \Delta(x) \left(xD(x)+ \frac{1}{t(1+x)}\right)^2=
     \left( \frac 1 x-x-{\frac {1+2\,t}{t \left( x+1 \right) }}\right)^2\\
     +\tA_2 \left( \frac 1 x-x-{\frac {1+2\,t}{t \left( x+1 \right) }}\right)-A_1 \frac{\PIID+2x\PIID/(1-\PIID)-x^2}{2x(1+x)\PIID}
     \sqrt{\Delta_+(x)}+\tA_0,   
\end{multline}
with the above values of $\Delta_+(x), \tA_0, A_1, \tA_2$ and
\beq\label{Delta-def-DK}
   \Delta(x)= (1-t(x+\bx))^2-4t^2\bx(1+x)^2.
 \eeq
\end{theorem}

For walks ending at a prescribed position, we obtain the following result.

\begin{cor}\label{cor:excursions-DK}
  Let us define $A_1$ as above.   For any $(i,j)\in \Cc$, the \gf\ of walks avoiding the negative quadrant and ending at $(i,j)$ is algebraic of degree (at most) $16$ and belongs to $\qs(t, A_1)$.
\end{cor}

A first example is provided by the expression of $C_{0,0}$ in~\eqref{C00-DK}. A simpler one is 
\beq\label{Cm10-D}
 t C_{-1,0} =\frac{A_1}4-1.
\eeq

Now for the \gf\ C(1,1) that counts all walks avoiding the negative quadrant, it follows from Theorem~\ref{thm:DK} that
\[ 
   (1-6t)^2 \left( C(1,1) + \frac 1 {2t} \right)^2=
    \left( {\frac {1+2\,t}{2t  }}\right)^2
    -\tA_2 {\frac {1+2\,t}{2t }}+A_1 \frac{\PIID+2\PIID/(1-\PIID)-1}{4\PIID}\sqrt{\Delta_+(1)}+\tA_0.
\] 
  Upon noticing that $\Delta_+(1)=(1-\PIID)^2(1-4\PID^2)$, one sees
  that the right-hand side belongs to $\qs(t, A_1)$, and could thus be
  expected to have degree $16$ (which would give an expected
  degree~$32$ for $C(1,1)$). However, the right-hand side belongs to
  $\qs(t,A_1^2)$ and has degree $8$ only, as made explicit in the
  following corollary. Details of the derivation are available in our \Maple\ session.

\begin{cor}\label{cor:C11-DK}
  The \gf\ $C(1,1)$ counting by their length all walks with steps in $\cS$ avoiding the negative quadrant is algebraic of degree $16$ over $\qs(t)$, and of degree $2$ over $\qs(t,A_1^2)$. More precisely,
  \begin{multline}
  (1-6t)^2 \left( C(1,1) + \frac 1 {2t} \right)^2=    \label{C11-DK} \\
  {\frac {\PID A_1^{4}}{512\, \left( \PID+1 \right) ^{3}}}+
  {\frac { \left( 10\,{\PID}^{4}-34\,{\PID}^{3}-18\,{\PID}^{2}-2\,\PID-1 \right) A_1
      ^{2}}{32\PID \left( \PID+1 \right) ^{3}}}
  -{\frac {14\,{\PID}^{4}-22\,{\PID}^{3}
-6\,{\PID}^{2}+2\,\PID-1}{{4\PID}^{2}}} 
    .
   \end{multline}
  It has radius $1/6$, with a unique dominant singularity at $1/6$. As $t$ approaches this point,
  \[
    C(1,1) \sim  \kappa (1-6t)^{-5/8},
  \]
  with
  \[
    \kappa= 2^{1/4} 3^{9/8} (\sqrt 2-1).
    \]
  Hence  the number of walks of length $n$ avoiding the negative quadrant satisfies
  \[
    c_n \sim \frac  {\kappa} {\Gamma(5/8)} 6^{n} n^{-3/8}.
  \]
  \end{cor}

\noindent{\bf Harmonic function.} We have also derived the counterpart of Corollary~\ref{cor:harmonic-K}, that is, the
harmonic function for walks with steps in $\cS=\{\rightarrow, \nearrow,
\uparrow, \leftarrow, \swarrow, \downarrow\}$ confined to $\Cc$. We still have an asymptotic behaviour of the form~\eqref{cij-est}, with the factor $3^n$ replaced of course by $6^n$, and without any periodicity condition. We do not give here all details of the  numbers $H_{i,j}$,
but simply the values of
\[
  \Hc_-(x):=\sum_{i>0} H_{-i,0} x^i=
  \frac{3^{7/4} x\sqrt{\sqrt 2 -1}}{\sqrt 2 (1+x) }
   \left(
  \sqrt{\sqrt 2 + \frac{2-\sqrt 3 +x}{1-x}\sqrt{\frac{7+4\sqrt 3 -x}{1-x}}}
    -\sqrt{\sqrt 2-1}\right)
  \]
and
\[
  \Hc_d(y):=\sum_{i\ge 0} H_{i,i}y^i =
  \frac{3^{7/4}\sqrt 2 \sqrt{\sqrt 2 -1}}{(1-y)\sqrt{1-14y+y^2}}
  \sqrt{\sqrt 2 - \frac{2-\sqrt 3 +y}{1-y}\sqrt{\frac{7+4\sqrt 3-y}{1-y}}}.
\]
The proof is analogous to Kreweras' case, and in fact a bit easier since the model $\cS$ is aperiodic. Once again, details of the calculations can be found in our \Maple\ session. The number of quadrant $\tS$-walks of length~$n$ going from $(0,0)$ to $(i,j)$ in the first quadrant is now asymptotic to $h_{i,j} 6^n n^{-5/2} / \Gamma(-3/2)$, with
\[
  \sum_{i\ge 0} h_{i,0} x^i=\frac 3 {2(1+x)} \left( \frac{2-\sqrt 3 +x}{1-x} \sqrt{\frac{7+4\sqrt 3 -x}{1-x}} +1 \right).
\]

\subsection{Proofs for double Kreweras steps}

\begin{proof}[Proof of Theorem~\ref{thm:DK}]
The first equation of the  theorem is of course the basic functional equation~\eqref{eqfunc-gen}. 
  The functional equations~\eqref{eqD2} and~\eqref{eq-U} defining $D(y)$ and $U(x,y)$ read
  \beq\label{eqD-DK}
  (1-t(y+\by)) D(y) =1 -t\by D_0+2t(1+\by)U(0,y) -2t\by U_{0,0},
  \eeq
 \begin{multline*}
      2xy(1-t(\bx\by+x+y+\bx+\by+xy)) U(x,y) = \\
  y+
  y\left(t(y+\by) +2tx(1+y)-1\right) D(y) -2t(1+x) U(x,0) - t D_0.
 \end{multline*}
 The invariants of Proposition~\ref{prop:inv-tq} are
\[ 
   I(x)=\left( 2t(1+x)U(x,0)+\frac 1 {x}+1+x -\frac {1+2t} {t(1+x)} + tD_0\right)^2, \qquad J(y)=\Delta(y)\left(yD(y)+ \frac 1 {t(1+y)} \right)^2,
\] 
 with $\Delta$ defined by~\eqref{Delta-def-DK}.  Note that they have now poles at $0$ \emm and, at $-1$.
 The rational invariants $I_0,J_0$ are given by Table~\ref{tab:inv-qu-1}:
  \[
   I_0(x)= \bx -x -\frac{1+2t}{t(1+x)}, \qquad J_0(y)=I_0(y).
 \]
They satisfy
 \beq\label{I0-DK}
   I_0(x)-J_0(y)= \frac{x-y}{t(1+x)(1+y)} \tK(x,y).
 \eeq
 The invariants~\eqref{I1J1-def} related to quadrant walks with steps in $\tS$ are
 \[
 I_1(x)= t(1+x)\tQ(x,0) -\frac{x-t-tx^2}{t(1+x)}, \qquad
 J_1(y)=-t(1+y)\tQ(0,y)+t\tQ_{0,0}-\by .
 \]
 They satisfy
   \beq\label{I1-DK}
     I_1(x)-J_1(y)= -\tK(x,y) \left( xy\tQ(x,y)+ \frac x{t(1+x)}\right).
   \eeq
   
\medskip
\paragraph{{\bf Applying the invariant lemma.}}  Again, we want to combine these three pairs of invariants  to form a pair satisfying the condition of Lemma~\ref{lem:invariants}.  As in the previous section, we look for a polynomial combination of $I(x), I_0(x)$ and $I_1(x)$ that, at  least, has no pole at $x=0$ nor at $x=-1$. We first eliminate poles at~$0$ using $I_0(x)$, and then poles at $-1$ using $I_1(x)$. We are thus led to consider
 \[
\tilde I(x):=   I(x) -I_0(x)^2 -A_2I_0(x)-A_1 I_1(x),
 \]
 with
 \beq\label{A12-DK}
   A_2=2(1+tD_0+2tU_{0,0}), \qquad A_1= 4(1+tU_{0,0}),
 \eeq
 which indeed has no pole at $x=0$ nor at $x=-1$. Let us define accordingly
 \[
  \tilde J(y)=       J(y) -J_0(y)^2 -A_2J_0(y)-A_1 J_1(y),
 \]
 and examine whether the ratio $(\tilde I(x)-\tilde J(y))/\tK(x,y)$ is a multiple of $xy$. We proceed as in the previous section. Denoting by $\Rat$, $\Rat_0$ and $\Rat_1$  the  right-hand sides of~\eqref{IJ-id},~\eqref{I0-DK} and~\eqref{I1-DK},  the expression~\eqref{ratio-RK-tilde} of  $(\tilde I(x)-\tilde J(y))/\tK(x,y)$
 still holds (because $\tilde I(x)$ and $\tilde J(y)$ have the same form for both models). In order to see that this ratio is indeed a multiple of $x$, we need to inject~\eqref{eqD-DK}. Proving that it is a multiple of $y$ requires no additional information. We conclude that the pair $(\tilde I(x), \tilde J(x))$ is constant, equal to $(A_0, A_0)$ for some series $A_0$ that depends on $t$ only.

\medskip Recall that $I_1(x)$ is expressed in terms of the \gf\ $\tQ(x,0)$ of quadrant walks with double Kreweras steps. This series is known, and we obtain from Proposition~15 in~\cite{BMM-10}:
\[ 
I_1(x)= -\frac 1 2 I_0(x) + I_2(x)
   -{\frac {2\,{\PIID}^{4}+{\PIID}^{3}+3\,{\PIID}^{2}-\PIID+1}{2\PIID \left(1- \PIID \right) ^{2}}},
\]
where $\PIID$ is defined by~\eqref{PIID-def} (this series is denoted
by $Z$ in~\cite{BMM-10}) and
\[
  I_2(x):= \frac{\PIID+2x\PIID/(1-\PIID)-x^2}{2x(1+x)\PIID}\sqrt{\Delta_+(x)},
\]
with $ \Delta_+(x)$ given by~\eqref{Delta-plus-DK}. In what follows,
we use $I_2(x)$ rather than $I_1(x)$. That is, we rewrite the identity $\tilde I(x)=A_0$, which reads
\[
  I(x)= I_0(x)^2 +A_2I_0(x)+A_1 I_1(x)+A_0
\]
as
\beq \label{I-expr-DK}
  I(x)= I_0(x)^2 +\tA_2I_0(x)+A_1 I_2(x)+\tA_0,
\eeq
where the two new series $\tA_0$ and $\tA_1$ are
\begin{align}
  \label{tA2-def}
  \tA_2&= A_2-A_1/2=2t(D_0+U_{0,0}),
\\
  \tA_0&=A_0-A_1 {\frac {2\,{\PIID}^{4}+{\PIID}^{3}+3\,{\PIID}^{2}-\PIID+1}{2\PIID \left(1- \PIID \right) ^{2}}}. \nonumber
\end{align}
Similarly, we find that the quadrant invariant $J_1(y)$ is given by
\[ 
J_1(y)= -\frac 1 2 J_0(y) - I_2(y)
   -{\frac {2\,{\PIID}^{4}+{\PIID}^{3}+3\,{\PIID}^{2}-\PIID+1}{2\PIID \left(1- \PIID \right) ^{2}}},
 \]
 and write
\beq\label{J-expr-DK}
  J(y)= J_0(y)^2 +\tA_2J_0(y)-A_1 I_2(y)+\tA_0,
\eeq
with the above values of $\tA_2$ and $ \tA_0$.

This gives the expressions~\eqref{C-DK} and~\eqref{D-DK} announced in Theorem~\ref{thm:DK}, which are those  of $I(x)$ and $J(x)$.   But we still need to determine the three series $\tA_0, A_1$ and $\tA_2$.

\medskip
\paragraph{{\bf The series $\tA_0, A_1$ and $\tA_2$.}} We cannot follow exactly  the same approach as in the previous section, where we had written all invariants as polynomials\footnote{It is still possible to write $I(x)$ as a \emm rational function, of $I_1(x)$; see Section~\ref{sec:rat-uniform}.} in $I_1(x)$ using the identity~\eqref{I1-I0-RK}. Indeed, the counterpart of this identity relates now $I_0$ and $I_2$, and is no longer linear in $I_0$:
\[
   I_2(x)^2= \frac{1}4 I_0(x)^2+ c_1 I_0(x) + c_0,
\]
 with
 \[
   c_1= \frac{1+\PIID+\PIID^2-\PIID^3}{2\PIID(1-\PIID)^2} \qquad \hbox{and} \qquad
   c_0={\frac { \left( \PIID^{2}+1 \right)  \left( 4\,\PIID^{4}-9\,\PIID^{3}+
         13\,\PIID^{2}-\PIID+1 \right) }{4 \PIID^{2}\left( 1-\PIID                                  \right) ^{3}}}.
 \]
However, compare the expressions~\eqref{I-expr-DK}
and~\eqref{J-expr-DK} of $I(x)$ and $J(y)$, and recall that
$J_0(y)=I_0(y)$. This, and the above expression of $I_2(x)^2$, suggests to form the product $I(x) J(x)$, which will be a  polynomial in $I_0(x)$:
 \begin{align*}
   I(x) J(x)&= \left( I_0(x)^2 +\tA_2I_0(x)+\tA_0\right)^2-A_1^2
              I_2(x)^2 \\
            &= \left( I_0(x)^2 +\tA_2I_0(x)+\tA_0\right)^2-A_1^2
   \left(\frac{1}4 I_0(x)^2+ c_1 I_0(x) + c_0\right)
 \\ &  := P_4(I_0(x)),
 \end{align*}
 where
 \[
 P_4(u)= \left( u^2 +\tA_2u+\tA_0\right)^2-A_1^2
 \left(\frac{1}4 u^2+ c_1 u + c_0\right).
 \]
 We can now apply the same ideas  as in the previous case. The product $I(x)J(x)$  vanishes
 \begin{itemize}
 \item 
   when $x$ is one of the two roots of $\Delta(x)$, denoted $X_+$ and $X_-$, that are finite at $t=0$,
   \item  when $x=X=t+3t^3+ \mathcal O(t^4)$ is the (only) root of $I(x)$ lying in $\qs[[t]]$, that is,
 \[
X= t\,\frac{1+X}{1+2t}\left(   2t (1+X) C_-(X)+1 +X+X^2 +t XC_{0,0}\right).
\]
 \end{itemize}
The values $I_0(X_\pm)$ are found to be the two roots of
\[
  P_2(u)= \PIID \left( 1-\PIID \right) ^{3}u^{2}+2\,\PIID \left( 1-\PIID \right) 
 \left( {\PIID}^{3}+{\PIID}^{2}+\PIID-1 \right) u- \left( {\PIID}^{2}+1
 \right)  \left( {\PIID}^{4}-{\PIID}^{3}+13\,{\PIID}^{2}-9\,\PIID+4 \right) 
 .
\]
(This polynomial should not be mixed up with the series $\UII$ of~\eqref{P2-char}.) They expand as
\[
I_0(X_\pm)=\pm \frac 2{\sqrt t} + 1 \pm \sqrt t +\LandauO(t).
\]
Now the series $I_0(X)=-1+\LandauO(t)$ is also a root of $P_4(u)$, distinct from $I_0(X_\pm)$, and in fact is even a double root of $P_4$ (as before, one has to check that $I'_0(X)\not = 0$). Hence we have the following factorization:
\[
  P_4(u)=\frac 1{\PIID\left( 1-\PIID \right) ^{3}} (u-I_0(X))^2 P_2(u).
\]
We now equate the coefficients of $u^3, \ldots, u^0$ in both sides of this identity. From the coefficients of $u^3$, we obtain
\[
  I_0(X)= \frac{\PIID^3+\PIID^2+\PIID-1}{(1-\PIID)^2}-\tA_2.
\]
We then replace every occurrence of $I_0(X)$ by this expression.  The coefficients of $u^2$ give
\beq \label{tA0sol-DK}
\tA_0= {\frac { \left( {\PIID}^{3}+{\PIID}^{2}+\PIID-1 \right) {\tA_2}}{
    \left( 1-\PIID \right) ^{2}}}
+\frac{{A_1}^{2}}8-{\frac {{\PIID}^{7}+4\,{\PIID}^{6}-3\,{\PIID}^
    {5}+12\,{\PIID}^{4}-15\,{\PIID}^{3}+10\,{\PIID}^{2}-5\,\PIID+2}{\PIID
    \left( 1-\PIID \right) ^{4}}}.
\eeq
We proceed with the coefficients of $u^1$, after replacing every occurrence of $\tA_0$ by the above expression. This gives
\beq \label{tA2sol-DK}
  \tA_2=
  -2\,{\frac { \left( {\PIID}^{3}-{\PIID}^{2}-\PIID-1 \right)  \left( 1-\PIID \right) ^
{4}A_1^{2}-16\, \left( {\PIID}^{3}+{\PIID}^{2}+\PIID-1 \right)  \left( {\PIID}^
{2}-\PIID+1 \right) ^{3}}{ \left( 1-\PIID \right) ^{2} \left( A_1^{2}\PIID
 \left( 1-\PIID \right) ^{4}+16\, \left( {\PIID}^{2}-\PIID+1 \right) ^{3}
 \right) }}
,
\eeq
and we finally derive by comparing constant terms that
\begin{multline}\label{algA1Z-DK}
    {\PIID}^{2} \left( 1-\PIID \right) ^{8}A_1^{4}
  +4\,\PIID \left( \PIID+1 \right) ^{3} \left( 1-\PIID \right) ^{7}A_1^{3}
  +32\,\PIID \left( {\PIID}^{2}-\PIID+1 \right) ^{3} \left( 1-\PIID \right) ^{4}A_1^{2}\\
  -64\, \left( 1-\PIID \right) ^{3} \left( \PIID+1 \right) ^{3} \left( {\PIID}^{2}-\PIID+1
 \right) ^{3}A_1+256\, \left( {\PIID}^{2}-\PIID   +1 \right) ^{6}
 =0.
\end{multline}
One can now check, using this equation and the first terms $A_1=4+4t^2+\LandauO(t^3)$
that $A_1$ is really the series described at the beginning of the section, and that the values of $\tA_2$ and $\tA_0$ given in the theorem in terms of the series $\PID$ and $A_1$ coincide with those derived from~\eqref{tA2sol-DK} and~\eqref{tA0sol-DK}.

It also follows from~\eqref{A12-DK} and~\eqref{tA2-def} that
\[
tC_{0,0}=tD_0=1-\frac{A_1}4 + \frac {\tA_2}2 ,
\]
and the expression~\eqref{C00-DK} can again be checked using~\eqref{tA2sol-DK} and~\eqref{algA1Z-DK}.

\medskip
\paragraph{{\bf Degrees.}} From the identities in  this theorem, it follows that $\C_-(x,0)$ and $D(x)$ have degree at most $16\times 4=64$ above $\qs(t,x)$. To prove that this bound is tight, we compute by elimination the minimal polynomials  of $C_-(2,0)$ and $D(2)$. We proceed similarly for  the degree of $C(x,y)$, for which the natural upper bound is $16\times 4\times 4=256$. In this case we not only specialize $x$ and $y$, but also $t$, in such a way that $xyt<1/6$, which guarantees the convergence of all series under consideration. 
\end{proof}

\begin{remark}
 It may be interesting for some readers to know how one can derive from the polynomial equation~\eqref{algA1Z-DK} the description of the series $A_1$, and of the subfields of $\qs(t,A_1)$, given at the beginning of the section (see Figure~\ref{fig:ext-DK}). This is fully documented in the \Maple\ session that accompanies this paper, but here are a few details.

A first step is to investigate the subextensions of $\qs(t,\PIID)$ and to discover the series $\PID$. This can be done, for instance, using the {\tt Subfields} command.

Then, it is easy to prove, by elimination of $\PID$, that $A_1$ has degree $16$ over $\qs(t)$.  The idea of introducing the generator $\UI$ comes by looking at the denominator of $\tA_2$ in~\eqref{tA2sol-DK}, and observing that the defining equation~\eqref{algA1Z-DK} of $A_1$ can also be written in the form:
\[
  \left(A_1^2\PIID(1 - \PIID)^4 + 16(\PIID^2 - \PIID + 1)^3\right)^2+
  4(1+\PIID )^3(1 - \PIID)^3A_1 \left(A_1^2\PIID(1 - \PIID)^4 - 16(\PIID^2 - \PIID + 1)^3\right)=0,
\]
or, in terms of $\PID$,
\[
  \left(\PID A_1^2 + 16 (1+\PID )^3\right)^2 + 4 A_1 (1+4 \PID )^{3/2}
  \left(\PID A_1^2 - 16 (1+\PID )^3\right)=0.
\]
Once $\UI$ is thus defined, one explores the subfields of $\qs(t,A_1)$ using the {\tt Subfields} command. \qee
 \end{remark}

\medskip

We now consider  number of walks ending at $(i,j)$.

\begin{proof}[Proof of Corollary~\ref{cor:excursions-DK}]
  This model is aperiodic. The right-hand side of~\eqref{C-DK} is a  (Laurent) series in $x$  with coefficients in $\qs(t,A_1)$, with first term $1/x^2$. So its square root  is a Laurent series in $x$, starting with $1/x$, with coefficients in $\qs(t,A_1)$ again. This proves the corollary for the coefficients of $U(x,0)$, that is, for the series $C_{i,0}$ with $i<0$. The value of $C_{-1,0}$ given in~\eqref{Cm10-D} comes directly from~\eqref{A12-DK}.

  A similar argument, based on~\eqref{D-DK}, proves the result for the series $C_{i,i}$.
  We then prove the corollary for the points $(i,i+1)$, by induction on $i\ge 0$, by writing
  \begin{align*}
    C_{i,i}&= \mathbbm1_{i=0}+ t \left( C_{i,i+1} + C_{i-1,i}+ C_{i-1,i-1} + C_{i,i-1} + C_{i+1,i}+ C_{i+1,i+1}\right)\\
    &=\mathbbm1_{i=0}+ t \left( 2C_{i,i+1} +2 C_{i-1,i}+ C_{i-1,i-1} + C_{i+1,i+1}\right).
  \end{align*}
  We finally prove it for all points $(i,j)$ with $j\ge i$ by induction on $2j-i\ge 0$, using
  \[
    C_{i,j-1}= t(C_{i,j-2}+C_{i,j}+C_{i-1,j-1}+ C_{i+1,j-1} +C_{i-1,j-2} +C_{i+1,j}).
    \]
\end{proof} 


\section{A D-algebraic model}
\label{sec:DA}

Let us now consider the $6$th model of the  Table~\ref{tab:sym}, with step polynomial $S(x,y)= x+\bx +y+\by +xy$. The companion model is given by $\tS(x,y)= x+\bx +xy +\bx \by +y$.
In this case there are no rational invariants for $\tS$, but there
exists a pair of quadrant invariants of the form~\eqref{I1J1-def}  (see Table~\ref{tab:inv-qu-2}), namely $(I_1(x), J_1(y))$, with
\beq \label{inv:DA}
  I_1(x) =  t \tQ(x,0)+ \bx, \qquad J_1(y) = -t(1+y)\tQ(0,y)+t\tQ_{0,0} + \frac{y(1-ty)}{t(1+y)} .
\eeq
This pair satisfies
\[
  I_1(x) -J_1(y) =- \tK(x,y) \left(xy\tQ(x,y)+\frac y{t(1+y)} \right) .
\]
The series $\tQ(x,y)$ is known to be D-algebraic, and admits an explicit expression  in terms of elliptic functions~\cite[Thm.~5.7]{BeBMRa-17}. Our approach relates $C(x,y)$ to this series.

\subsection{Generating functions}

\begin{theorem}\label{thm:DA}
Let $I_1(x)$ and $J_1(y)$ be defined as above.   The \gf\  $C(x,y)$  of walks with steps in $\{\rightarrow, \nearrow, \uparrow, \leftarrow, \downarrow \}$ starting from $(0,0)$ and avoiding the negative quadrant  is given by the following equation:
\[
  (1-t(x+y +\bx+\by+xy)) C(x,y)=1 -t\by C_-(\bx) -t\bx C_-(\by),
\]
where  $\C_-(x)$ can be expressed in terms of $I_1$ and $J_1$ by
\beq\label{Cm-expr-DA}
  \left( t \bx C_-(x) +x+\bx-\frac 1{2t} \right) ^{2}
  = \frac{ (I_1(x)- A)^2\left( {I_1(x)}- {J_1(Y)}\right) }{I_1(x)}.
\eeq
In the above expression, $Y$ is the only root of 
\beq\label{Delta-DA}
  \Delta(y):=(1-ty)^2-4t^2\by(1+y)^2
\eeq
that is a formal power series in $t$. Then, 
\[
A=\sqrt{\frac{1-t^2\tQ_{0,0} -t^2\tQ_{0,1}}{tJ_1(Y)}}.
\]

The \gf\ $D(y)$ of walks ending on the diagonal satisfies
\beq\label{Dy-expr-DA}
  \frac{\Delta(y) }4 \left( yD \left( y  \right) +
       \frac {1-y}{t (1+y)}  \right) ^{2}
  = \frac{ (J_1(y)- A)^2\left( J_1(y)- {J_1(Y)}\right) }{J_1(y)}.
\eeq
In particular, $C_-(x)$, $D(y)$ and $C(x,y)$ are D-algebraic in all variables.
\end{theorem}

      \begin{proof}
         The invariants of Proposition~\ref{prop:inv-tq} are
\beq\label{IJ-DA}       
          I(x)= \left( 2\,tU \left( x,0 \right) +2\,x+2\bx-\frac 1{t} \right) ^{2},
          \qquad
          J(y)=\Delta(y)  \left( yD \left( y\right)
            +{\frac {1-y}{t \left(1+ y \right) }}
          \right) ^{2},
\eeq
where $\Delta(y)$ is defined by~\eqref{Delta-DA}.
We want to combine them polynomially with the invariants $(I_1(x), J_1(y))$ given by~\eqref{inv:DA} to form a pair of invariants $(\tilde I(x), \tilde J(y))$ to which Lemma~\ref{lem:invariants} would apply, thus proving that $\tilde I(x)= \tilde J(y)$ is a series in $t$.

Observe that $J(y) $ has a pole at $y=0$ (coming from $\Delta(y)$) and
at $y=-1$. The modified
invariant $\tilde J(y)$ should have no poles. Observing that
$J_1(0)=0$, we first form the product $J(y) J_1(y)$ to remove the pole
of  $J(y)$ at $0$. The resulting series now has a triple pole at $-1$, which can be removed by subtracting a cubic polynomial in $J_1(y)$. This leads us to introduce
\[
  \tilde J(y) = J(y) J_1(y)- 4J_1(y)^3 - A_2 J_1(y)^2 - A_1 J_1(y),
\]
and to define accordingly
\[
  \tilde I(x) = I(x) I_1(x)- 4I_1(x)^3 - A_2 I_1(x)^2 - A_1 I_1(x),
\]
for values of $A_2$ and $A_1$ that involve $D(-1), D'(-1)$ and various by-products of $\tQ(x,y)$ (we refer to our \Maple\ session for details).
Then, we use the functional equations satisfied by $U(x,y)$, $D(y)$ and $\tQ(x,y)$ to check that the ratio
\[
  \frac{\tilde I(x) -\tilde J(y)}{\tK(x,y)}
\]
is a multiple of $xy$, and conclude that $ \tilde I(x)=  \tilde J(y)=A_0$, for some series $A_0\in \qs((t))$. Hence, denoting $P(u)= 4u^3+A_2u^2+A_1u+A_0$, we have
\beq\label{IJ-DA-form}
  I(x) I_1(x)= P(I_1(x)), \qquad   
  J(y) J_1(y)= P(J_1(y)).
\eeq
Since the left-hand sides of~\eqref{Cm-expr-DA} and~\eqref{Dy-expr-DA} in Theorem~\ref{thm:DA} are $I(x)/4$ and $J(y)/4$, the proof will be complete if we can prove that
\[
  P(u)= 4(u-A)^2(u-J_1(Y)).
\]
We will identify the roots of $P(u)$ by looking at the values of $y$ that cancel $J(y)$, and thus $P(J_1(y))$. First, $\Delta(y)$ has a (unique) solution $Y=\LandauO(t^2)$ in $\qs[[t]]$. Thus $J_1(Y)=4t+\LandauO(t^3)$ must be a root of $P$. Then, another  root of $J(y)$ arises from the factor of $J(y)$ involving $D(y)$. Indeed, the (unique) series $Y_0=1+2t+\LandauO(t^2)$ such that $Y_0=1+tY_0(1+Y_0)D(Y_0)$ is a double root of $J(y)$, and thus $J_1(Y_0)=1/(2t)+\LandauO(1)$ is a double root of $P(u)$, unless $J_1'(Y_0)=0$, which we readily check not to be the case. 
At this stage we can write
\[
  P(u)=4\left(u-J_1(Y_0)\right)^2\left(u-J_1(Y)\right).
\]
But we would like to make $A:=J_1(Y_0)$ more explicit. We expand around $y=0$ the second identity of~\eqref{IJ-DA-form}, and obtain
\[
  tJ_1(Y) J_1(Y_0)^2= 1-t^2\tQ_{0,0} -t^2\tQ_{0,1} ,
\]
which concludes our derivation after choosing the correct (positive) sign for the square root.

The fact that all series under consideration are D-algebraic comes from the closure properties of D-algebraic series~\cite[Sec.~6.1]{BeBMRa-17}, and from the  fact that $\tQ(x,y)$ is  D-algebraic~\cite[Sec.~6.4]{BeBMRa-17}.
\end{proof}

We can now specialize the first two equations of Theorem~\ref{thm:DA} to relate the series $C(1,1)$ to the quadrant series $\tQ$.

\begin{cor} 
  The \gf\ $C(1,1)$ counting  walks with steps in $\{\rightarrow, \nearrow, \uparrow, \leftarrow, \downarrow \}$ that start from $(0,0)$ and avoid the negative quadrant  is given by
  \beq\label{C11-fork}
    \frac 1 4 (1-5t)^2 \left( C(1,1)+\frac 1 t \right)^2 = \left( tC_-(1)+2 - \frac 1 {2t}\right)^2= \frac{(I_1(1)-A)^2(I_1(1)-J_1(Y))}{I_1(1)},
  \eeq
  where $I_1(x)$ and $J_1(y)$ are given by~\eqref{inv:DA}, and $Y$ and $A$ are defined in Theorem~\ref{thm:DA}; in particular, $I_1(1)= 1+t\tQ(1,0)$.
\end{cor}

\subsection  {An expression in terms of a \emm weak, invariant}
Instead of using the quadrant-related invariants $I_1$ and $J_1$ of~\eqref{inv:DA}, one can use the analytic approach of~\cite[Sec.~5]{BeBMRa-17} and express
$D(y)$ (say) in terms of an explicit \emm weak, invariant $w(y)$ (expressed itself in terms of elliptic functions). Let us
give a few details, using the notation of~\cite{BeBMRa-17}, which we
will not re-define. First, the model denoted $\tS$ here is model $\#5$ in~\cite[Tab.~5]{BeBMRa-17}. Then, for $t$ a small positive real number, the series $D(y)$ can be defined analytically in the domain $\mathcal G_{\mathcal L}$, with finite limits on the boundary curve $\mathcal L$~\cite[Lem.~5]{RaTr-19}. This curve is bounded because $\tS$ contains a West step~\cite[Lem.~5.2]{BeBMRa-17}. The points $0$ and $-1$ both belong to $\mathcal G_{\mathcal L}$, by~\cite[Lem.~5.2 and Lem.~5.8]{BeBMRa-17}. Hence  the invariant $J(y)$ defined by~\eqref{IJ-DA}, which is a simple variation on $D(y)$, is meromorphic in $\mathcal G_{\mathcal L}$, with finite limits on $\mathcal L$, and exactly two poles in $\mathcal G_{\mathcal L}$, at $0$ and $-1$. The first pole is simple, the second is double. Thus $J(y)$ is a \emm weak invariant, for the model $\tS$, in the sense of~\cite[Def.~5.4]{BeBMRa-17}. We can  apply the (analytic) invariant lemma~\cite[Lem.~5.6]{BeBMRa-17}: there exists a polynomial $P(u)$  (with complex coefficients that depend on the value $t$), of degree at most $3$,  such that
\[
  J(y)= \Delta(y)  \left( y D(y) + \frac{1-y}{t(1+y)}\right)^2 = \frac{P(w(y))}{(w(y)-w(0))(w(y)-w(-1))^2}.
\]
As before, we can obtain some information on the roots of $P(u)$ by examining the roots of $J(y)$. First, the formal power series $Y=\LandauO(t^2)$ that cancels $\Delta(y)$ specializes to the value denoted $y_2$ in~\cite[Lem.~5.1]{BeBMRa-17} (the value $y_1$ is zero for this model). This is precisely the (unique) pole of $w(y)$ in $\mathcal G_{\mathcal L}$, by~\cite[Prop.~5.5]{BeBMRa-17}, and we conclude that $P(u)$ has  degree~$2$ at most. The series $Y_0=1+2t+\LandauO(t ^3)$ that cancels the factor of $J(y)$ involving $D(y)$  can also be shown to specialize to a value $y_0$ of $\mathcal G_{\mathcal L}$ (because the positive point where the curve $\mathcal L$ intersects the real axis, denoted $Y(x_2)$ in~\cite{BeBMRa-17}, is  $1/\sqrt t$ at first order, so that $y_0$ is smaller), so that $P(u)$ admits $w(y_0)$ as a double root.
Finally,
\[
J(y)=  \Delta(y)  \left( y D(y) + \frac{1-y}{t(1+y)}\right)^2 = \frac{\alpha (w(y)-a)^2}{(w(y)-w(0))(w(y)-w(-1))^2}
\]
with $a=w(y_0)$. One can determine $\alpha$ and $a$ in terms of $w(0)$, $w(-1)$, $w'(0)$ and $w'(-1)$ by expanding this identity at first order around $y=0$ and $y=1$.

Alternatively, the above identity  can  be derived by combining  the expression of $J(y)/4$ given in~\eqref{Dy-expr-DA}  in terms of $J_1(y)$ (or equivalently, in terms of $\tQ(0,y)$) with the following expression, borrowed from~\cite[Thm.~5.7]{BeBMRa-17}:
\[
  y(1+t) \tQ(0,y)=-y-\frac{1+t}{t(1+y)} -1 +\frac{1+t}t \left( \frac {w'(-1)}{w(y)-w(-1)} + \frac{w''(-1)}{2w'(-1)}\right).
\]
Note that this gives in particular the following simple expression:
\begin{align*}
  J_1(y)&=-t(1+y)\tQ(0,y)+t\tQ_{0,0} + \frac{y(1-ty)}{t(1+y)}\\
  &=\frac{1+t} t \cdot \frac{w'(-1)}{w(0)-w(-1)} \cdot \frac{w(y)-w(0)}{w(y)-w(-1)},
\end{align*}
and, since $Y=y_2$ is a pole of $w(y)$, 
\[
  J_1(Y)=\frac{1+t} t \cdot \frac{w'(-1)}{w(0)-w(-1)} .
  \]

  It is interesting to note that, while  quadrant \gfs, when they can be expressed  in terms of $w(y)$, are in fact homographic functions of $w(y)$, the above expression of $J(y)$ has degree $3$ in $w(y)$.  Moreover, the three-quadrant \gf\ $D(y)$ is no longer rational, but algebraic in $w(y)$. This is another sign of the higher difficulty of three-quadrant problems. \qee

  \subsection{Harmonic functions}
  As already sketched in Remark~\ref{rem:invariants-harmonic} in the Kreweras case, we can relate the asymptotic behaviours of $\cS$-walks in $\Cc$ and $\tS$-walks in $\Qc$, under highly plausible assumptions.

Let us begin with quadrant $\tS$-walks. The number $\tilde q_{i,j}(n)$ of $\tS$-walks of length $n$ in the quadrant $\Qc$ ending at $(i,j)$ is known to have the following asymptotic form~\cite{BoRaSa-14,denisov-wachtel}:
\[
  \tilde q_{i,j}(n) \sim \frac{h_{i,j} }{\Gamma(-\alpha)} \mu^n n^{-1-\alpha},
\]
where $\mu\simeq 4.729$ is the positive solution of $\mu^3+\mu^2-18\mu-43=0$, and
$\alpha= \pi/\arccos(-c)\simeq 1.39$, where $c\simeq 0.626$ satisfies $64\,{c}^{6}-64\,{c}^{4}+28\,{c}^{2}-5=0$. 
 The numbers $h_{i,j}$ satisfy  the harmonic relation
  \[
    h_{i,j}= \frac 1 \mu \left(h_{i-1,j}+h_{i+1,j}+h_{i,j-1}+h_{i-1,j-1}+h_{i+1,j+1} \right),
  \]
  from which one derives the functional equation
\beq\label{eq-H-DA-comp}
    \left(1+y+xy^2+x^2y+x^2y^2-\mu xy\right) \tH(x,y)=\tH(x,0) +(1+y) \tH(0,y) -\tH(0,0),
    \eeq
    where $  \tH(x,y)=\sum_{i,j \ge 0} h_{i,j} x^i y^j$.
    As discussed in~\cite[Sec.~6]{raschel-harmonic},    no simple expression is known for $ \tH(x,y)$ (the results of~\cite{raschel-harmonic} only hold for walks with no drift; the same is true for the results of~\cite{trotignon-harmonic} on harmonic functions  in $\Cc$).

The number $ c_{i,j}(n)$ of $\cS$-walks of length $n$ in the three-quadrant cone $\Cc$ ending at $(i,j)$ is widely believed  to have the following asymptotic form:
\[
  c_{i,j}(n) \sim- \frac{H_{i,j} }{\Gamma(-\alpha/2)} \mu^n n^{-1-\alpha/2},
\]
with  $\mu$ and $\alpha$  as above.
More precisely, what is known is a pair of lower and upper bounds on $c_{i,j}(n)$ that have the above form and differ simply by a multiplicative constant~\cite{mustapha-3quadrant}. The minus sign has been chosen so that $H_{i,j}>0$.

If this holds, then
\[
  H_{i,j}=H_{j,i}= \frac 1 \mu \left(H_{i-1,j}+H_{i+1,j}+H_{i,j-1}+H_{i,j+1}+H_{i-1,j-1} \right),
\]
which yields for the \gf\ $\Hc(x,y) = \sum_{j \ge 0, j\ge i} H_{i,j} x^{j-i} y^j$ the functional equation
\beq\label{eq-Hc-DA}
    \left(1+y+xy^2+x^2y+x^2y^2-\mu xy\right) \Hc(x,y)=\Hc_-(x)+\frac 1 2 \left( 2+2y+ xy^2-\mu xy\right) \Hc_d(y)
  \eeq  
where as before $\Hc_-(x)=\sum_{i>0} H_{-i,0} x^i$ and $\Hc_d(y)=\Hc(0,y)$.

  We can now adapt the argument of Remark~\ref{rem:invariants-harmonic}. The term $c(x,y)= \left( 2+2y+ xy^2-\mu xy\right)/2$ satisfies
  \begin{align*}
    c(x,y)^2 &=\frac{\delta(y)}4 x^2 -\mu xy (1+y) \tK(x,y;1/\mu) \\
    & = \frac{\delta(y)}4 x^2+ (1+y) \left(1+y+xy^2+x^2y+x^2y^2-\mu xy\right),
  \end{align*}
  where $\delta(y)= (\mu^2y-2\mu y^2+y^3-4 y^2-8 y-4) y$ is the discriminant of $\mu xy  \tK(x,y;1/\mu) $ with respect to $x$. Upon multiplying~\eqref{eq-Hc-DA} by $\bx^2(\Hc_-(x) - c(x,y) \Hc_d(y))$, we conclude that the series
  \[
    \btH(x,y) :=  \bx^2\Hc \left( x,y \right) \left({ {\Hc_- \left( x \right)-  c(x,y)\Hc_d \left( y \right) } } \right) 
    + \bx^2  { \left( 1+y \right)   \Hc_d\left( y \right)   ^{2}}
  \]
  is a \fps \ in $x$ and $y$, which satisfies the same equation~\eqref{eq-H-DA-comp} as $\tH(x,y)$, as well as
  \[
    \btH(x,0)= \bx^2 \Hc_-(x)^2, \qquad \text{and} \qquad
    \btH(0,y)= -\frac{\delta(y)}{4(1+y)}\Hc(0,y)^2  + \frac{H_{-1,0}^2}{1+y} .
  \]
  In particular, $\btH(0,0)=H_{-1,0}^2$. Let us now assume that
  \begin{itemize}
  \item $\tH(x,y)$ is uniquely determined  (up to a multiplicative factor) by Equation~\eqref{eq-H-DA-comp} and the positivity of its coefficients,
    \item the series $\btH(x,y)$ has positive coefficients.
  \end{itemize}
  Then  there exists a positive constant $\kappa^2$ such that $\btH(x,y)=\kappa^2 \tH(x,y)$, which implies in particular that
 \beq \label{prediction}
 \Hc_-(x) =\kappa x \sqrt{\tH(x,0)}  \qquad \text{and} \qquad
 \Hc_d(y) =    2\kappa \sqrt{-\frac{(1+y) \tH(0,y)-\tH(0,0)}{\delta(y)}}.
  \eeq
Expanding these predictions at $x=0$ and $y=0$ gives
\[ 
    \kappa = \frac{H_{-1,0}}{\sqrt{h_{0,0}}}= \frac{H_{0,0}}{\sqrt{h_{0,0}+h_{0,1}}},
\]
so that we expect
\[
      \frac{H_{0,0}}{H_{-1,0}} = \sqrt{1 + \frac{h_{0,1}}{h_{0,0}}}.
    \]
This seems to hold, as shown by Figure~\ref{fig:plot}, where we have plotted against $1/n$ the sequences
    \beq\label{sequences}
    \frac{c_{0,0}(n)}{c_{-1,0}(n)} \qquad \text{and} \qquad
     \sqrt{1 + \frac{\tilde q_{0,1}(n)}{\tilde q_{0,0}(n)}}.
    \eeq
    We have checked numerically more consequences of the prediction~\eqref{prediction}, like
    \[
      \frac{H_{-1,0}}{H_{-2,0}}= 2\, \frac{h_{0,0}}{h_{1,0}}
      \qquad \text{and} \qquad
     \frac{H_{1,1}}{H_{0,0}}=    \frac{\mu^2} 8 -1 +\frac 1 2 \cdot \frac{h_{0,1}+h_{0,2}}{h_{0,0}+h_{0,1}}, 
    \]
    each time with perfect agreement.

      \begin{figure}[htb]
        \centering
        \includegraphics[height=40mm]{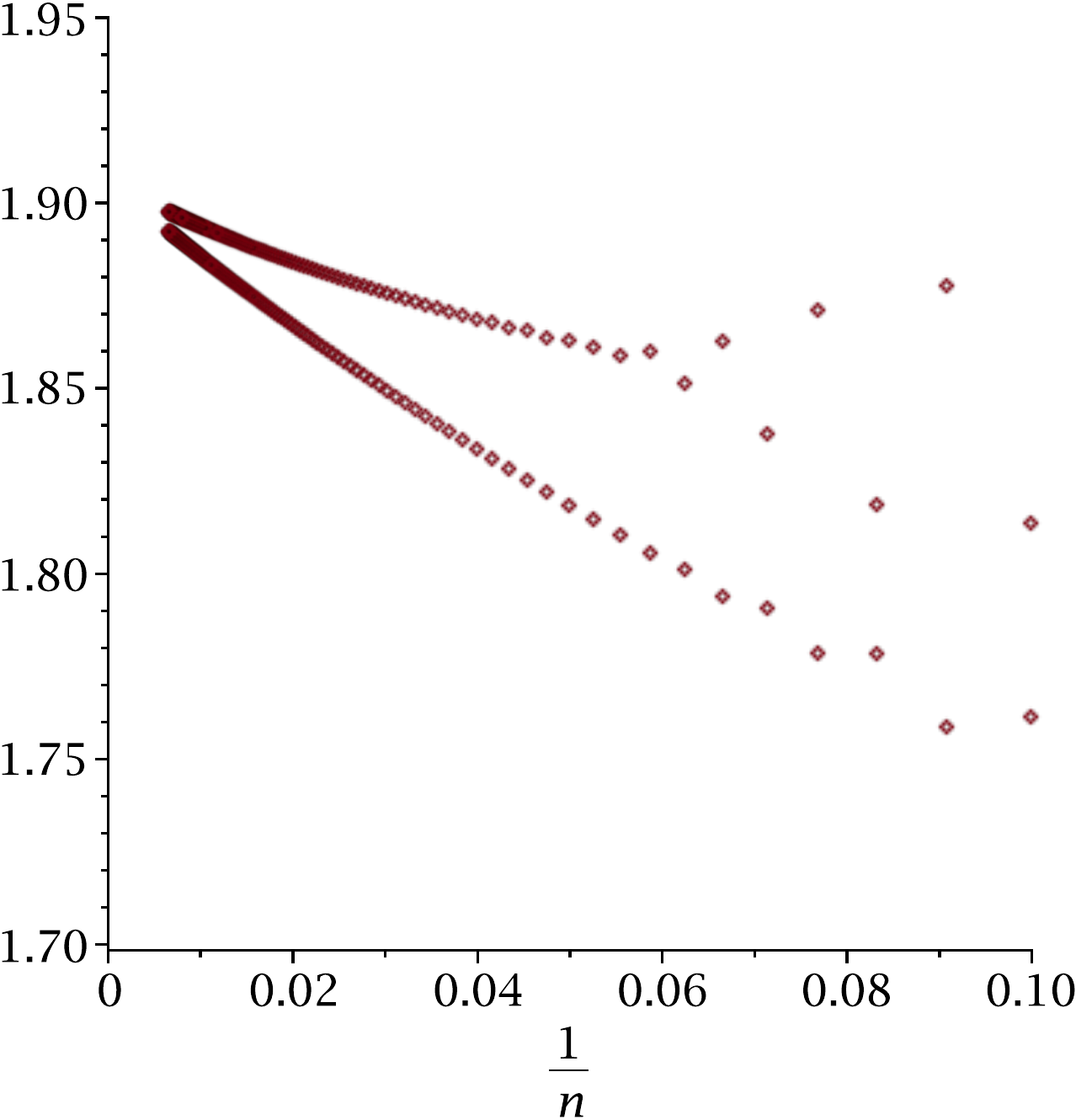}
        \caption{Plots of the sequences~\eqref{sequences} against $1/n$, for $n\le 150$.}
        \label{fig:plot}
      \end{figure}

\section{The simple and diagonal models}
\label{sec:DF}

In this section, we explain how the use of invariants also give new solutions of three-quadrant walks when $\cS_1=\{\rightarrow, \uparrow, \leftarrow, \downarrow\}$ (simple model) and $\cS_2=\{\nearrow, \nwarrow, \swarrow, \searrow\}$ (diagonal model). This may seem unexpected, for the following reason.

For both models, the basic functional equation of
Lemma~\ref{lem:func_eq} holds (with a slightly different definition of $U$ and $D$ for the model $\cS_2$); but by
Proposition~\ref{prop:dec-three-quadrants}, there is no way to decouple~$y$ in the desired form. Moreover, if we could still construct a pair of invariants involving $U(x,0)$, and relate it polynomially to the invariants $(I_1(x), J_1(y))$ that involve quadrant \gfs\  with steps in $\tS_1$ or $\tS_2$ (the so-called Gessel model for $\tS_1$, or its reflection in the main diagonal for $\tS_2$), then $U(x,0)$ would be algebraic, because quadrant walks with steps in $\tS_1$ (or $\tS_2$) \emm are, algebraic. But it is known that $U(x,0)$ is transcendental for both models~\cite{BM-three-quadrants}.

  However, for both models ($\cS=\cS_1$ or $\cS_2$), it is natural~\cite{BM-three-quadrants,mbm-wallner} to introduce the series $A(x,y)$ defined by
  \beq\label{A-def}
    A(x,y)= C(x,y) - \frac 1 3 \left(Q(x,y) -\bx^2 Q(\bx,y) - \by^2
Q(x,\by) \right),
\eeq
where $Q(x,y)$ counts quadrant $\cS$-walks. The series $A(x,y)$ is easily shown to satisfy the same functional equation~\eqref{eqfunc-gen} as $C(x,y)$, but with the initial term $1$ replaced by $ (2+\bx^2+\by^2)/3$. And then the heart of~\cite{BM-three-quadrants} is to prove that $A(x,y)$ \emm is, algebraic, for both models.

In this section, we show how to determine the series $A(x,y)$ using
invariants, rather than the method of~\cite{BM-three-quadrants}. More
precisely,  we  establish a relation between $A(x,y)$ and the \gf\ of quadrant walks with Gessel's steps.
We then derive from this the harmonic functions $H_{i,j}$ for the
simple and diagonal models in the three-quadrant plane.  The \gfs\ $\sum _{(i,j)\in \cC} H_{i,j} x^i y^j$ are algebraic in both cases, even though the \gfs\ $C(x,y)$ are D-finite but transcendental. We do not work out more exact results (for instance the degree of $A(x,y)$ or of the series $A_{i,j}$), as such results already appear in~\cite{BM-three-quadrants}.

\subsection{The simple model}
\label{sec:simple}
When $\cS=\{\rightarrow, \uparrow, \leftarrow, \downarrow\}$, the series $A(x,y)$ defined by~\eqref{A-def} satisfies
\beq\label{A-eq-simple}
(1-t(x+\bx+y+\by))A(x,y)= (2+\bx^2+\by^2)/3-t\by A_-(\bx) -t \bx A_-(\by) ,
\eeq
where the series $A_-(\bx)$ is defined in a similar fashion as $C_-(\bx)$; see~\eqref{Chv}.
The series $A(x,y)$ has a simple combinatorial interpretation: it counts weighted walks in the three-quadrant cone~$\Cc$, starting either from $(0,0)$, or from $(-2,0)$, or from $(0,-2)$, with a weight $2/3$ in the first case, $1/3$ in the other two. We now define two series $U(x,y)$ and $D(y)$ by
\beq\label{A-UD-simple}
  A(x,y) = \bx U(\bx,xy)+ D(xy)+ \by U(\by,xy).
\eeq
Note that they do \emm not, have  the same combinatorial meaning as in the previously solved four models.  We can reproduce the step-by-step arguments that led to Lemma~\ref{lem:func_eq},
and we thus obtain
\beq\label{eqU-simple}
  2 xy\tK(x,y)U(x,y)= \frac 2 3 y(1+x^2) +y\left(2tx(1+y)-1\right)D(y)-2 tU(x, 0),
\eeq
where $\tK(x,y)=1-t\tS(x,y)=1-t(x+\bx +xy+\bx\by)$ is the step polynomial of Gessel's model. Compared to the original functional equation~\eqref{eq-U}, the only thing that has changed is the initial term $y$, which has become $2 y(1+x^2) /3$.

For the model $\tS$, we  know two pairs of invariants (see Table~\ref{tab:inv-qu-1}): one is rational, and will not be used; the other takes the form~\eqref{I1J1-def} and involves quadrant \gfs:
\beq\label{I1J1-gessel}
  I_1(x)=  t\tQ(x, 0)+ \bx , \qquad
  J_1(y) =-t(1+y)\tQ(0, y)+t\tQ_{0, 0} +\frac y{t(1+y)}.
\eeq
These two series are algebraic, and we refer to~\cite{BoKa-10,mbm-gessel} for explicit expressions. They satisfy
\[
 I_1(x)-J_1(y)= -y\tK(x,y) \left(x \tQ(x,y) + \frac 1 {t(1+y)}\right)  .
\]
We will  obtain expressions for $U(x,0)$ and $D(y)$ in terms of  $I_1(x)$ and $J_1(y)$, which are reminiscent of those of the previous section (Theorem~\ref{thm:DA}). Given that a polynomial identity relates $(I_0(x),J_0(y))$ and $(I_1(x),J_1(y))$, of degree $3$ in the latter pair, we could also decrease to~$2$ the degree in $I_1(x)$ and $J_1(y)$ of our expressions, upon introducing $I_0(x)$ and $J_0(y)$.

\begin{theorem}\label{thm:simple}
Let $I_1(x)$ and $J_1(y)$ be defined as above.   The \gf\  $A(x,y)$ defined by~\eqref{A-def} satisfies  the  equation~\eqref{A-eq-simple}, where  $A_-(x)$ can be expressed in terms of $I_1$ and $J_1$ by
\beq\label{Am-simple}
  \left(3 t \bx A_-(x) +1+\bx^2-\frac\bx t \right) ^{2}
  = I_1(x)\, \big(I_1(x)- B\big)^2\big(I_1(x)- J_1(Y)\big).
\eeq
In the above expression,  $Y$ is the only root of
\beq\label{Delta-simple}
  \Delta(y)= 1-4t^2\by (1+y)^2
\eeq
that is a formal power series in $t$, and
\[
B=\frac 1 t + 2 t \tQ_{0,0} - \frac{J_1(Y)}2.
\]

The \gf\ $D(y)$ defined by~\eqref{A-UD-simple} satisfies
\beq \label{Dy-simple}
\frac 94 {\Delta(y) }          \left( yD ( y) +\frac {2y}{3t^2(1+y)^2} \right) ^{2}
= J_1(y)\, \big(J_1(y)- B\big)^2\big(J_1(y)- J_1(Y)\big).
\eeq
In particular, the series $A_-(x)$, $D(y)$ and $A(x,y)$ are algebraic.
\end{theorem}

\begin{remark}
  \label{rem:degrees-simple}
   {\bf  Degrees and rational parametrizations of $A_-(x)$ and $D(y)$.}
In~\cite{mbm-gessel}, the series $\tQ(xt,0)$  and $\tQ(0,y)$ (involved in the expressions of  $I_1(xt)$ and  $J_1(y)$, respectively) are expressed as rational functions  in four algebraic series denoted $T$, $Z=\sqrt T$, $U\equiv U(x)$ and $V\equiv V(y)$.
Here~$T$ has degree~$4$ over $\qs(t)$, while $U(x)$ is cubic over $\qs(x, T)$ and $V(y)$ is cubic over $\qs(y,T)$. All these series are series in $t^2$. Moreover, $t^2\in \qs(T)$, $x\in \qs(T,U)$ and $y\in \qs(T,V)$.  One can derive from these results rational expressions of $I_1(xt)/t$ and $J_1(y)/t$ in terms of $\sqrt T$, $U$ and $V$, and well as the following simple expressions:
  \[
    J_1(Y)=\frac{256 t  T \sqrt{T} }{(T+3)^3}, \qquad B= {\frac {64t\,T\sqrt{T} \left( {T}^{2}+4\,T-1 \right) }{ \left( T-1
 \right)  \left( T+3 \right) ^{3}}}.
\]
 Details of these calculations are given in our \Maple\ session. Using~\eqref{Am-simple} (resp.~\eqref{Dy-simple}), we then obtain a rational expression of
\[
  \left( 3 \bx A_-(xt) +1 + \frac{\bx^2}{t^2}-\frac{\bx}{t^2}\right)^2
  \qquad \qquad \left( \text{resp.} \quad
  \left( y D(y) + \frac{2y}{3t^2(1+y)^2} \right)^2\right)
\]
in terms of $T$ and $U$ (resp.~$T$ and $V$), showing that this series has degree $12$ only (as $\sqrt T$ is not involved here). Moreover, this rational expression is of the form $T \Rat(T,U)^2$ (resp.~$T \Rat(T,V)^2$) for some rational function $\Rat$, so that
\[
  3 \bx A_-(xt) +1 + \frac{\bx^2}{t^2}-\frac{\bx}{t^2}
  \qquad \qquad \left( \text{resp.} \quad
 y D(y) + \frac{2y}{3t^2(1+y)^2} \right)
\]
belongs to $\sqrt T \qs(T,U)$ (resp. $\sqrt T \qs(T,V)$), hence to the same extension of degree $24$ of $\qs(t,x)$ (resp.~$\qs(t,y)$) as $\tQ(xt,0)$ (resp.~$\tQ(0,y)$).
 This is in contrast with the  Kreweras-like models solved in Sections~\ref{sec:K} to~\ref{sec:DK}, for which the degree of $C_-(x)$ and $D(y)$ is twice the degree of $\tQ(x,0)$ (or $\tQ(0,y)$). \qee
\end{remark}

\begin{proof}[Proof of Theorem~\ref{thm:simple}.]
 We start from the functional equation~\eqref{eqU-simple}, and observe that  we have  a decoupling relation:
\[
y  (1+x^2) = (2tx(1+y)-1)G(y)+F(x) + H(x,y) \tK(x,y),
\]
where
\[
  F(x)=-1-\bx^2+ \frac 1{tx}, \qquad G(y)= \frac{y}{t^2(1+y)^2}, \qquad H(x,y)=\frac {y\left(1-t(x+\bx+\bx y+xy)\right)}{t^2(1+y)^2}.
\]
This, together with Lemma~\ref{lem:square}, leads us to define a new pair of $\tS$-invariants $(I(x),J(y))$:
\[ 
   I(x)= \left(2tU(x,0) + \frac 2 3 \left( 1+\bx^2- \frac\bx t\right)\right)^2,
\qquad  \qquad 
  J(y)= \Delta(y) \left(yD(y) + \frac{2y}{3t^2(1+y)^2}\right)^2,
\] 
where $\Delta(y)$ is given by~\eqref{Delta-simple}. We want to combine this new pair with the pair $(I_1(x), J_1(y))$ given by~\eqref{I1J1-gessel} to form a pair of invariants $(\tilde I(x), \tilde J(y))$ to which Lemma~\ref{lem:invariants} would apply, thus proving that $\tilde I(x)= \tilde J(y)$ is a series in $t$.

Observe that $I_1(x)$ and $I(x)$ have poles at $x=0$, while $J_1(y)$ and $J(y) $ have poles at  $y=-1$ (only). More precisely, the series $I(x)$  has a (quadruple) pole at $0$, which can be removed by subtracting a well-chosen quartic polynomial in $I_1(x)$. This leads us to introduce
\[
  \tilde I(x) = I(x) - \frac 4 9 I_1(x)^4-B_3 I_1(x)^3 - B_2 I_1(x)^2 - B_1 I_1(x),
\]
and analogously
\[
  \tilde J(y) = J(y) - \frac 4 9 J_1(y)^4 - B_3 J_1(y)^3 - B_2 J_1(y)^2 - B_1 J_1(y),
\]
for values of $B_3$, $B_2$ and $B_1$ that involve $U_{0,0}, U_{1,0}$ and various by-products of $\tQ(x,y)$ (we refer to our \Maple\ session for details; we have denoted the auxiliary series $B_i$ instead of~$A_i$ to avoid any confusion with the series $A(x,y)$ or $A_-(x)$).
Then, we use the functional equations satisfied by $U(x,y)$, $D(y)$  and $\tQ(x,y)$ to check that the ratio
\[
  \frac{\tilde I(x) -\tilde J(y)}{\tK(x,y)}
\]
is a multiple of $xy$, and conclude that $ \tilde I(x)=  \tilde J(y)=B_0$, for some series $B_0\in \qs((t))$. By expanding $\tilde J(y)$ at $y=0$, we realize that $B_0=0$. Hence, denoting $P(u)= 4u^3/9+B_3u^2+B_2u+B_1$, we have
\beq\label{IJ-simple-form}
  I(x)= I_1(x)P(I_1(x)) \qquad    \text{and} \qquad
  J(y)= J_1(y)P(J_1(y)).
\eeq
We can express the roots of $P$ in terms of the series $\tQ$ by looking at the formal power series $X\in \qs[[t]]$ or $Y\in \qs[[t]]$ that cancel $I(x)$ or $J(y)$. First, $\Delta(y)$ has a unique solution in $\qs[[t]]$, which we denote by $Y=4t^2+\LandauO(t^4)$. Thus $J_1(Y)=4t+\LandauO(t^3)$  must be a root of $P$. Then, another  root of $P$ arises from the  series $X=t+\LandauO(t^3)$ such that $I(X)=0$, that is,
$X=3t^2U(X, 0)X^2+tX^2+t$. This series is a double root of $I(x)$, and thus $I_1(X)=1/t+\LandauO(t^3)$ is a double root of $P(u)$, unless $I_1'(X)=0$, which we readily check not to be the case. 
      At this stage we can write 
      \[
      P(u)=\frac 4 9 \big(u-J_1(Y)\big)\big(u-I_1(X)\big)^2,
      \]
but we still have to relate the double root $B:=I_1(X)$ to~$\tQ$. We expand around $x=0$ the first identity of~\eqref{IJ-simple-form}, and obtain the expression of the root $B$ given in the statement of the theorem.

The expressions~\eqref{Am-simple} and~\eqref{Dy-simple}, which are those of $9I(x)/4$ and $9J(y)/4$, are now proved.
\end{proof}

Here is now  our description of the harmonic function associated with simple walks in $\Cc$. 

\begin{cor}\label{cor:harmonic-simple}
  For $(i,j) \in \cC$, there exists a positive constant $H_{i,j}$ such that, as $n\rightarrow \infty$ with $n\equiv i+j \mod 2$,
  \beq\label{cij-est-simple}
  c_{i,j}(n) \sim -\frac{H_{i,j}}{\Gamma(-2/3)}\, 4^n n^{-5/3}.
\eeq
The \gf\
\[
  \Hc(x,y):=\sum_{j\ge 0, i\le j} H_{i,j} x^{j-i} y^j,
\]
which is a \fps\ in $x$ and $y$, is algebraic of degree $9$, given by
\[ 
  \left(1+y+x^2y+x^2y^2-4xy\right)  \Hc(x,y)=  \Hc_-(x)+(1+y-2xy) \Hc_d(y),
\] 
where
\[
  \Hc_-(x):=\sum_{i>0} H_{-i,0} x^i\qquad \text{and} \qquad
   \Hc_d(y):=\sum_{i\ge 0} H_{i,i}y^i.
\]
Each of these series is algebraic of degree $3$. Let us define $L\equiv L(x)=\sqrt 3+\LandauO(x)$ and $P\equiv P(y)=1/3+\LandauO(y)$ by
\[
  x=9\,{\frac {3-{L}^{2}}{ \left( 2\,L-3 \right)  \left( {L}^{2}-12\,L+9
 \right) }}, 
\qquad
y={\frac {1-3\,P}{{P}^{2} \left( P-3 \right) }}.
\]
Then
\beq\label{Hminus-simple}
  \Hc_-(x) =
   {\frac {128\,\sqrt {3} x \left( 2\,L-3 \right) }{9\, \left( L-3 \right) 
^{2}}}
\qquad \text{and} \qquad 
  \Hc_d(y)
  = {\frac {64\,\sqrt {3} \left( P+1 \right)  \left( P-3 \right) ^{2}{P}^
{3}}{27\, \left(1- P \right) ^{5}}}.
\eeq
\end{cor}
Note that an explicit expression of $\Hc_d(y)$ was given in~\cite[Eq.~(53)]{trotignon-harmonic} in terms of radicals. The fact that the degree  is only $3$ is not obvious on this alternative expression.
\begin{proof}[Proof (sketched)] The principle is the same as in the proof of Corollary~\ref{cor:harmonic-K}, where we determined the harmonic function for Kreweras' walks in three quadrants. We focus on the asymptotic behaviour of the coefficients $a_{i,j}(n)$ of the series $A(x,y)$ defined by~\eqref{A-def}, and establish for them the estimate~\eqref{cij-est-simple}. Since the coefficients of $Q(x,y)$ grow like $4^n n^{-3}$ only, the coefficients $a_{i,j}(n)$ and $c_{i,j}(n)$ are asymptotically equivalent.

  We start from the expressions of $A_-(x)$ and $D(y)$ given in Theorem~\ref{thm:simple}, and use the results  of~\cite{mbm-gessel} to express $A_-(xt)$ and $D(y)$ rationally in terms of three algebraic series denoted $\sqrt T$, $U(x)$ and $V(y)$, as described in Remark~\ref{rem:degrees-simple}.

  We then perform  the singular analysis of $T$,  $U(x)$, $V(y)$ in the neighborhood of their dominant singularity, located at $t^2=1/16$.  The series $L(x)$ (resp.~$P(y)$) occurs in the singular expansion of $U(x)$ (resp.~$V(y)$). The three series  $T$,  $U(x)$, and $V(y)$ have a singular behaviour in $(1-16t^2)^{1/3}$, but we need to work out more terms because cancellations occur when moving from these series to
   $A_-(xt)$ and $D(y)$, which are found to have a singular behaviour in $(1-16t^2)^{2/3}$. This is how we finally obtain the expressions of  $\Hc_-(x)$ and $\Hc_d(y)$.
  The expression of $\Hc(x,y)$ simply comes from the harmonicity of $H_{i,j}$.
\end{proof}

  \begin{remark}
   In passing, we have also determined the harmonic function for Gessel's walks in the first quadrant: for $(i,j)\in \Qc$ and $n\equiv i$ mod $2$, the number of $n$-step walks ending at $(i,j)$ is asymptotic~to
  \[
    \frac{h_{i,j}}{\Gamma(-4/3)}  4^n n^{-7/3},
  \]
  where
  \[
    \sum_{i\ge0} h_{i,0} x^i=48 \sqrt {3} \ \frac {\left( 2\,L-3 \right) ^{2}}{ \left( L-3 \right) ^{4}}
\qquad\text{and} \qquad 
    \sum_{j\ge 0} h_{0,j} y^i=\frac{32}{\sqrt 3}{\frac {{P}^{3} \left(3- P \right) }{ \left( P+1 \right)  \left( P-1 \right) ^{4}}}.
  \]
  We observe that these two series are cubic over $\qs(\sqrt 3)$, as the two univariate series describing the harmonic function of simple walks in three quadrants~\eqref{Hminus-simple}. This is in contrast with the previously solved models, where the degree of $\Hc_-(x)$ and $\Hc_d(y)$ is twice the degree of the corresponding quadrant harmonic function; see for instance~\eqref{Hminus-K} and~\eqref{harmonic-RK-quadrant}. As already mentioned in Remark~\ref{rem:degrees-simple}, a similar property holds at the level of (counting) \gfs: $\tQ(x,0)$ and $\tQ(0,y)$ have degree $24$, and the same holds for $A_-(x)$ and $D(y)$.
    \end{remark}

\subsection{The diagonal model}
\label{sec:diag}
When $\cS=\{\nearrow, \nwarrow, \swarrow, \searrow\}$, the series $A(x,y)$ defined by~\eqref{A-def} satisfies
\beq\label{A-eq-diag}
(1-t(x+\bx)(y+\by))A(x,y)= (2+\bx^2+\by^2)/3-t\by (x+\bx)A_-(\bx) -t \bx(y+\by) A_-(\by)-t\bx\by A_{0,0} .
\eeq
As before, the series $A(x,y)$ counts weighted walks in the three-quadrant cone starting either from $(0,0)$, or from $(-2,0)$, or from $(0,-2)$, with a weight $2/3$ in the first case, $1/3$ in the other two. As discussed in Lemma~\ref{lem:func_eq}  for the corresponding series $C(x,y)$, it makes sense to define two series $U(x,y)$ and $D(y)$ by
\beq\label{A-UD-diag}
  A(x,y) = \bx^2 U(\bx^2,xy)+ D(xy)+ \by^2 U(\by^2,xy).
\eeq
Now we can reproduce the step-by-step arguments that led to~\eqref{eqT-C-diag}  for the series $C(x,y)$, and we thus obtain
\beq\label{eqT-diag}
  2 xy\tK(x,y)U(x,y)= \frac 2 3 y(1+x) +y\left(t(y+\by)+2xyt-1\right)D(y)-2 t(1+x)U(x, 0)-tD_0,
\eeq
where  $\tK(x,y)=1-t\tS(x,y)=
1-t(y+\by+xy+\bx\by)$ is  the kernel of Gessel's model, reflected in
the first diagonal. Compared to the original functional equation~\eqref{eqT-C-diag}, the only thing that has changed is the initial term $y$, which has become $2 y(1+x) /3$.

For the model $\tS$, we  know two pairs of  invariants (see Table~\ref{tab:inv-qu-1}): one is rational, and will not be used; the other takes the form~\eqref{I1J1-def} and involves quadrant \gfs:
\beq\label{I1J1-gessel-flipped}
I_1(x) =t(1+x)\tQ( x,0)-\frac x{t(1+x)},  \qquad
J_1(y)=  -t \tQ(0,y)+ t\tQ_{0, 0}-\by .
  \eeq
  These  two series are algebraic, and we refer to~\cite{BoKa-10,mbm-gessel} for explicit expressions. Denoting by $(I_1^g(x), J_1^g(y))$ the pair of invariants~\eqref{I1J1-gessel} that we used for Gessel walks, we observe that
  \[
    I_1(x)=-J_1^g(x)+ t \tQ_{0,0}, \qquad J_1(y)=- I_1^g(y)+ t \tQ_{0,0},
  \]
  in accordance with the fact that the two models differ by an $x/y$-symmetry.

\begin{theorem}\label{thm:diag}
  Let $\cS=\{\nearrow, \nwarrow, \swarrow, \searrow\}$, and let $\tS=
  \{\uparrow, \nearrow, \downarrow, \swarrow\}$ be the set of
  reflected Gessel's steps. Let $C(x,y)$ be the \gf\ of walks with steps in $\cS$, confined to the three-quadrant cone~$\Cc$. Let us define $A(x,y)$ by~\eqref{A-def}.
  
  Let $I_1(x)$ and $J_1(y)$ be the quadrant $\tS$-invariants defined by~\eqref{I1J1-gessel-flipped}.   The \gf\  $A(x,y)$, which counts (weighted) walks in $\Cc$ with steps in $\cS$, satisfies  the  equation~\eqref{A-eq-diag}, where  $A_-(x)$ can be expressed in terms of $I_1$ and $J_1$ by
\beq\label{Am-diag}
\frac 9 4  \left(2 t (\bx+1) A_-(\sqrt x) +tD_{0} -\frac 2{3t(1+x)} \right) ^{2}
  = \big(I_1(x)- J_1(Y_0)\big)\big(I_1(x)- J_1(Y_1)\big).
\eeq
In this expression, the series $Y_{0,1}$ are the two roots of
\beq\label{Delta-diag}
  \Delta(y)= \left(1-t^2\by (1+y)^2\right)\left(1-t^2\by (1-y)^2\right)
  \eeq
  that are formal power series in $t$.

The \gf\ $D(y)$ defined by~\eqref{A-UD-diag} satisfies
\beq\label{Dy-diag}
\frac 9 4 {\Delta(y) }  \left( yD ( y) +\frac {2}{3t} \right) ^{2}
=   \big(J_1(y)- J_1(Y_0)\big)\big(J_1(y)- J_1(Y_1)\big).
\eeq
In particular, the series $D_0=D(0)$ involved in~\eqref{Am-diag} satisfies
\[
  3 t^2 D_0= 2+ t(J_1(Y_0)+ J_1(Y_1)).
\]
The series $A_-(x)$, $D(y)$ and $A(x,y)$ are algebraic.
\end{theorem}

\begin{remark}
   {\bf Degrees and rational parametrizations of $A_-(\sqrt x)$ and $D(y)$.}
We observe for this model the same phenomenon as for the simple model: the series $\tQ(x,0)$ and $\tQ(0,y)$ have degree $24$, but $A_-(\sqrt x)$ and $D(y)$ have also degree $24$ (only). More precisely,  using again the results and notation of~\cite{mbm-gessel}, one can express  $\tQ(x,0)$  and $\tQ(0,yt)$ (involved in the expressions of  $I_1(x)$ and  $J_1(yt)$, respectively)  as rational functions  in the four algebraic series $T$, $Z=\sqrt T$, $U\equiv U(y)$ and $V\equiv V(x)$ (note the exchange of $x$ and $y$, which comes from the diagonal reflection of steps).
One then derives from these results rational expressions of $I_1(x)/t$ and $J_1(yt)/t$ in terms of $\sqrt T$, $U$ and $V$, and well as the following  expressions:
  \[
    J_1(Y_0)+J_1(Y_1)= {\frac {64\,t{Z}^{3} \left( {Z}^{3}+3\,{Z}^{2}-Z+1 \right) }{ \left(1- Z \right)  \left( {Z}^{2}+3 \right) ^{3}}}
,
\]
\[
   J_1(Y_0)J_1(Y_1)=  4\,{\frac {{Z}^{7}+7\,{Z}^{6}-3\,{Z}^{5}+19\,{Z}^{4}-45\,{Z}^{3}+53\,{
Z}^{2}-Z+1}{ \left( Z-1 \right)  \left( {Z}^{2}+3 \right) ^{3}}}  .
\]
Details of these calculations are given in our \Maple\
session. Using~\eqref{Am-diag} (resp.~\eqref{Dy-diag}), we then obtain
rational a expression of
\[
  2  (\bx+1) A_-(\sqrt x) +D_{0} -\frac 2{3t^2(1+x)}
  \qquad\left( \text{resp.} \quad
 D ( yt) +\frac {2}{3t^2y}  \right)
\]
of the form $\sqrt T \Rat(T,V(x))$ (resp. $\sqrt T \Rat(T,U(y))$). \qee
\end{remark}

\begin{proof}[Proof of Theorem~\ref{thm:diag}]
 We start from the functional equation~\eqref{eqT-diag}, and observe that  we have  a decoupling relation:
\[
y  (1+x) = (t(y+\by) +2txy-1)G(y)+F(x) + H(x,y) \tK(x,y),
\]
where
\[
  F(x)=\frac 1 {t(1+x)}, \qquad G(y)= \frac 1 t, \qquad H(x,y)=\frac x {t(1+x)}.
\]
This, together with Lemma~\ref{lem:square}, leads us to define a new  pair of invariants:
\[ 
   I(x)= \left(2 t(1+x) U(x,0) + tD_0-\frac 2 {3t(1+x)}\right)^2,
  \qquad
  J(y)= \Delta(y) \left(yD(y) + \frac{2}{3t}\right)^2,
\] 
where $\Delta(y)$ is given by~\eqref{Delta-diag}. We want to combine
them polynomially with the invariants $(I_1(x), J_1(y))$ given
by~\eqref{I1J1-gessel-flipped} to form a pair of invariants $(\tilde I(x), \tilde J(y))$ to which Lemma~\ref{lem:invariants}  applies.

Observe that $I_1(x)$ and $I(x)$ have poles at $x=-1$, while $J_1(y)$ and $J(y) $ have poles at  $y=0$. More precisely, the series $J(y)$  has a (double) pole at $0$, which can be removed by subtracting a well-chosen quadratic polynomial in $J_1(y)$. This leads us to introduce
\[
  \tilde J(y) = J(y) - \frac 4 9 J_1(y)^2  - B_1 J_1(y),
\]
and to define accordingly
\[
  \tilde I(x) = I(x) - \frac 4 9 I_1(x)^2 - B_1 I_1(x),
\]
for a series  $B_1$ that involves $D_0$  (we refer to our \Maple\ session for details).
Then, we use the functional equations satisfied by $U(x,y)$ and $\tQ(x,y)$ to check that the ratio
\[
  \frac{\tilde I(x) -\tilde J(y)}{\tK(x,y)}
\]
is a multiple of $xy$, and conclude that $ \tilde I(x)=  \tilde J(y)=B_0$, for some series $B_0\in \qs((t))$. Hence, denoting $P(u)= 4u^2/9+B_1u+B_0$, we have
\[ 
  I(x)= P(I_1(x)), \qquad   
  J(y)= P(J_1(y)).
\] 
We can express the roots of $P$ in terms of the series $\tQ$ (and more precisely, in terms of $J_1$) by looking at the two roots of $\Delta(y)$, denoted $Y_i$, with $i={0,1}$, that lie in  $\qs[[t]]$. For each of them $J_1(Y_i)$ is a root of $P(u)$. This gives the expressions of the theorem, which are those of $9I(x)/4$ and $9J(y)/4$.
\end{proof}

Let us finish with the harmonic function associated with diagonal walks in $\Cc$. 

\begin{cor}
  For $(i,j) \in \cC$, there exists a positive constant $H_{i,j}$ such that, as $n\rightarrow \infty$ with $n\equiv i+j \mod 2$,
 \[  
  c_{i,j}(n) \sim -\frac{H_{i,j}}{\Gamma(-2/3)}\, 4^n n^{-5/3}.
\] 
Clearly, $H_{i,j}=0$ if $i\not \equiv j$ mod $2$. The \gf\
\[
  \Hc(x,y):=\sum_{j\ge 0, i\le j} H_{i,j} x^{(j-i)/2} y^j,
\]
which is a \fps\ in $x$ and $y$, is algebraic of degree $9$, given by
\[ 
\left(1+x+xy^2+x^2y-4xy\right)  \Hc(x,y)= \frac x 2 H_{0,0}
+(1+x)\Hc_-(x)+\left(1-2xy+\frac{x}2 (1+y^2)\right) \Hc_d(y),
\] 
where
\[
  \Hc_-(x):=\sum_{i>0} H_{-i,0} x^{i/2}\qquad \text{and} \qquad
   \Hc_d(y):=\sum_{i\ge 0} H_{i,i}y^i.
\]
Each of these series is algebraic of degree $3$. Let $ L(x)$ and $P(y)$ be defined  as in Corollary~\ref{cor:harmonic-simple}.
Then
\[ 
  \Hc_-(x) ={\frac {32 \sqrt {3} \left( 3\,P(x)-1 \right) }{ 9\left( P(x)+1 \right) 
 \left( P(x)-1 \right) ^{2}}}
  \qquad \text{and} \qquad 
  \Hc_d(y) = {\frac {144\sqrt {3}\, L(y) \left( {L(y)}^{2}-3 \right) ^{2}}{y^2 \left( {L(y)}^{2
}+6\,L(y)-9 \right)  \left( 3-L(y)\right) ^{5}}}
.
\] 
(The exchange of $x$ and $y$ is intentional.)
\end{cor}
\begin{proof}
  As in the proof of Corollary~\ref{cor:harmonic-simple},   one starts from the rational expressions of $A_-(\sqrt x)$ and $D(y)$ in terms of $\sqrt T$, $U(y)$ and $V(x)$ derived from  Theorem~\ref{thm:diag} and~\cite{mbm-gessel}, and then plugs in them the singular expansions of $T$, $U(y)$ and $V(x)$ around $t^2=1/16$. This gives $\Hc_-(x)$ and $\Hc_d(y)$.
  The expression of $\Hc(x,y)$ simply comes from the harmonicity of $H_{i,j}$.
\end{proof}

\section{Discussion and perspectives}
\label{sec:final}

\subsection{Generic form of the results}
\label{sec:rat-uniform}
For the last three models that we have solved ($6$th model of Table~\ref{tab:sym} in Section~\ref{sec:DA}, simple and diagonal models in Section~\ref{sec:DF}), we have directly written the invariant $I(x)$, involving $C_-(x)$ or $A_-(x)$, as a rational function in the quadrant invariant $I_1(x)$ involving $\tQ(x,0)$, and analogously for $J(y)$ and $J_1(y)$; see Theorems~\ref{thm:DA}, \ref{thm:simple}, and~\ref{thm:diag}. The coefficients of these rational functions are series in $t$. This is also possible for the three models of the Kreweras trilogy, although we have (sometimes) also used the pair $(I_0(x), J_0(y))$ to determine $(I(x),J(y))$. For the Kreweras model and the reverse Kreweras model, $I(x)$ is a polynomial in $I_1(x)$; see~\eqref{eqinv-K} and~\eqref{IJ-P-RK}. For the double Kreweras model, one can ignore the pair $(I_0(x), J_0(y))$ as well, and prove that
\[
  I(x)= \frac{P(I_1(x))}{(I_1(x)-I_1(0))^2} \quad \text{and} \quad 
  J(y)= \frac{P(J_1(y))}{(J_1(y)-I_1(0))^2}
\]
for some polynomial $P(u)$ of degree $4$.

\subsection{Solving more models via invariants?}
Let us finally discuss if, and how, one could go further in the
solution of three-quadrant walks using the approach of this paper. Let
us first recall that this approach, when it works, relates invariants
involving $C(x,y)$ to pre-existing invariants (or to \emm weak,
invariants, in the sense of~\cite{BeBMRa-17}), which so far are
systematically D-algebraic. Hence there is little hope to solve with this
approach models that would not be (at least) D-algebraic. This means
that for the nine models of Table~\ref{tab:sym},
invariants have done for us all that we could hope for: they have
solved the six models that were not already known to be \emm
hypertranscendental, (that is, non-D-algebraic).

If we want to go further, two different difficulties arise: the model may contain NW and/or SE steps, and it may not be diagonally symmetric. Let us consider two examples.

 \begin{table}[b!]
    \centering
    \begin{tabular}{|c||a:cc||cccccc|}
      \hline
\rule{0pt}{4ex}   &   diagonal & king &  \begin{minipage}{15mm}
                   double\\ tandem
                  \end{minipage} & &&&&&\\ 
 \rule{0pt}{7ex} $\cS$&  \begin{tikzpicture}[scale=.45] 
    \draw[->] (0,0) -- (1,-1);
    \draw[->] (0,0) -- (1,1);
    \draw[->] (0,0) -- (-1,1);
     \draw[->] (0,0) -- (-1,-1);
  \end{tikzpicture}     &
 \begin{tikzpicture}[scale=.45] 
    \draw[->] (0,0) -- (1,-1);
    \draw[->] (0,0) -- (1,1);
    \draw[->] (0,0) -- (-1,1);
    \draw[->] (0,0) -- (-1,-1);
    \draw[->] (0,0) -- (0,1);
    \draw[->] (0,0) -- (-1,0);
    \draw[->] (0,0) -- (0,-1);
    \draw[->] (0,0) -- (1,0);
  \end{tikzpicture}&
    \begin{tikzpicture}[scale=.45] 
      \draw[->] (0,0) -- (0,-1);
    \draw[->] (0,0) -- (1,-1);
    \draw[->] (0,0) -- (-1,0);
    \draw[->] (0,0) -- (0,1);
    \draw[->] (0,0) -- (-1,1);
    \draw[->] (0,0) -- (1,0);
  \end{tikzpicture}
&
    \begin{tikzpicture}[scale=.45] 
      \draw[->] (0,0) -- (0,-1);
      \draw[->] (0,0) -- (-1,0);
    \draw[->] (0,0) -- (1,1);
    \draw[->] (0,0) -- (1,-1);
          \draw[->] (0,0) -- (-1,1);
  \end{tikzpicture}
&
    \begin{tikzpicture}[scale=.45] 
      \draw[->] (0,0) -- (1,-1);
    \draw[->] (0,0) -- (-1,1);
    \draw[->] (0,0) -- (-1,-1);
    \draw[->] (0,0) -- (0,1);
       \draw[->] (0,0) -- (1,0);
     \end{tikzpicture}
&
    \begin{tikzpicture}[scale=.45] 
         \draw[->] (0,0) -- (1,-1);
        \draw[->] (0,0) -- (-1,1);
       \draw[->] (0,0) -- (1,1);
       \draw[->] (0,0) -- (0,-1);
    \draw[->] (0,0) -- (-1,-1);
    \draw[->] (0,0) -- (-1,0);
  \end{tikzpicture}
      &   \begin{tikzpicture}[scale=.45] 
        \draw[->] (0,0) -- (1,-1);
        \draw[->] (0,0) -- (-1,1);
        \draw[->] (0,0) -- (0,1);
    \draw[->] (0,0) -- (1,1);
    \draw[->] (0,0) -- (1,0);
      \draw[->] (0,0) -- (-1,-1);
  \end{tikzpicture}
&
    \begin{tikzpicture}[scale=.45] 
       \draw[->] (0,0) -- (1,-1);
        \draw[->] (0,0) -- (-1,1);
        \draw[->] (0,0) -- (0,-1);
       \draw[->] (0,0) -- (-1,0);
    \draw[->] (0,0) -- (0,1);
    \draw[->] (0,0) -- (1,1);
    \draw[->] (0,0) -- (1,0);
  \end{tikzpicture}
&
    \begin{tikzpicture}[scale=.45] 
       \draw[->] (0,0) -- (1,-1);
        \draw[->] (0,0) -- (-1,1);
        \draw[->] (0,0) -- (0,-1);
       \draw[->] (0,0) -- (-1,0);
    \draw[->] (0,0) -- (0,1);
    \draw[->] (0,0) -- (-1,-1);
    \draw[->] (0,0) -- (1,0);
  \end{tikzpicture}
\\
 \hline
      $C_\cS(x,y;t)$ & DF & DF &\emph{ DF }? &
        \multicolumn{6}{c|}{Not    DF }
                                    \\
      &Sec.~\ref{sec:diag} & & &\multicolumn{6}{c|}{D-algebraicity not known}\\
            & \cite{BM-three-quadrants} &\cite{mbm-wallner} &\cite{mbm-wallner} & &&&&&\\
       \hline
    \end{tabular}
    \vskip 4mm  \caption{The nine models  with $x/y$-symmetry and
      steps $\nwarrow$ and $\searrow$. The first three have a finite
      group and are \emm Weyl models, in the sense
      of~\cite{mbm-wallner}; the others have an infinite group.}
      \label{tab:NWSE}
  \end{table}

\medskip
\paragraph{\bf Diagonally symmetric models with NW and SE steps.} Apart from the diagonal model, which we have actually solved via invariants, there are eight models in this class, shown in Table~\ref{tab:NWSE}. The first two are associated with a finite group (orders $4$ and $6$) and the other six with an infinite group. Let us consider for instance the fourth model of the table, $\cS=\{\nearrow, \nwarrow, \leftarrow, \downarrow, \searrow\}$ (the \emm scarecrow,), and write as before
\[
  C(x,y)=\bx U(\bx, xy)+ D(xy)+ \by U(\by,xy),
\]
with
$D(y) \in \qs[y][[t]]$ and $U(x,y) \in \qs [x,y][[t]]$.

We can write functional equations, first for $D(xy)$ and $\bx U(\bx,xy)$, as we did in Section~\ref{sec:eqfunc}. A step-by-step construction of walks gives
\[
  D(xy) =  1+ t xy D(xy) 
  +2t \by \left(\bx U(0,xy)-\bx U_{0,0}\right)
 + 2t x \by \left(\bx^2 U_1(xy)-\bx^2 U_{1,0}\right),
\]
 where $U_1(y)$ is the coefficient of $x$ in $U(x,y)$. The term involving $U_1$ is new, and corresponds to walks that jump from the lines $j-i=\pm 2$ to the diagonal, using a NW of SE step.
 For walks ending above the diagonal, we find
 \begin{multline*}
  \bx \left(1- tS(x,y)\right) U(\bx,xy)=    t \bx (1+y) D(xy)+ t\bx y \left(\by U(0,xy)-\by U_{0,0}\right)\\
  -t \by (1+x) \left(\bx U(\bx,0)+ \bx U(0,xy) -\bx U_{0,0}\right) -tx\by \left(\bx^2 U_1(xy)-\bx^2 U_{1,0}\right),
\end{multline*}
with $S(x,y)=xy+\bx y +\bx +\by +x\by$. The second term on the first line accounts for walks jumping from the line $j=i-1$ to the line $j=i+1$ by a NW step (we forbid steps from $(0,-1)$ to $(-1,0)$). The last term of the second line accounts for walks that would jump from the line $j=i+2$ to the diagonal.

We can eliminate the series $U_1(xy)$ by taking a linear combination of the two equations; but we still have $U(0,xy)$ and $D(xy)$, while we only had one of them when the steps NW and SE were forbidden. Upon performing the change of variables $x\mapsto \bx, y\mapsto xy$, we finally obtain the following counterpart of~\eqref{eq-U}:
\begin{multline*}
  2xy \tK(x,y) U(x,y)=   y +  y\, \big(t  y+ 2tx (1+xy) -1\big) D(y)
  -2t (1+\bx)  U(x,0) + 2t(xy-\bx) \left( U(0,y)-U_{0,0}\right),
\end{multline*}
with $\tK(x,y):=1-tS(\bx,xy)= 1-t (x+ \bx \by +y +\bx^2 \by +x^2 y)$. The main two novelties are the facts that the kernel has degree $2$ and valuation $-2$ in $x$, and that the right-hand side involves two series in $y$, namely $D(y)$ and $U(0,y)$. Hence it seems that a first step towards solving this three-quadrant model via invariants would be to learn how to solve quadrant models with large steps via invariants -- a question already raised in~\cite{BBMM-18}.

\medskip
\paragraph{\bf Asymmetric models.} There are $74-8-9=57$ models that do not have the $x/y$-symmetry. It is remarkable that none of them has been solved.  Exactly $16$ have a finite group, among which $4$ (at least) are conjectured to have a D-finite (even algebraic, in one case) \gf~\cite{mbm-wallner}. They are shown in Table~\ref{tab:DF-conj}.

 \begin{table}[ht]
    \centering
    \begin{tabular}{|c||ccc:c|}
      \hline
\rule{0pt}{4ex}  &   diabolo & tandem&  \begin{minipage}{22mm}
                   Gouyou-\\ Beauchamps
                  \end{minipage}  & Gessel \\ 
 $\cS$&  \begin{tikzpicture}[scale=.45] 
    \draw[->] (0,0) -- (1,-1);
    \draw[->] (0,0) -- (1,1);
    \draw[->] (0,0) -- (-1,1);
    \draw[->] (0,0) -- (-1,-1);
    \draw[->] (0,0) -- (0,1);
    \draw[->] (0,0) -- (0,-1);
  \end{tikzpicture}     &
 \begin{tikzpicture}[scale=.45] 
    \draw[->] (0,0) -- (1,-1);
        \draw[->] (0,0) -- (-1,0);
    \draw[->] (0,0) -- (0,1);
  \end{tikzpicture}&
    \begin{tikzpicture}[scale=.45] 
        \draw[->] (0,0) -- (1,-1);
    \draw[->] (0,0) -- (-1,0);
       \draw[->] (0,0) -- (-1,1);
    \draw[->] (0,0) -- (1,0);
  \end{tikzpicture}
&
    \begin{tikzpicture}[scale=.45] 
      \draw[->] (0,0) -- (-1,0);
      \draw[->] (0,0) -- (1,0);
    \draw[->] (0,0) -- (1,1);
    \draw[->] (0,0) -- (-1,-1);
  \end{tikzpicture}
\\
      \hline
      $C_\cS(x,y;t)$ & \emph{ DF } ? & \emph{ DF } ? &\emph{ DF }? & \emph{ algebraic } ? \\
       \hline
    \end{tabular}
    \vskip 4mm  \caption{Four  interesting asymmetric models with a
      finite group. The first three are \emm Weyl models, in the
      sense of~\cite{mbm-wallner}.}
     \label{tab:DF-conj}
  \end{table}

 To illustrate the new difficulties that arise,  let us work out functional  equations for Gessel's model. Since symmetry is lost, we now write
  \[
  C(x,y)=\bx U(\bx, xy)+ D(xy)+ \by L(\by,xy),
\]
with $D(y) \in \qs[y][[t]]$ and $U(x,y)$, $L(x,y) \in \qs [x,y][[t]]$. The equation for walks ending on the diagonal reads
\[
  \left(1-t (xy+\bx\by) \right)D(xy)=1- t\bx\by D_0+ tx\left( \bx U(0, xy) \right)+ t \bx\left( \by L(0,xy)-\by  L_{0,0}\right).
\]
Now for walks ending above the diagonal, we obtain
\[
  \left(  1-t S(x,y)\right) \bx U(\bx,xy)= t\bx D(xy) - t\bx \by \left(\bx U(\bx,0)\right) -tx \left(\bx U(0,xy)\right),
\]
with $S(x,y)=x+\bx +xy+\bx\by$. Finally, for walks ending below the diagonal, we find:
\[
  \left(  1-t S(x,y)\right) \by L(\by,xy) = tx D(xy)-t \bx(1+\by) \by L(\by,0) -t\bx \left(\by L(0,xy)- \by L_{0,0}\right).
\]
In the first two equations, we perform the same change of variables as
before, namely $x\mapsto \bx, y\mapsto xy$. In the third equation, we
apply  (in parallel) $ x\mapsto xy, y\mapsto \bx $. This gives:
\begin{align*}
  \left(  1-t S(\bx,xy)\right) x U(x,y)&= \hskip 6mm tx D(y) - tx \by  U(x,0) -t U(0,y),
  \\
  \left(1-t (y+\by) \right)D(y)&=1- t\by D_0+ t U(0, y) + t \by  \left(L(0,y)-   L_{0,0}\right),
\\
  \left(  1-t S(xy,\bx)\right) x L(x,y) &=  \hskip 6mm txy D(y)-t \by(1+x)  L(x,0)
  -t \by \left( L(0,y)- L_{0,0}\right).
\end{align*}
We now have a system with \emm two,  trivariate series $U(x,y)$ and $L(x,y)$. In addition, the kernels
\[
  1-t S(\bx,xy)= 1-t(x+\bx+y+\by)\quad \text{ and } \quad 1-tS(xy,\bx)=1-t(y+\by +xy+\bx\by)
\]
are not the same.

   \bigskip\bigskip



 \bibliographystyle{abbrv}
 \bibliography{qdp}

\end{document}